\documentclass{article}
\usepackage[a4paper, margin=1in]{geometry} 
\usepackage{graphicx} 
\usepackage{forest,tikz,amsfonts,amsmath,amssymb,amsthm}

\usepackage[a4paper, margin=1in]{geometry}

\usepackage{subcaption} 

\usetikzlibrary{shapes.geometric, arrows}
\usetikzlibrary{arrows.meta, positioning}

\usepackage{ytableau}
\usepackage{multicol}

\usepackage{algorithm}
\usepackage{algpseudocode}
\usepackage{float}

\usepackage{dsfont}

\usepackage{enumitem}

\usepackage{enumitem}

\usepackage{dsfont}

\usepackage{hyperref}
\hypersetup{
    colorlinks=true,
    linkcolor=blue,
    filecolor=magenta,      
    urlcolor=cyan,
    pdftitle={Overleaf Example},
    pdfpagemode=FullScreen,
    }

\usepackage{cleveref}

\def\V#1{{\mathbf #1}}
\def\floor#1{{\lfloor #1 \rfloor}}

\def\Bern{\operatorname{Bern}}

\def\Unif{\operatorname{Unif}}

\def\Var{\operatorname{Var}}

\def\P{\mathbb{P}}
\def\E{\mathbb{E}}
\def\B{\operatorname{B}}

\def\Bin{\operatorname{Bin}}

\def\supp{\operatorname{supp}}

\def\supp{\operatorname{supp}}

\def\hs{\operatorname{HS}}

\newtheorem{proposition}{Proposition}
\newtheorem{theorem}{Theorem}
\newtheorem{corollary}{Corollary}
\newtheorem{conjecture}{Conjecture}
\newtheorem{remark}{Remark}

\newtheorem{example}{Example}

\allowdisplaybreaks

\title{The Horton-Strahler number of  butterfly trees}
\author{John Peca-Medlin\thanks{Department of Mathematics, University of California, San Diego, \href{mailto:jpecamedlin@ucsd.edu}{jpecamedlin@ucsd.edu}}}
\date{}

\begin{document}

\maketitle

\begin{abstract}
    The Horton--Strahler (HS) number, a classical measure of branching complexity arising in hydrology and register allocation, is studied for butterfly trees, a recursive family of binary trees generated by block-merging operations. These trees arise as binary search trees of butterfly permutations, which form the $2$-Sylow subgroup of the symmetric group on $N = 2^n$ elements and appear in models of parallel computation and structured Gaussian elimination. For a single merging step applied to two independent Catalan trees with $m$ nodes, we show that $\hs(\mathcal T_1 \oplus \mathcal T_2)/\log_2(2m) \to 1/2$ in probability, so the classical Catalan scaling is preserved under this restricted construction. In the simple butterfly model, where each level is formed from identical copies and encoded by an $n$-bit string $\mathbf{x}$, the HS number admits an exact representation as an additive functional of an explicit $8$-state Markov chain driven by iid bits $x_j \sim \mathrm{Bern}(p)$, and can be computed in $\mathcal O(n)$ time from $\mathbf{x}$. This yields a complete limit theory, including a strong law $\hs(\mathcal T_n^{\B})/n \to \mu_p = pq/(1-pq)$ almost surely and a functional central limit theorem with variance $\sigma_p^2 = pq(1 - 3pq - 2p^2q^2)/(1-pq)^3$. For general butterfly trees, obtained by recursively merging independent subtrees, the increment depends on an expanding edge profile, and the process does not admit a finite-state reduction. We give an $\mathcal O(N)$ algorithm to compute the HS number directly from the $(N-1)$-bit encoding, characterize the zero-HS class, and combine exact enumeration for small $n$ with Monte Carlo simulations up to $n=25$, supporting $\hs(\mathcal T_n^{\B})/n \to \alpha \approx 0.4450$ in probability for uniform butterfly trees, placing the general model strictly between the simple butterfly limit $1/3$ and the Catalan limit $1/2$.
\end{abstract}

\section{Introduction}\label{sec: intro}

The {Horton--Strahler (HS) number} was originally introduced in hydrology to quantify the branching complexity of river networks, used to identify the order of a river ~\cite{horton1945,strahler1952}.  
Within computer science, it is known as the \emph{register function} and corresponds to the minimal number of auxiliary registers needed to evaluate an arithmetic expression (see, e.g.,~\cite{flajolet1979number}), and so is relevant in measuring parallel architecture efficiency. See \cite{esparza2014history, viennot1990trees} for surveys that cover other applications of the HS number, including to such disparate disciplines as linguistics, biological systems, Newton iterations, parity games, and social networks.

Let $\mathcal T$ be a rooted binary tree. The HS number of a node $v$ is defined recursively as
\[
    \hs(v) =
    \begin{cases}
        0, & \text{if $v$ is a leaf,}\\[4pt]
        \max\left(\hs(v_\ell), \hs(v_r)\right) + \mathds{1}_{\hs(v_\ell) = \hs(v_r)}, & \text{otherwise,}
    \end{cases}
\]
where $v_\ell$ and $v_r$ denote the left and right children of $v$ (if they exist).  
The HS number of  $\mathcal{T}$, written $\hs(\mathcal T)$, is the value at its root. Equivalently, $\hs(\mathcal T)$ is the largest integer $k$ such that $\mathcal T$ contains an embedded perfect binary tree of height $k$ (that is, a full binary tree with $2^k-1$ internal nodes). 

The HS number can be computed in $\mathcal O(m)$ time for a tree with $m$ nodes via a single depth-first traversal, since each node’s value is determined exactly once from its children. (See \cite{ganardi2025complexity} for very recent work on complexity of computing HS number on binary trees using parallel architectures.) Extensions to broader graph classes and refined variants have also been studied~\cite{brandenberger2021horton,addario-berry2024refined}. \Cref{fig: BST HS number keys} displays a binary search tree (BST) with both labels from its insert keys and its HS numbers.


\begin{figure}
        \centering
        \begin{subfigure}{0.4\textwidth}
    \includegraphics[width=\linewidth]{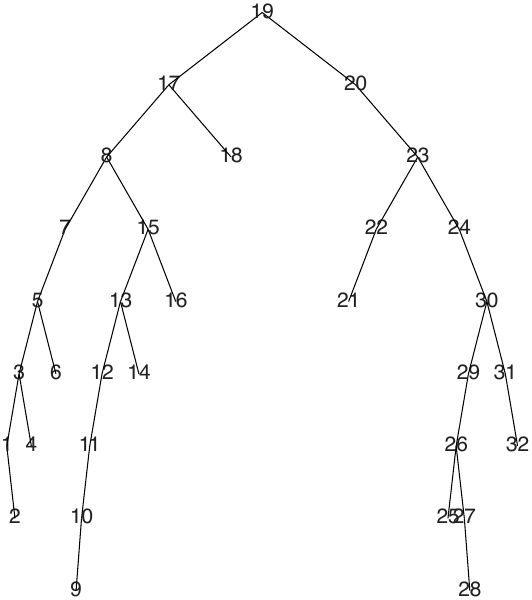}
    \end{subfigure}\qquad
    \begin{subfigure}{0.4\textwidth}
    \includegraphics[width=\linewidth]{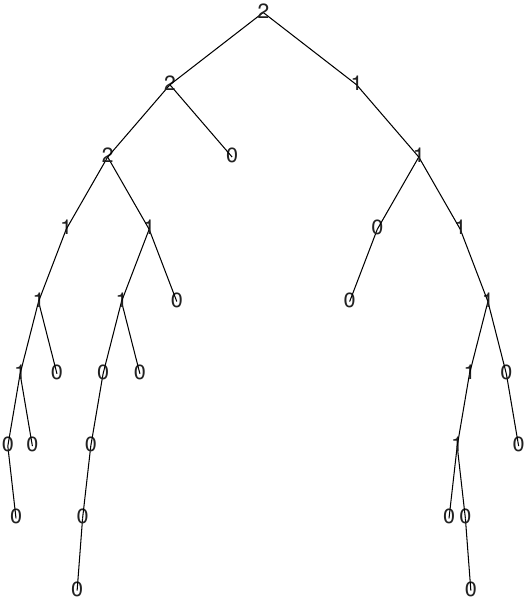}
    \end{subfigure}
            
            
        \caption{The BST $\mathcal T(\pi)$ presented with insert key labels and HS number labels, using the  butterfly permutation $\pi = [19,\!20,\!17,\!18,\!23,\!24,\!22,\!21,\!30,\!29,\!31,\!32,\!26,\!25,\!27,\!28,\!8,\!7,\!5,\!6,\!3,\!4,\!1,\!2,\!15,\!16,\!13,\!14,\!12,\!11,\!10,\!9] \in \B_5.$}
        \label{fig: BST HS number keys}
    \end{figure}

Results on the HS number of binary trees date back to the 1950s, beginning with the work of Ershov~\cite{ershov58} on register allocation. The study of HS number statistics for random trees emerged in the 1970s through the analytic combinatorics framework of Flajolet~\cite{flajolet1977}. A central model in this line of work is that of \emph{Catalan trees}, uniformly sampled rooted binary trees with $m$ internal nodes, $\mathcal T_m$. This moniker follows as these trees are enumerated by the Catalan numbers,
$$C_m = \frac1{m+1}\binom{2m}m.$$ 
and arise in a wide range of combinatorial settings (e.g., see~\cite{Stanley2015CatalanNumbers} for 214 such interpretations).

In contrast, results for random BSTs are more limited, as the induced distribution on tree shapes is non-uniform: each tree is weighted proportionally to the number of permutations that generate it. For example, the permutations $213$ and $231$ produce the same BST, $\mathcal T(213) = \mathcal T(231)$, while the permutation $123$ yields a different shape, and each underlying binary tree shape would be equiprobable under the Catalan  modeling.

For Catalan trees $\mathcal T_m$, classical results~\cite{flajolet1977,kemp1978,meir1980} show that $\hs(\mathcal T_m)\in\{0,\dots,\lfloor \log_2(m+1)\rfloor\}$ and
\begin{equation}\label{eq:catalan_mean_var}
\mathbb{E}[\hs(\mathcal T_m)]
= \log_4 m + D(\log_4 m) +\mathcal O(1),
\qquad
\operatorname{Var}(\hs(\mathcal T_m)) =\mathcal O(1),
\end{equation}
where $D$ is a continuous $1$-periodic function. As a consequence, a \textit{Weak Law of Large Numbers (WLLN)} follows from Chebyshev’s inequality:
\begin{equation}\label{eq:catalan_wlln}
\frac{\hs(\mathcal T_m)}{\log_2 m}
\xrightarrow[m\to\infty]{\mathbb{P}} \frac{1}{2}.
\end{equation}
Sharper concentration bounds were later obtained in~\cite{DK95,K99}. These results have since been extended to other random tree models, including conditioned Galton–Watson trees and stable variants~\cite{brandenberger2021horton,khanfir2023horton}.

Despite this progress, second-order fluctuations of the HS number in random tree models remain poorly understood. In classical settings, such as Catalan or Galton--Watson trees, periodic oscillations in the underlying profile typically obstruct a central limit theorem (CLT) at the level of the standard HS statistic. A recent development by Khanfir~\cite{khanfir2024fluctuations} introduces a real-valued variant of the HS number that converges to a functional of the stable Lévy tree arising in the Galton--Watson scaling limit, highlighting the strong dependence of fluctuation behavior on the underlying tree geometry.

This sensitivity motivates the study of alternative recursive constructions in which the HS number admits more tractable second-order structure. In this work, we consider a class of recursively generated binary trees—\emph{simple butterfly trees}—for which the HS evolution reduces to an additive functional of a finite-state Markov chain. This structure yields both a law of large numbers and a functional central limit theorem with explicit variance, providing a rare setting in which Gaussian fluctuations can be established for the HS number.

Our construction is based on block-merging operations.
Given rooted binary trees $\mathcal T_1$ and $\mathcal T_2$, define $\mathcal T_1 \oplus \mathcal T_2$ (resp. $\mathcal T_1 \ominus \mathcal T_2$) by attaching $\mathcal T_2$ along the rightmost (resp. leftmost) branch of $\mathcal T_1$. See \Cref{fig: BST gluing example} for examples of $\mathcal T_1 \oplus \mathcal T_2$ and $\mathcal T_1 \ominus \mathcal T_2$. These operations correspond to direct and skew sums of generating permutations for corresponding BST and arise naturally in the study of separable permutations and divide-and-conquer algorithms~\cite{PZ24,bassino2018brownian,maazoun2020brownian,Blelloch2016parallel,HHNA22}.

We focus on \emph{butterfly trees} $\mathcal T_n^{\B}$, which are constructed using only merging operators $\oplus$ or $\ominus$. These are formed recursively with two input butterfly trees of the previous level and a gluing bit, defined by
\begin{equation}\label{eq: butterfly tree def}
\mathcal T_{n+1}^{\B} =
\begin{cases}
\mathcal T_n^{\B} \oplus \widetilde{\mathcal T}_n^{\B}, & x=0,\\
\mathcal T_n^{\B} \ominus \widetilde{\mathcal T}_n^{\B}, & x=1,
\end{cases}
\end{equation}
starting from a single node. In the \emph{simple} case, the two input trees are identical at each step.

Butterfly trees were introduced in \cite{PZ25}, and are so named as they are directly connected to \textit{butterfly matrices}, a recursive family of orthogonal transformations used in numerical linear algebra (see \cite{Pa95,Tr19, PT23}). Butterfly trees are exactly formed as BSTs of the induced permutations from applying Gaussian elimination with partial pivoting to butterfly matrices \cite{P24,PZ24}. Such permutations are named \textit{butterfly permutations}, again from this common provenance. \Cref{fig: BST HS number keys} shows an example of a butterfly tree with $32$ nodes.

\subsection{Main results}

\begin{figure}
    \centering
    \begin{minipage}[c]{0.2\textwidth} 
        {
        \begin{subfigure}[b]{\textwidth}
            \centering
            \begin{tikzpicture}
            \node[circle, draw] (n4) at (2,1) {0};
            \node[circle, draw] (n3) at (1,2) {1};
            \node[circle, draw] (n1) at (0,1) {0};
            \node[circle, draw] (n2) at (1,0) {0};
            \draw[thick] (n3) -- (n4);
            \draw[thick] (n3) -- (n1);
            \draw[thick] (n2) -- (n1);
            \end{tikzpicture}
            \caption{$\mathcal T_1$}
        \end{subfigure}
        \vspace{0.5cm}
        \begin{subfigure}[b]{\textwidth}
            \centering
            \begin{tikzpicture}
            \node[circle, draw] (n3) at (2,1) {0};
            \node[circle, draw] (n2) at (1,2) {1};
            \node[circle, draw] (n1) at (0,1) {0};
            \draw[thick] (n3) -- (n2);
            \draw[thick] (n2) -- (n1);
            \end{tikzpicture}
            \caption{$\mathcal T_2$}
        \end{subfigure}}
    \end{minipage}
    \hfill
    \begin{minipage}[c]{0.3\textwidth}
        \begin{subfigure}[c]{\textwidth}
            \centering
            \begin{tikzpicture}
            \node[line width=1mm, circle, draw=blue, fill=blue!20] (n4) at (2,1) {$\V1$};
            \node[line width=1mm, circle, draw=blue] (n3) at (1,2) {1};
            \node[circle, draw] (n1) at (0,1) {0};
            \node[circle, draw] (n2) at (1,0) {0};
            \draw[line width=1mm, blue] (n3) -- (n4);
            \draw[thick] (n3) -- (n1);
            \draw[thick] (n2) -- (n1);

            \node[circle, draw] (n7) at (4,-1) {0};
            \node[circle, draw] (n6) at (3,0) {1};
            \node[circle, draw] (n5) at (2,-1) {0};
            \draw[thick] (n7) -- (n6);
            \draw[thick] (n6) -- (n5);

            \draw[line width=1mm, dotted, blue] (n4) -- (n6);
            \end{tikzpicture}
            \caption{$\mathcal T_1 \oplus \mathcal T_2$}
        \end{subfigure}
    \end{minipage}
    \hfill
    \begin{minipage}[c]{0.4\textwidth}
        \begin{subfigure}[d]{\textwidth}
            \centering
            \begin{tikzpicture}
            \node[circle, draw] (n7) at (2,1) {0};
            \node[line width=1mm, circle, draw=blue] (n6) at (1,2) {1};
            \node[line width=1mm, circle, draw=blue, fill=blue!20] (n4) at (0,1) {$\V1$};
            \node[circle, draw] (n5) at (1,0) {0};
            \draw[thick] (n6) -- (n7);
            \draw[line width=1mm, blue] (n6) -- (n4);
            \draw[thick] (n5) -- (n4);

            \node[circle, draw] (n3) at (0,-1) {0};
            \node[circle, draw] (n2) at (-1,0) {1};
            \node[circle, draw] (n1) at (-2,-1) {0};
            \draw[thick] (n3) -- (n2);
            \draw[thick] (n2) -- (n1);

            \draw[line width=1mm, blue, dotted] (n4) -- (n2);
            \end{tikzpicture}
            \caption{$\mathcal T_1 \ominus \mathcal T_2$}
        \end{subfigure}
    \end{minipage}

    \caption{$\mathcal T_1 \oplus \mathcal T_2$ and $\mathcal T_1 \ominus \mathcal T_2$ formed using $\mathcal T_1 = \mathcal T(3142)$ and $\mathcal T_2 = \mathcal T(213)$ with dotted gluing edge. Updated HS number labels on the merged tree highlight nodes that can have updated HS number labels, with an additional light blue shade indicating a node that has HS number label exceed its associated prior HS number label.}
    \label{fig: BST gluing example}
\end{figure}

We begin with a single gluing step. Let $\mathcal T_1$ and $\mathcal T_2$ be binary trees with $m$ nodes, and consider the merged tree obtained via one application of $\oplus$ (or $\ominus$). Our first result shows that, at leading order, the HS number of these merged block trees matches that of a binary tree with $2m$ nodes.

\begin{theorem}\label{thm:block}
Let $\mathcal T_1,\mathcal T_2$ be independent and identically distributed (iid) Catalan trees with $m$ nodes. Then
\[
\frac{\hs(\mathcal T_1 \oplus \mathcal T_2)}{\log_2(2m)}
\xrightarrow[m \to \infty]{\P} \frac12, \qquad \supp(\hs(\mathcal T_1 \oplus \mathcal T_2)) = \{0,\ldots,\floor{\log_2(2m + 1)}\}.
\]
The same holds for $\mathcal T_1 \ominus \mathcal T_2$.
\end{theorem}
The proof of \Cref{thm:block} is based on explicit merging rules for the HS number. For binary trees $\mathcal T_1$ and $\mathcal T_2$,
\[
\hs(\mathcal T_1 \oplus \mathcal T_2)
= \max\big(\hs(\mathcal T_1),\hs(\mathcal T_2)\big) + I,
\]
where $I \in \{0,1\}$ is an increment determined by the local structure at the gluing interface (see \Cref{prop: merge increment}).

A key simplification is that the HS update depends only on the \emph{edge profile} of each tree, rather than the full collection of node labels. More precisely, it suffices to track the HS values along the top gluing edge, since updates can occur only along this path; see \Cref{fig: BST gluing example,fig:gluing-perfect-binary}. The resulting profile dynamics are given explicitly in \Cref{prop: profile update}, and allow one to compute the HS number of a merged tree using only the edge profiles of its inputs.

This reduction leads to an efficient algorithmic representation. A butterfly tree with $N=2^n$ nodes can be encoded by a bitstring $\mathbf{x}\in\{0,1\}^{N-1}$, where each bit specifies a $\oplus/\ominus$ merge in the recursive construction. The HS number can then be computed in $\mathcal O(N)$ operations directly from $\mathbf{x}$ using the profile update rules (see \Cref{alg:hs_profile_corrected}). Equivalently, the tree may be viewed as a functional of the first $n$ levels of an infinite binary tree with $\oplus/\ominus$ labels at each node, yielding $N-1$ labels in total.

\medskip

We now consider butterfly trees $\mathcal T_n^{\B}$ with $N=2^n$ nodes. For all such trees,
\[
\supp(\hs(\mathcal T_n^{\B})) = \{0,\ldots,n/2\}
\quad \text{(see \Cref{prop: butterfly support})}.
\]
Thus, while a single merge preserves the classical support, iterated merging leads to a strictly smaller range. In particular, $n/2 = \tfrac12 \log_2 N$ becomes a deterministic upper bound, in contrast to the concentration behavior in Catalan trees.

\medskip

In the \emph{simple} butterfly model, the encoding compresses further: the recursive construction is determined by a bitstring $\mathbf{x}\in\{0,1\}^n$, corresponding to identical merge choices at each level. In this case, the HS number can be computed in $\mathcal O(n)$ time directly from $\mathbf{x}$ (see \Cref{sec: computing} and \Cref{alg:HS number_run_length}), an exponential improvement over the $\mathcal O(N)$ complexity required for general binary trees. This compression enables experiments at otherwise inaccessible scales; for example, Figure~\ref{fig: HS number normal limit} shows a histogram based on $10^6$ samples of simple butterfly trees with $N=2^{100{,}000}$ nodes.

For random simple butterfly trees, the HS number admits an exact representation as an additive functional of a finite-state Markov chain. Writing
\[
X_j = \hs(\mathcal T_j^{\B}) - \hs(\mathcal T_{j-1}^{\B}),
\]
the increment process $(X_j)$ is determined by the local gluing rules (see \Cref{sec: block}), which simplify substantially in the simple butterfly setting. As a result, the HS recursion reduces to a finite-state dynamical system, and
\[
\hs(\mathcal T_n^{\B}) = \sum_{j=1}^n X_j
\]
can be expressed as an additive functional of a finite irreducible Markov chain.

This reduction fundamentally alters the probabilistic behavior of the model. In classical random tree models (e.g., Catalan trees), the HS statistic depends on global combinatorial structure and exhibits persistent oscillatory effects that obstruct standard limit theorems. In contrast, the simple butterfly construction places the HS process in an ergodic Markovian setting, making it amenable to classical limit theory, including Gaussian fluctuations.

\begin{figure}
    \centering
    \includegraphics[width=0.7\linewidth]{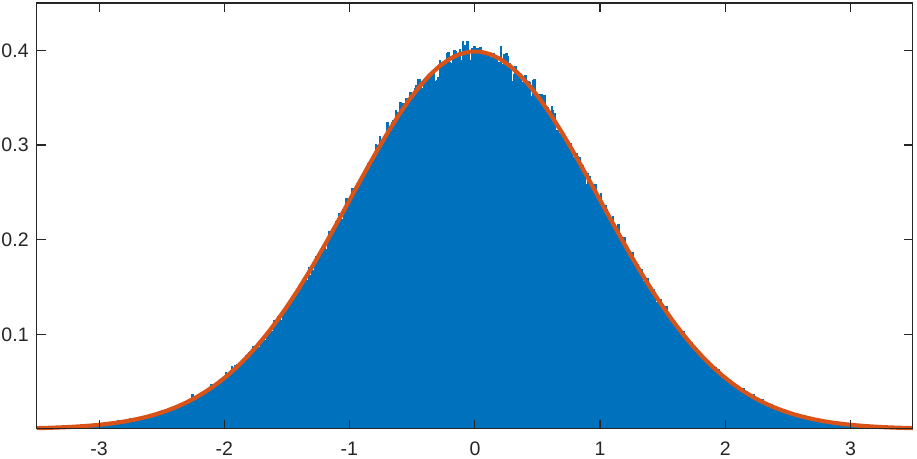}
    \caption{Histogram of normalized HS numbers for $1{,}000{,}000$ uniform simple butterfly trees with $2^{100{,}000} \approx 10^{30{,}103}$ nodes.}
    \label{fig: HS number normal limit}
\end{figure}

\begin{theorem}[Markov modeling of HS number]\label{thm:HS number_MC_model}
Let $(x_j)_{j\ge1}$ be iid $\Bern(p)$ and define
\[
X_0=0, 
\qquad 
X_j=\textnormal{\texttt{xor}}(x_j,x_{j-1})\,(1-X_{j-1}), \quad j\ge1,
\]
and
\[
M_j=(x_j,x_{j-1},X_{j-1})\in\{0,1\}^3.
\]
Then $(M_j)$ is an irreducible, aperiodic $8$-state Markov chain and
\[
\hs(\mathcal T_n^{\B}(\mathbf x))=\sum_{j=1}^n X_j = \sum_{j=1}^n f(M_j),
\]
where $f(a,b,c)=\textnormal{\texttt{xor}}(a,b)(1-c)$.
\end{theorem}

Thus, the HS number of simple butterfly trees is an additive functional of a finite irreducible Markov chain, placing the model in a classical ergodic setting and yielding a complete limit theory.

\begin{theorem}[SLLN]\label{thm:SLLN}
For $p\in[0,1]$,
\[
\frac{1}{n}\hs(\mathcal T_n^{\B})
\xrightarrow[n\to\infty]{\mathrm{a.s.}}
\mu_p
=
\frac{pq}{1-pq}, \qquad q=1-p.
\]
\end{theorem}

The almost sure convergence in Theorem~\ref{thm:SLLN} strengthens the typical state of results for HS statistics in random tree models, where one generally obtains only weak laws of large numbers for suitably normalized versions of the HS number.

\begin{theorem}[Functional CLT]\label{thm:FCLT}
Define
\[
W_n(p,t)
=
\sqrt{n}\left(\frac{1}{n}\sum_{j=1}^{\lfloor nt\rfloor}X_j - t\mu_p\right),
\qquad t\in[0,1].
\]
Then
\[
W_n(p,\cdot)\Longrightarrow \sigma_p B(\cdot)
\quad \text{in } D[0,1], \qquad 
\sigma_p^2
=
\frac{pq(1-3pq-2p^2q^2)}{(1-pq)^3}.
\]
\end{theorem}

Theorem~\ref{thm:FCLT} upgrades the limit theory from scalar fluctuations to a full functional central limit theorem, providing a Gaussian process description of HS fluctuations at diffusive scale with explicit variance. 

As immediate consequences, setting $t=1$ yields both a weak law and a central limit theorem:
\begin{corollary}[WLLN]\label{cor:wlln}
\[
\frac{\hs(\mathcal T_n^{\B})}{n}
\xrightarrow[n\to\infty]{\mathbb{P}} \mu_p.
\]
\end{corollary}

\begin{corollary}[CLT]\label{cor:CLT}
\[
\frac{\hs(\mathcal T_n^{\B})-\mu_p n}{\sqrt n}
\Longrightarrow \mathcal N(0,\sigma_p^2).
\]
\end{corollary}

In particular, the HS number exhibits Gaussian fluctuations with linear variance growth. The limiting constant satisfies $\mu_p \le \mu_{1/2} = 1/3$, strictly below the Catalan tree benchmark $1/2$, reflecting the reduced effective branching complexity induced by the butterfly construction. Together, these results provide a setting in which both first- and second-order asymptotics of the HS number are fully explicit.

\medskip

We next consider general butterfly trees, in which level-$(n+1)$ trees are formed from independent copies of level-$n$ trees via the recursion
\[
\hs(\mathcal T_{n+1}^{\B})
=
\max\left(\hs(\mathcal T_n^{\B}), \hs(\widetilde{\mathcal T}_n^{\B})\right) + I_n,
\]
where $I_n \in \{0,1\}$ is determined by the interface profiles of the input trees and the choice of $\oplus/\ominus$.

In contrast to the simple butterfly model, this recursion does not admit a finite-state representation: the interface profile grows with $n$ and governs the increment mechanism. Consequently, the HS process evolves via a scale-dependent max-plus recursion.

Using exact computations for small $n$ and Monte Carlo simulations up to $n=25$, we obtain the following asymptotic behavior.

\begin{conjecture}[WLLN]\label{conj: wlln butterfly}
For $\mathcal T_n^{\B}$ a uniform butterfly tree with $N = 2^n$ nodes,
\[
\frac{1}{n}\hs(\mathcal T_n^{\B})
\xrightarrow[n\to\infty]{\mathbb{P}} \alpha \approx 0.4450.
\]
\end{conjecture}

\begin{figure}
    \centering
    \includegraphics[width=0.75\linewidth]{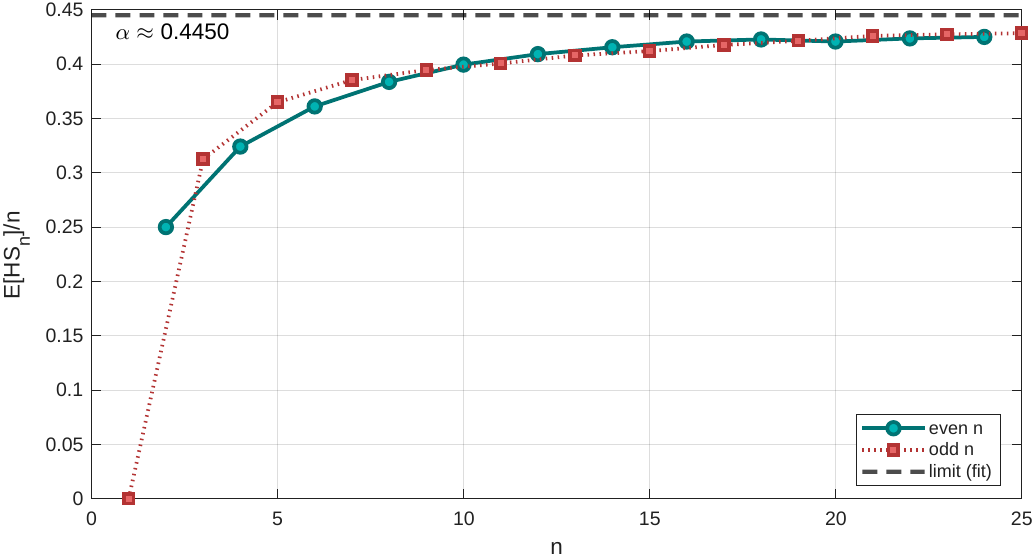}
    \caption{Plot of $\frac1n\E[\hs(\mathcal T_n^{\B})]$ using exact computations for $n \le 8$ and Monte Carlo sampling up to $n = 25$ for uniform butterfly trees $\mathcal T_n^{\B}$ with $N = 2^n$ nodes.}
    \label{fig:mean_limit}
\end{figure}

\Cref{fig:mean_limit} plots values of $\frac1n \E[\hs(\mathcal T_n^{\B})]$, along with a computed best fit limiting $\alpha \approx 0.4450$, which further highlights parity structure on each exact (up to $n = 8$) or sampled mean (up to $n = 25$). In contrast to classical random tree models, the limiting behavior here is governed by a cascade-driven max-plus recursion with a growing interface state, rather than a fixed-point distributional equation. Moreover, assuming \Cref{conj: wlln butterfly}, then uniform butterfly trees comprise a distinct universality class with limiting scaled HS number that lies between the $1/3$ limit seen by simple butterfly trees and the $1/2$ limit of Catalan trees.

\subsection{Outline} 
The remainder of \Cref{sec: intro} provides additional notation conventions and background. \Cref{sec: block} outlines in more details the HS number dynamics determined by the a single gluing operator $\oplus/\ominus$ and contains the proof of Theorem \ref{thm:block} for the HS number of block trees. This includes \Cref{prop: merge increment,prop: profile update}, that explicitly outline merge increment and profile update dynamics that drive the recursive HS modeling in subsequent sections. \Cref{sec: butterfly} addresses the results for butterfly trees. We first establish the reduced support for the HS number (\Cref{prop: butterfly support}), before shifting focus to simple butterfly in \Cref{sec: simple}. We provide a recursive and direct $\mathcal O(n)$ method to compute the HS number of simple butterfly trees in \Cref{sec: computing}, while \Cref{sec: random simple} provides the proof of the finite-state Markov process modeling of the HS number (\Cref{thm:HS number_MC_model}), that is used to then prove the main limit theorems (\Cref{thm:SLLN,thm:FCLT} and \Cref{cor:wlln,cor:CLT}), while \Cref{sec: unif simple} outlines additional combinatorial results for the uniform case ($p=1/2$). \Cref{sec: nonsimple} returns to the general butterfly tree models, where the general RDE is revisited in \Cref{sec: rand gen butterfly}, while \Cref{sec: limiting dist} includes exact and empirical results in support of the WLLN for uniformly random butterfly trees (\Cref{conj: wlln butterfly}). \Cref{sec: conclusions} addresses open problems and future directions for butterfly trees. \Cref{sec: algorithms} includes explicit algorithms to compute the HS number of butterfly trees from an input bitstring, which were used for the supporting exact computations and experiments with further results summarized in \Cref{sec: exact,sec: experiment data}. 

\subsection{Notation and background}
\label{sec: notation}

Write $\operatorname{Sym}(m)$ for the symmetric group of permutations $\pi$ of length $m$. For permutations $\pi \in \operatorname{Sym}(n)$ and $\pi' \in \operatorname{Sym}(m)$, then we define the direct sum $\pi \oplus \pi'$ and the skew sum $\pi \ominus \pi'$ permutations in $\operatorname{Sym}(n + m)$ by the induced action on the underlying permutation matrices, $P_{\pi \oplus \pi'} = P_\pi \oplus P_{\pi'}$ and $P_{\pi \ominus \pi'} = P_{\pi} \ominus P_{\pi'}$, where
\[
A \oplus B = \begin{bmatrix}
    A & \V 0\\ \V 0 & B
\end{bmatrix}, \qquad A \ominus B = \begin{bmatrix}
    \V 0 & B\\ A & \V0
\end{bmatrix}.
\]

\textit{Butterfly permutations} $\B_n \subset \operatorname{Sym}(2^n)$ are formed by iterated wreath products of $\operatorname{Sym}(2) \cong C_2$, where $\B_1 = \operatorname{Sym}(2) = \{1,(1 \ 2)\}$ and
\[
\B_n  = (\B_{n-1} \oplus \B_{n-1}) \rtimes \operatorname{Sym}(2) \cong C_2 \wr \cdots \wr C_2 = C_2^{\wr n},
\]
where the semidirect action of $\rho \in \operatorname{Sym}(2) = \{1, (1 \ 2)\}$ on $ \pi_1 \oplus   \pi_2 \in \B_n \oplus \B_n$ is 
\[
\rho \cdot ( \pi_1 \oplus  \pi_2) =\pi_{\rho(1)} \oplus  \pi_{\rho(2)}.
\]
This aligns with the definition of the wreath product, where we note then precisely each element of $\B_n$ is of the form ${\pi} \oplus {\pi}'$ or ${\pi} \ominus {\pi}'$ for $\pi,\pi' \in \B_{n-1}$. The \textit{simple} butterfly permutations $$\B_{n,s} \cong C_2^{\otimes n}$$ are similarly formed using iterated Kronecker products, that correspond to using identical copies of $\pi_1 = \pi_2 \in \B_{n-1,s}$ at each recursive step. These constructions result in the simple butterfly permutations comprising an abelian subgroup of order $N = 2^n$ inside the group of butterfly permutations of order $2^{N-1}$ that thus themselves comprise a 2-Sylow subgroup of $\operatorname{Sym}(2^n)$.

Moreover, butterfly permutations thus also are comprised of \emph{separable} permutations, which are permutations formed using only direct and skew sums; equivalently separable permutations are permutations that avoid the patterns 2413 and 3142.

\section{Block trees}\label{sec: block}

We initiate our HS number study on \textit{block trees}, formed by merging two rooted binary trees $\mathcal T_1$ and $\mathcal T_2$ using a single operator $\oplus$ or $\ominus$, where we will refer to $\mathcal T_1$ as the parent tree and $\mathcal T_2$ as the child tree. The resulting trees $\mathcal T_1 \oplus \mathcal T_2$ and $\mathcal T_1 \ominus \mathcal T_2$ arise from BST constructions where $\mathcal T(\pi_i)$ has underlying shape $\mathcal T_i$, so that these are precisely the shapes of $\mathcal T(\pi_1 \oplus \pi_2)$ and $\mathcal T(\pi_1 \ominus \pi_2)$.

These are special cases of block trees introduced in \cite{PZ25}, given by
\[
\mathcal T_1 \oplus \mathcal T_2 = \mathcal T(12) \, \boxdot \, (\mathcal T_1,\mathcal T_2), 
\qquad 
\mathcal T_1 \ominus \mathcal T_2 = \mathcal T(21) \, \boxdot \, (\mathcal T_1,\mathcal T_2).
\]
Equivalently, $\oplus$ (resp.\ $\ominus$) attaches the root of $\mathcal T_2$ to the terminal node of the top-right (resp.\ top-left) edge of $\mathcal T_1$. These nodes correspond to the extremal BST labels, ensuring that all keys in $\mathcal T_2$ are larger (resp.\ smaller) than those of $\mathcal T_1$. See \Cref{fig: BST gluing example,fig:gluing-perfect-binary} for illustrations.

A related study for BST height was carried out in \cite{PZ25}. There, a single gluing operation on two random BSTs with $m$ nodes increases the asymptotic height compared to a uniform BST on $2m$ nodes. Recall that for a random BST with $m$ nodes and uniform permutation labels, $\texttt{height}(\mathcal T_m) \sim c^* \log m$, where $c^* \approx 4.311$ solves $x \log(2e/x)=1$ \cite{devroye1986note}. In contrast, for iid $\pi,\pi' \sim \Unif(\operatorname{Sym}(m))$, the glued trees satisfy
\[
\texttt{height}(\mathcal T(\pi \oplus \pi')) \sim (c^*+1)\log(2m),
\]
and similarly for $\ominus$ \cite{PZ25}.

A key ingredient in this analysis is the length evolution along the top edges, requiring one to track both subtree profiles and the joining dynamics. For the HS number, an analogous but more intricate structure arises: while height depends only on edge lengths, the HS number requires tracking the full HS profile along each top edge, together with how these profiles interact under merging.

\begin{figure}
    \centering
    \begin{subfigure}{.5\textwidth}
    \centering
    \resizebox{\textwidth}{!}{
    \begin{tikzpicture}[x=1.5cm, y=1.5cm]
        \node[circle, draw] (n0d) at (3,3) {2};
        
        \node[circle, draw] (n00d) at (1,2) {2};
        \node[circle, draw] (n01d) at (5,2) {1};
        
        \node[circle, draw] (n000d) at (0,1) {1};
        \node[circle, draw] (n001d) at (2,1) {1};
        \node[circle, draw] (n010d) at (4,1) {1};
        
        \node[circle, draw] (n0000d) at (-0.5,0) {0};
        \node[circle, draw] (n0001d) at (0.5,0) {0};
        \node[circle, draw] (n0010d) at (1.5,0) {0};
        \node[circle, draw] (n0011d) at (2.5,0) {0};
        \node[circle, draw] (n0100d) at (3.5,0) {0};
        \node[circle, draw] (n0101d) at (4.5,0) {0};
        
        \draw[thick] (n0d) -- (n00d);
        \draw[thick] (n0d) -- (n01d);
        
        \draw[thick] (n00d) -- (n000d);
        \draw[thick] (n00d) -- (n001d);
        \draw[thick] (n01d) -- (n010d);

        \draw[thick] (n000d) -- (n0000d);
        \draw[thick] (n000d) -- (n0001d);
        \draw[thick] (n001d) -- (n0010d);
        \draw[thick] (n001d) -- (n0011d);
        \draw[thick] (n010d) -- (n0100d);
        \draw[thick] (n010d) -- (n0101d);
    \end{tikzpicture}
    }
    \caption{$\mathcal T_1$}
    \end{subfigure}%
    \hfill
    \begin{subfigure}{.25\textwidth}
    \centering
    \resizebox{.5\textwidth}{!}{
    \begin{tikzpicture}[x=1.5cm, y=1.5cm]
        \node[circle, draw] (n011d) at (0,1) {1};
        \node[circle, draw] (n0110d) at (-0.5,0) {0};
        \node[circle, draw] (n0111d) at (0.5,0) {0};

        \draw[thick] (n011d) -- (n0110d);
        \draw[thick] (n011d) -- (n0111d);
    \end{tikzpicture}
    }
    \caption{$\mathcal T_2$}
    \end{subfigure}%
    \hfill
    \begin{subfigure}{.6\textwidth}
    \centering
    \resizebox{\textwidth}{!}{
    \begin{tikzpicture}[x=1.5cm, y=1.5cm]
        \node[line width=1mm, circle, draw=blue, fill=blue!20] (n0d) at (3,3) {$\V3$};
        
        \node[circle, draw] (n00d) at (1,2) {2};
        \node[line width=1mm, circle, draw=blue, fill=blue!20] (n01d) at (5,2) {$\V2$};
        
        \node[circle, draw] (n000d) at (0,1) {1};
        \node[circle, draw] (n001d) at (2,1) {1};
        \node[circle, draw] (n010d) at (4,1) {1};
        \node[circle, draw] (n011d) at (6,1) {1};
        
        \node[circle, draw] (n0000d) at (-0.5,0) {0};
        \node[circle, draw] (n0001d) at (0.5,0) {0};
        \node[circle, draw] (n0010d) at (1.5,0) {0};
        \node[circle, draw] (n0011d) at (2.5,0) {0};
        \node[circle, draw] (n0100d) at (3.5,0) {0};
        \node[circle, draw] (n0101d) at (4.5,0) {0};
        \node[circle, draw] (n0110d) at (5.5,0) {0};
        \node[circle, draw] (n0111d) at (6.5,0) {0};
        
        \draw[thick] (n0d) -- (n00d);
        \draw[line width=1mm, blue] (n0d) -- (n01d);
        
        \draw[thick] (n00d) -- (n000d);
        \draw[thick] (n00d) -- (n001d);
        \draw[thick] (n01d) -- (n010d);
        \draw[line width=1mm, dotted, blue] (n01d) -- (n011d);

        \draw[thick] (n000d) -- (n0000d);
        \draw[thick] (n000d) -- (n0001d);
        \draw[thick] (n001d) -- (n0010d);
        \draw[thick] (n001d) -- (n0011d);
        \draw[thick] (n010d) -- (n0100d);
        \draw[thick] (n010d) -- (n0101d);
        \draw[thick] (n011d) -- (n0110d);
        \draw[thick] (n011d) -- (n0111d);
    \end{tikzpicture}
    }
    \caption{$\mathcal T_1 \oplus \mathcal T_2$}
    \end{subfigure}
    \caption{$\mathcal T_1$, $\mathcal T_2$, and their gluing $\mathcal T_1 \oplus \mathcal T_2$ showing the cascading increase in HS number.}
    \label{fig:gluing-perfect-binary}
\end{figure}

We will now outline precisely how the composite tree can have a strictly larger HS number than the input trees.

\begin{proposition}[Merge increment]\label{prop: merge increment}
    Let $\mathcal T_1,\mathcal T_2$ be rooted binary trees with $m$ nodes. Then 
    \begin{equation}\label{eq: HS number merge ineq}
        \max(\hs(\mathcal T_1),\hs(\mathcal T_2)) \le \hs(\mathcal T_1 \oplus \mathcal T_2)\le \max(\hs(\mathcal T_1),\hs(\mathcal T_2)) + 1,
    \end{equation}
    and the upper bound is tight if and only if all of the following increment conditions are met:
    \begin{enumerate}[label=(\roman*)]
        \item \textbf{Max parent}: $\hs(\mathcal T_1) \ge \hs(\mathcal T_2)$.
        \item \textbf{Full profile}: The top right edge of $\mathcal T_1$ has nodes with HS numbers $\hs(\mathcal T_2),\hs(\mathcal T_2) + 1, \ldots, \hs(\mathcal T_1)$.
        \item \textbf{Escape paths}: For each $t \in \{\hs(\mathcal T_2),\ldots,\hs(\mathcal T_1)\}$, the top right edge of $\mathcal T_1$ has a node $v$ with value $\hs(v) = t$ that has a child node $w$ located off of top right edge with $\hs(w) = t$.
    \end{enumerate}
    Identical results hold for $\mathcal T_1 \ominus \mathcal T_2$ when replacing the top right edge with the top left edge in the above increment conditions.
\end{proposition}

\begin{remark}
    In the composite tree, the only nodes that can have HS numbers differ from the corresponding labels from the input trees are found precisely on the top gluing edge of the parent tree. All other nodes inherit the previous HS numbers from the input trees. This is visualized in \Cref{fig: BST gluing example,fig:gluing-perfect-binary} with the top gluing edge of the parent tree highlighted in blue, and the HS number labels that differ from the input labels further highlighted in light blue. 
\end{remark}

\begin{proof}
The statement for $\mathcal T_1 \ominus \mathcal T_2$ follows by symmetry under reflection: if $\mathcal T^R$ denotes reflection of $\mathcal T$ about its root, then $\hs(\mathcal T) = \hs(\mathcal T^R)$ and 
\[
(\mathcal T_1 \ominus \mathcal T_2)^R = \mathcal T_1^R \oplus \mathcal T_2^R.
\]

First, we establish the inequalities in \eqref{eq: HS number merge ineq}. We will use the characterization of $\hs(\mathcal T)$ as the height of the largest embedded perfect binary subtree within $\mathcal T$. The lower bound is immediate, since any perfect subtree of $\mathcal T_1$ or $\mathcal T_2$ remains embedded after gluing. For the upper bound, any increase beyond this maximum must come from a perfect subtree within the composite tree that uses the gluing edge. Note in particular that the left and right subtrees of the root of the embedded perfect tree (note we can identity the root as the composite root) comprise perfect binary trees of height one less than this tree. Removing a single edge thus makes the binary tree no longer perfect, and the largest remaining perfect subtree is then contained entirely within the opposite subtree from the child root (and so is embedded within the parent tree). So its height decreases by at most one. See \Cref{fig:gluing-perfect-binary} for an example of such an increment.

\medskip

We now characterize when the increment occurs. Let
\[
\textbf{(A)} \quad \hs(\mathcal T_1 \oplus \mathcal T_2)
=
\max(\hs(\mathcal T_1),\hs(\mathcal T_2)) + 1.
\]
We recall again updates to HS numbers in the composite tree occur only along the top gluing edge of the parent tree and propagate upward with increments of at most one per step.

\medskip

\noindent
\fbox{\textbf{$\neg$(i)$\Rightarrow\neg$(A)}}
If $\hs(\mathcal T_2) > \hs(\mathcal T_1)$, then all nodes on the gluing edge of $\mathcal T_1$ have strictly smaller HS numbers. The propagated updates from the child tree thus form a path of nodes with HS number $\hs(\mathcal T_2)$ from the child root to the composite root. Since every node in $\mathcal T_1$ has HS number strictly smaller than $\hs(\mathcal T_2)$, the merging cannot create two children of equal maximal value for any node along this path. Hence, no increment occurs in particular at the composite root.

\medskip

\noindent
\fbox{\textbf{(i)$\wedge \neg$(ii)$\Rightarrow\neg$(A)}}
Assume $\hs(\mathcal T_1)\ge \hs(\mathcal T_2)$ but the full profile condition fails. Then necessarily $\hs(\mathcal T_1) > \hs(\mathcal T_2)$, since otherwise the interval $[\hs(\mathcal T_2),\hs(\mathcal T_1)]$ is degenerate. Let $t$ be the largest value in $[\hs(\mathcal T_2),\hs(\mathcal T_1)]$ that does not appear on the top gluing edge, and let $v$ be the first node along this edge with $\hs(v)=t+1$. Then $v$ must inherit its HS number within $\mathcal T_1$ from a child $w$ off the gluing edge, with $\hs(w) = t+1$. After merging, updates along the gluing edge can increase the HS number for any edge descendant of $v$ (if they exist) by at most one, so the composite gluing edge child $w'$ has HS number at most $t$ (if $v$ had no gluing edge child in $\mathcal T_1$, then $w'$ is the root of $\mathcal T_2$). Thus in the composite tree, $v$ has one child of value $t+1$ and one child of value at most $t$, and hence $\hs(v)=t+1$. Therefore no increment propagates past $v$ as this matches the HS value from the parent tree, and the HS numbers above $v$, including at the root, remain unchanged.

\medskip

\noindent
\fbox{\textbf{(i)$\wedge$(ii)$\wedge\neg$(iii)$\Rightarrow\neg$(A)}}
Let $t$ be maximal in $[\hs(\mathcal T_2),\hs(\mathcal T_1)]$ such that no node off the gluing edge of $\mathcal T_1$ has HS number $t$. Let $v$ be the maximal depth gluing edge node in $\mathcal T_1$ with $\hs(v)=t$, which then necessarily has two children, each with HS number $t-1$. Propagated updates through the merging can raise the gluing edge child HS number in the composite to at most $t$ while its internal child still has HS number $t-1$, so $\hs(v)$ remains at level $t$. Thus the increment cannot propagate past $v$, and the remaining HS number profile on the gluing edge above $v$ is inherited from that of $\mathcal T_1$.

\medskip

\noindent
\fbox{\textbf{(i)$\wedge$(ii)$\wedge$(iii)$\Rightarrow$(A).}}
Under all three conditions, for each $t \in [\hs(\mathcal T_2),\hs(\mathcal T_1)]$ there is a node $v$ on the gluing edge with an off-edge child $w$ with $\hs(v) = \hs(w) = t$ within $\mathcal T_1$. The initial update at the terminal node introduces value $\hs(\mathcal T_2)$ and propagates upward along the edge. When this propagated value reaches such a node at level $t$, its edge child now has value $t$ while its off-edge child already has value $t$, forcing an increment to $t+1$ at $v$. This increment then propagates further up the edge and repeats the same mechanism at the next level. Iterating this argument produces a cascade of increments up to the root, yielding $\hs(\mathcal T_1 \oplus \mathcal T_2)
=
\hs(\mathcal T_1)+1$.

\medskip

These implications together establish \fbox{\textbf{(i)$\wedge$(ii)$\wedge$(iii)$\Leftrightarrow$(A)}}, as we wanted to show.
\end{proof}

We can further condense the encoding of the merge increment dynamics via HS profiles, which track only the relevant information along the top edges. For a rooted binary tree $\mathcal T$, define an \textit{edge profile} $(m,(a_k,b_k)_{k\ge 0})$ by:
\begin{itemize}
    \item $m$ indicates whether there is a multiplicity of nodes on the edge with maximal HS number,
    \item $a_k \in \{0,1\}$ indicates whether the edge contains a node with HS number $k$,
    \item $b_k \in \{0,1\}$ indicates whether the maximal-depth node on the edge with HS number $k$ has a child off the edge with the same HS number (an \emph{escape path}).
\end{itemize}
If $a_k=1$, then nodes with HS number $k$ form a path along the edge. If additionally $b_k=1$, this path extends off the edge to an internal node (escaping off the edge), while if $b_k=0$, the maximal-depth node with HS number $k$ has two children of value $k-1$, so the path is confined to the edge (i.e., no escape path).

Since only nodes along the top gluing edge of the parent tree can change under merging, it suffices to track these edge profiles rather than the full tree. The only additional interaction between the two top edges occurs at the root, and is captured by the multiplicity indicator for the maximal HS value.

We therefore define the HS profile of $\mathcal T$ as
\[
\texttt{profile}(\mathcal T) = (\hs(\mathcal T); L; R),
\]
where $L$ and $R$ are the top left and right edge profiles. Each edge profile can be represented as a $2 \times (\hs(\mathcal T)+1)$ matrix with rows $(a_k)$ and $(b_k)$ for $0 \le k \le \hs(\mathcal T)$, together with the multiplicity bit. We note reflection preserves HS number and swaps edge profiles: if $(h;L;R)$ is the profile of $\mathcal T$, then $(h;R;L)$ is the profile of $\mathcal T^R$. 

\begin{example}
The butterfly tree $\mathcal T_5^{\B}$ in \Cref{fig: BST HS number keys} has profile
\[
\textnormal{\texttt{profile}}(\mathcal T_5^{\B}) 
= \left(2;\; 1, \begin{bmatrix}
1&1&1\\
1&0&0
\end{bmatrix};\; 0, \begin{bmatrix}
1&1&1\\
0&1&1
\end{bmatrix}\right).
\]
Since $\hs(\mathcal T_5^{\B})=2$, each edge profile is encoded by a $2\times 3$ matrix (for $k=0,1,2$) together with a multiplicity bit. The top rows are all $1$’s, as both edges contain nodes with HS numbers $0,1,2$. On the left edge, the maximal-depth node with HS number $0$ inherits from an off-edge child, so $b_0=1$, while for $k=1,2$ the maximal-depth nodes increment from two children of value $k-1$, giving $b_1=b_2=0$. On the right edge, the maximal-depth node with HS number $0$ is a leaf, so $b_0=0$, whereas the maximal-depth nodes with HS numbers $1$ and $2$ inherit from off-edge children, yielding $b_1=b_2=1$. Finally, only the left edge contains multiple nodes with maximal HS number $2$, so $m^{(L)}=1$ and $m^{(R)}=0$.
\end{example}

We now describe the profile update under $\oplus$ using \Cref{prop: merge increment}. Let $\mathcal T_p$ (parent) and $\mathcal T_c$ (child) have profiles $(h_p;L_p;R_p)$ and $(h_c;L_c;R_c)$.

\medskip

\noindent\textbf{Case 1: $h_c > h_p$.}
Then the path from the child root to the parent root carries HS number $h_c$, so multiple nodes attain this value. Moreover, then since the maximal depth node with HS value $t \le h_c$ are all located within the child tree, then the right profile is inherited from the child, with now $m^{(R)}=1$.

On the left edge, the parent profile is unchanged below level $h_p$, while a new level appears at $h_c$ with
\[
(a_{h_c}^{(L)},b_{h_c}^{(L)})=(1,1), \qquad m^{(L)}=0.
\]
At level $h_p$, the entry is inherited if $m^{(L_p)}=1$ (as the max depth node with HS number $h_p$ remains unchanged), and otherwise removed (since the root is updated, removing the only node with HS number $h_p$ if $m^{(L_p)} = 0$), giving $(0,0)$.

\medskip

\noindent\textbf{Case 2: $h_c \le h_p$.}
Attach the child at the terminal node of the top-right edge. Let $v$ be the maximal-depth node with $\hs(v)=h_c$. The update propagates upward along this edge, and an increment occurs at level $k$ precisely when $(a_k^{(R_p)},b_k^{(R_p)})=(1,1)$. This produces a cascade of increments up to
\[
k^\star := \min\{k \ge h_c : (a_k^{(R_p)},b_k^{(R_p)}) \neq (1,1)\}.
\]

Thus:
\begin{itemize}
    \item For $k \le h_c$, the right profile is inherited from the child. (The max depth edge node with HS number $t \le h_c$ is in the child tree.)
    \item For $h_c < k \le k^\star$, we obtain $(a_k,b_k)=(1,0)$ from the cascade. (Each max depth edge node with HS value $k$ incremented from having children with matching HS number.)
    \item For $k^\star < k \le h_p$, the profile matches that of the parent. (The cascade increment stops below this level, so the remaining parent tree profile remains unchanged.)
\end{itemize}

\medskip

\noindent\textit{Subcases.}

\begin{itemize}
    \item If $k^\star = h_p+1$, the cascade reaches the root, so
    \[
    h = h_p+1 = \max(\hs(\mathcal T_p),\hs(\mathcal T_c)) + 1.
    \]
    The right multiplicity is preserved, $m^{(R)}=m^{(R_p)}$, as the cascade increment overwrite the previous max HS number path from the parent. The left profile agrees with the parent below $h_p$, and gains a new level with
    \[
    (a_{h_p+1}^{(L)},b_{h_p+1}^{(L)})=(1,m^{(R)}), \qquad m^{(L)}=0.
    \]

    \item If $k^\star = h_p$, the cascade stops just below the root. Then $h=h_p$ and the right edge contains a path of nodes with value $h_p$, so $m^{(R)}=1$. The left edge updates to
    \[
    (a_{h_p}^{(L)},b_{h_p}^{(L)})=(1,1), \qquad m^{(L)}=0.
    \]

    \item If $k^\star < h_p$, the cascade stops strictly below the top level. The remaining right profile and multiplicity are inherited from the parent, and the left profile is unchanged.
\end{itemize}

\medskip

The $\ominus$ update follows by symmetry:
\[
\mathcal T_p \ominus \mathcal T_c = (\mathcal T_p^R \oplus \mathcal T_c^R)^R,
\]
since reflection swaps left and right profiles while preserving HS number.

\medskip

We summarize these rules formally in \Cref{prop: profile update}.

\begin{proposition}
    [Profile update]\label{prop: profile update}
    Let $\mathcal T_p$ (parent) and $\mathcal T_c$ (child) have profiles
    \[
    (h_p; L_p; R_p), \qquad (h_c; L_c; R_c),
    \]
    and define
    \[
    k^\star := \min\{k \ge h_c : (a_k^{(R_p)}, b_k^{(R_p)}) \neq (1,1)\}.
    \]
    
    Then $\mathcal T_p \oplus \mathcal T_c$ has profile $(h; L; R)$ as follows.
    
    \medskip
    
    \noindent\textbf{\textnormal{(i)} HS number.}
    \[
    h =
    \begin{cases}
    h_c, & h_c > h_p,\\
    h_p + 1, & h_c \le h_p \text{ and } k^\star = h_p + 1,\\
    h_p, & \text{otherwise}.
    \end{cases}
    \]
    
    \medskip
    
    \noindent\textbf{\textnormal{(ii)} Right profile.}
    
    If $h_c > h_p$, then
    \[
    (a_k^{(R)}, b_k^{(R)}) = (a_k^{(R_c)}, b_k^{(R_c)}) \quad (k \le h_c),
    \qquad m^{(R)} = 1.
    \]
    
    If $h_c \le h_p$, then
    \[
    (a_k^{(R)}, b_k^{(R)}) =
    \begin{cases}
    (a_k^{(R_c)}, b_k^{(R_c)}), & k \le h_c,\\
    (1,0), & h_c < k \le k^\star,\\
    (a_k^{(R_p)}, b_k^{(R_p)}), & k^\star < k \le h_p,
    \end{cases}
    \]
    and
    \[
    m^{(R)} =
    \begin{cases}
    1, & k^\star = h_p,\\
    m^{(R_p)}, & \text{otherwise}.
    \end{cases}
    \]
    
    \medskip
    
    \noindent\textbf{\textnormal{(iii)} Left profile.}
    
    For all $k < h_p$,
    \[
    (a_k^{(L)}, b_k^{(L)}) = (a_k^{(L_p)}, b_k^{(L_p)}).
    \]
    
    At the top level:
    
    \begin{itemize}
        \item \textbf{If $h_c > h_p$:}
        \[
        (a_{h_c}^{(L)}, b_{h_c}^{(L)}) = (1,1), \qquad m^{(L)} = 0,
        \]
        and
        \[
        (a_{h_p}^{(L)}, b_{h_p}^{(L)}) =
        \begin{cases}
        (a_{h_p}^{(L_p)}, b_{h_p}^{(L_p)}), & m^{(L_p)} = 1,\\
        (0,0), & m^{(L_p)} = 0.
        \end{cases}
        \]
    
        \item \textbf{If $h_c \le h_p$ and $k^\star = h_p+1$:}
        \[
        (a_{h_p+1}^{(L)}, b_{h_p+1}^{(L)}) = (1,\, m^{(R)}), \qquad m^{(L)} = 0,
        \]
        and
        \[
        (a_{h_p}^{(L)}, b_{h_p}^{(L)}) =
        \begin{cases}
        (a_{h_p}^{(L_p)}, b_{h_p}^{(L_p)}), & m^{(L_p)} = 1,\\
        (0,0), & m^{(L_p)} = 0.
        \end{cases}
        \]
    
        \item \textbf{If $h_c \le h_p$ and $k^\star = h_p$:}
        \[
        (a_{h_p}^{(L)}, b_{h_p}^{(L)}) = (1,1), \qquad m^{(L)} = 0.
        \]
    
        \item \textbf{If $h_c \le h_p$ and $k^\star < h_p$:}
        \[
        (a_{h_p}^{(L)}, b_{h_p}^{(L)}) = (a_{h_p}^{(L_p)}, b_{h_p}^{(L_p)}),
        \qquad m^{(L)} = m^{(L_p)}.
        \]
    \end{itemize}
    
    \medskip
    
    \noindent The profile for $\mathcal T_p \ominus \mathcal T_c$ is obtained via
    \[
    \mathcal T_p \ominus \mathcal T_c = (\mathcal T_p^R \oplus \mathcal T_c^R)^R.
    \]
\end{proposition}

\begin{example}
    We can see the profile of the perfect binary tree in \Cref{fig:gluing-perfect-binary} $$\textnormal{\texttt{profile}}(\mathcal T_1 \oplus \mathcal T_2) = \left(3; \left(0; \begin{bmatrix}
        1&1&1&1\\0&0&0&0
    \end{bmatrix}\right); \left(0; \begin{bmatrix}
        1&1&1&1\\0&0&0&0
    \end{bmatrix}\right)\right)$$ formed using 
    \begin{align*}
        \textnormal{\texttt{profile}}(\mathcal T_1) &= \left(2; \left(1,\begin{bmatrix}
        1&1&1\\0&0&0
    \end{bmatrix}\right); \left(0, \begin{bmatrix}
        0&1&1\\0&1&1
    \end{bmatrix}\right)\right), \\ 
    \textnormal{\texttt{profile}}(\mathcal T_2) &= \left(1; \left(0,\begin{bmatrix}
        1&1\\0&0
    \end{bmatrix}\right); \left(0, \begin{bmatrix}
        1&1\\0&0
    \end{bmatrix}\right)\right).
    \end{align*}
    since the top-right edge of $\mathcal T_1$ has ``escape'' paths number for HS numbers 1 and 2, so that the $\oplus$ merge with $\mathcal T_2$ that introduces a new right child with HS number 1 then initiates a cascading increment to the root.
\end{example}

We will revisit this explicit profile merging to build out recursive functions to compute the HS numbers for butterfly trees in \Cref{sec: butterfly}. \Cref{prop: profile update} is used to build efficient recursive functions to sample HS numbers of random butterfly trees (see \Cref{alg:hs_profile_corrected,alg:HS number_run_length}).

\medskip

We next consider random block trees $\mathcal T_1 \oplus \mathcal T_2$ and $\mathcal T_1 \ominus \mathcal T_2$, where $\mathcal T_1, \mathcal T_2$ are Catalan trees with $m$ nodes. Since there are $C_m$ choices for each parent or child tree in the block model, and two merging operators $\oplus/\ominus$, this class consists of exactly $2 C_m^2$ binary trees inside the class of $C_{2m}$ binary trees with $2m$ nodes. This is an asymptotically smaller subclass: by Stirling's approximation,
\begin{equation}\nonumber
\frac{2C_m^2}{C_{2m}} \sim  4 \sqrt{\frac2\pi} \cdot \frac1{m^{3/2}} = o_m(1).
\end{equation}

Our next goal for \Cref{thm:block} is to determine the support of the HS number for such models, which we show matches that of standard Catalan trees with $2m$ nodes. This follows immediately from \Cref{prop: merge increment}:

\begin{corollary}\label{prop: block support}
Let $\mathcal T_1,\mathcal T_2$ be binary trees each with $m$ nodes. Then 
\begin{equation}
	\supp(\hs(\mathcal T_1 \oplus \mathcal T_2)) =\supp(\hs(\mathcal T_1 \ominus \mathcal T_2))  = \{0,1,\ldots,\floor{\log_2(2m+1)}\}.
\end{equation}
\end{corollary}

\begin{proof}
The right-hand side is the full support of the HS number for binary trees with $2m$ nodes. Thus it suffices to show each value in this range can be achieved by a composite tree $\mathcal T_1 \oplus \mathcal T_2$.

For any $t \in \{0,1,\ldots,\floor{\log_2(m+1)}\}$, which is the full support of the HS number for binary trees with $m$ nodes, let $\mathcal T_2$ satisfy $\hs(\mathcal T_2)=t$. If $t=0$, take $\mathcal T_1=\mathcal T_2$ to be a path, so that $\mathcal T_1 \oplus \mathcal T_2$ is again a path and has HS number $0$. For $t>0$, let $\mathcal T_1$ be this same path; by \Cref{prop: merge increment},
\[
\hs(\mathcal T_1 \oplus \mathcal T_2) = \hs(\mathcal T_2) = t.
\]

Next note
\begin{equation}
\lceil \log_2(m+1)\rceil = \floor{\log_2(2m+1)}.
\end{equation}
Indeed, if $j = \lceil \log_2(m + 1)\rceil$, then $2^{j-1} < m + 1 \le 2^j$. Multiplying by 2 gives $2^j < 2m + 2 \le 2^{j+1}$, hence $2^j \le 2m+1 < 2^{j+1}$, so $j = \floor{\log_2(2m+1)}$.

If $m = 2^k - 1$ for a positive integer $k$, then $\floor{\log_2(m+1)} = k = \lceil \log_2 (m+1)\rceil = \floor{\log_2(2m+1)}$. If $2^k-1 < m \le 2^{k+1} - 2$, then $\floor{\log_2(m+1)} = k$ is the maximal HS number for a binary tree with $m$ nodes. Let $\mathcal T_1$ be the tree formed by a perfect binary tree of height $k$ with an up-right path of length $m - (2^k-1) \ge 1$ extending upward from the root of the embedded perfect tree. Each node on this path has HS number $k$. Taking $\mathcal T_1=\mathcal T_2$, the tree $\mathcal T_1 \oplus \mathcal T_2$ satisfies the increment rules in \Cref{prop: merge increment}: the root has a maximal HS path that escapes the top gluing edge since the top left edge of $\mathcal T_1$ contains a maximal HS path from the root to the embedded perfect tree. It follows
\[
\hs(\mathcal T_1 \oplus \mathcal T_2)
= \max(\hs(\mathcal T_1),\hs(\mathcal T_2)) + 1
= k + 1
= \lceil \log_2(m+1)\rceil
= \floor{\log_2(2m+1)}.
\]
\end{proof}

Although the support of the HS number for merged binary trees matches that of binary trees with $2m$ nodes, the gluing operator limits the ability of composites to attain larger HS number since any perfect binary tree using nodes from both components must use the gluing edge. Future work may determine exact probabilities for the HS number of merged trees. This is tractable via adaptations of Flajolet’s generating function approach, potentially yielding a full distributional description of the register function \cite{flajolet1977}.

Moreover, a direct consequence also of \Cref{prop: merge increment} for uniform composite Catalan trees is a WLLN result:
\begin{proposition}[WLLN]\label{prop:thm1b}
    Let $\mathcal T_1,\mathcal T_2$ be uniform Catalan trees with $m$ nodes. Then
        \[
    \frac{\hs(\mathcal T_1 \oplus \mathcal T_2)}{\log_2 (2m)} \xrightarrow[m \to \infty]{\mathbb P} \frac12.
    \]
\end{proposition} 
\begin{proof}
Note first
\[
\log_2(2m) = \log_2 m + 1 = \log_2 m \cdot (1 + o_m(1)),
\]
so that by the WLLN for Catalan trees with $m$ nodes and the Continuous mapping theorem, we have
\[
\frac{\max(\hs(\mathcal T_1), \hs(\mathcal T_2))}{\log_2(2m)} = \max\left(\frac{\hs(\mathcal T_1)}{\log_2 m}, \frac{\hs(\mathcal T_1)}{\log_2 m} \right) \cdot \frac{\log_2 m}{\log_2(2m)} \xrightarrow[m\to\infty]{\mathbb P} \max\left(\frac12,\frac12\right) \cdot 1 = \frac12.
\]
Since $I := \hs(\mathcal T_1 \oplus \mathcal T_2) - \max(\hs(\mathcal T_1),\hs(\mathcal T_2)) \in \{0,1\}$, then $I/\log_2(2m) \xrightarrow[m\to\infty]{} 0$ surely, so that by Slutsky's theorem, we have
\[
\frac{\hs(\mathcal T_1 \oplus \mathcal T_2)}{\log_2(2m)} = \frac{\max(\hs(\mathcal T_1), \hs(\mathcal T_2))}{\log_2(2m)} + \frac{I}{\log_2(2m)} \xrightarrow[m \to \infty]{\mathbb P} \frac12.
\]
\end{proof}

We now have

\begin{proof}[Proof of \Cref{thm:block}]
    Combine \Cref{prop: block support} and \Cref{prop:thm1b}.
\end{proof}

So although one gluing operator sufficed to increase the height of the composite tree, it did not change the first order scaling for the HS number. That said, it does hinder the ability of such merged trees to have larger HS number, as increasing HS number from the input trees must include the one gluing edge. We will spend the remainder of the paper on the natural question: 
\begin{quote}
    \textit{What happens to the HS number for binary trees formed using only these gluing operations? }
\end{quote}

This shifts the focus now to butterfly trees. 

\section{Butterfly trees}\label{sec: butterfly}

Butterfly trees are trees with $N = 2^n$ nodes built recursively using only $\oplus$ or $\ominus$ merging operators on two level $n-1$ butterfly trees, while simple butterfly trees are formed by taking identical simple butterfly trees at each previous level. Equivalently, butterfly trees are BSTs formed using butterfly permutations $\B_n$, with simple butterfly trees corresponding to BSTs built using simple butterfly permutations $\B_{n,s}$. By a straightforward induction, each permutation determines a unique butterfly tree shape. Hence, butterfly trees bridge permutation-weighted BSTs and Catalan trees. Moreover, it follows there are precisely $|\B_n| = 2^{N-1}$ butterfly trees and $|\B_{n,s}| = N$ simple butterfly trees. See \Cref{fig:Simple butterfly trees} for example simple butterfly trees, which we further note can fill out a rectangular lattice, while \Cref{fig: BST HS number keys} displays a general butterfly tree.

\begin{figure}[t]
    \centering
    \begin{subfigure}{0.32\textwidth}
    \includegraphics[width=\linewidth]{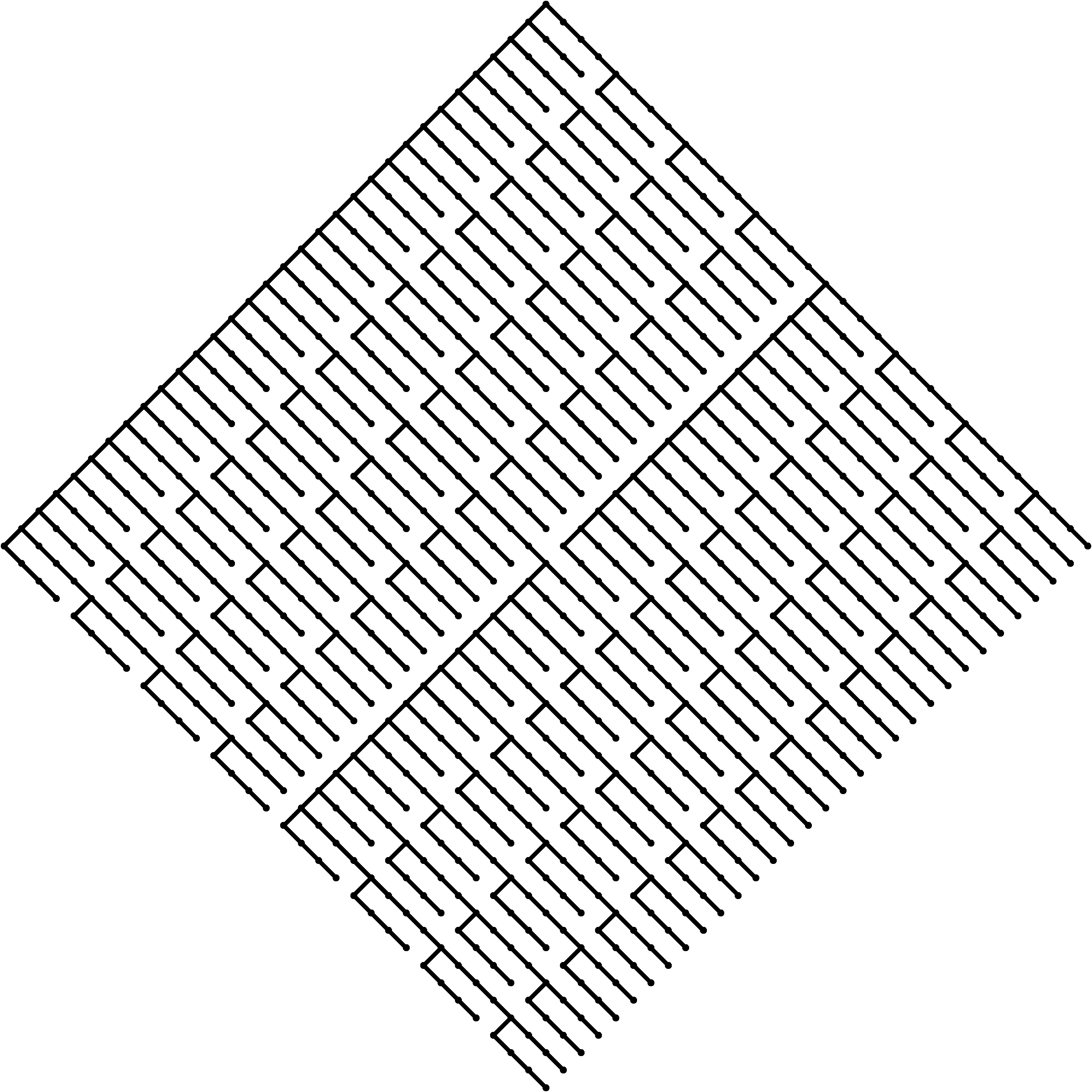}
    \end{subfigure}
    \begin{subfigure}{0.32\textwidth}
    \includegraphics[width=\linewidth]{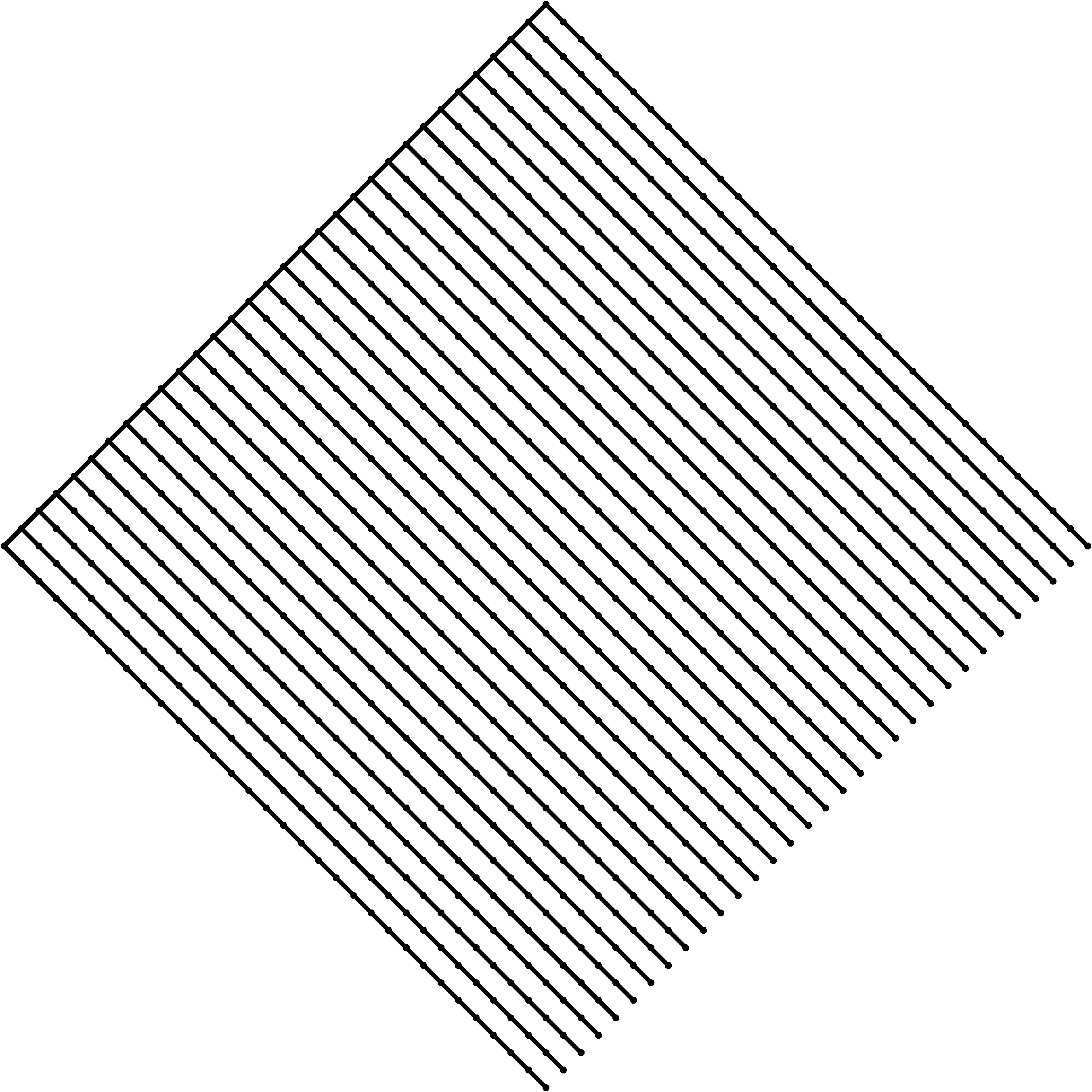}
    \end{subfigure}
    \begin{subfigure}{0.32\textwidth}
    \includegraphics[width=\linewidth]{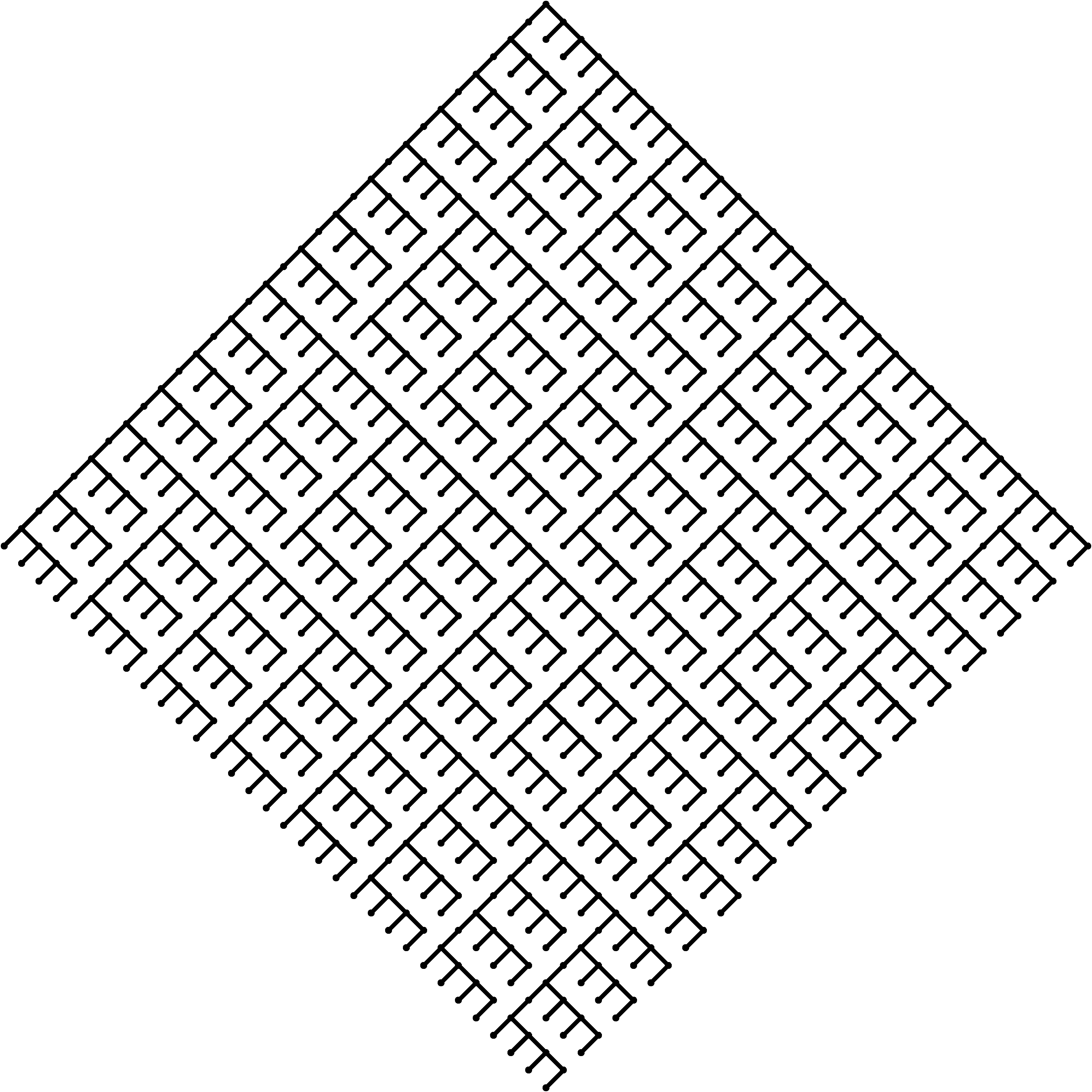}
    \end{subfigure}
    \caption{Simple butterfly trees with minimal height, $N = 1{,}024$ nodes}
    \label{fig:Simple butterfly trees}
\end{figure}

In \cite{PZ25}, we introduced butterfly trees and studied their height, showing it is order $N^{.585}$ on average (compared to $N^{1/2}$ for Catalan trees). The height and HS number are complementary statistics for binary trees: height is minimized by perfect binary trees and maximized by the identity permutation tree, whereas the HS number reverses these extremal relations. So while the height of butterfly trees increased over Catalan trees, we want to study and compare the complementary question for the HS number of butterfly trees.

We first consider the HS number support of butterfly trees. With a single gluing operation, the support remains unchanged (see Theorem \ref{thm:block}), but the block construction already limits larger HS values since perfect trees spanning both inputs must pass through the gluing edge. We now study the effect of building binary trees using only these operators. It turns out the support decreases substantially:

\begin{proposition}\label{prop: butterfly support}
    Let $\mathcal T_n^{\B}$ be a butterfly tree with $N = 2^n$ nodes. Then $$\supp(\hs(\mathcal T_n^{\B})) = \{0,1,\ldots,\floor{n/2}\}.$$
\end{proposition}

\begin{proof}
    We use induction on $n$. Note we will prove a further result that a butterfly tree with maximal HS number can only have a multiplicity of nodes with this maximal HS number if $4 \nmid N$. This is trivial for $n  \le 1$ since $\hs(\mathcal T_n^{\B}) = 0 = \floor{n/2}$, noting $\mathcal T_1^{\B}$ is a path and so both nodes have HS number 0. Moreover, for $n = 2$, then $\hs(\mathcal T_2^{\B}) = 0$ if $\mathcal T_2^{\B} = \mathcal T(\pi)$ for $\pi \in \B_2 = \{1234,4321,1243,4312\}$ and $\hs(\mathcal T_2^{\B}) = 1 = \floor{n/2}$ for $\pi \in \{2134,2143,3421,3412\}$, and this maximal HS number is achieved only by the root node within the entire tree.
    
    Now assume the result holds for $n \ge 2$. It suffices to then consider $\mathcal T_{n+1}^{\B} = \mathcal T_1^{\B} \oplus \mathcal T_2^{\B}$ where $\mathcal T_1^{\B}$ has maximal HS number, i.e., $\hs(\mathcal T_1^{\B}) = \floor{n/2}$. We will consider cases on the parity of $n$.\medskip
    
    \textbf{Case 1:} Suppose $n$ is even. Writing $n = 2k$, then $\floor{n/2} = k$. Since $\hs(\mathcal T_1^{\B}) = k$ and the previous level $n-1$ could have maximal HS number $\floor{n/4} = \floor{k - 1/2} = k-1$, then necessarily the root of $\mathcal T_1^{\B}$ is the only node in $\mathcal T_n^{\B}$ with HS number $k$. Since then gluing $ \mathcal T_2^{\B}$ to the end of a top edge of $\mathcal T_1^{\B}$ can at most extend a path from the root of $\mathcal T_2^{\B}$ to the composite root, which corresponds the root of $\mathcal T_1^{\B}$, and this path cannot then meet a path connecting nodes of HS number $k$ from the opposite top edge from the gluing (since only the previous root had this maximal HS number), then it follows $\hs(\mathcal T_1^{\B} \oplus \mathcal T_2^{\B}) = k = \floor{k + 1/2} = \floor{(n+1)/2}$. (Note then the composite would then have a multiplicity of nodes with max HS number $k$ along the gluing edge if the child butterfly tree was chosen to also have a maximal HS number.) \medskip
    
    \textbf{Case 2:} Now suppose $n = 2k+1$ is odd. A similar construction can now consider $\mathcal T_1^{\B}$ that has a multiplicity of nodes with HS number $\floor{n/2} = \floor{k + \frac12} = k$. We can then choose the gluing operation so that the gluing occurs on the opposite side of the parent edge with at least one additional node with HS number $k$, and if we further choose the child butterfly tree to have HS number $k$, then the merging rules then create a path from the child root to the composite root connecting nodes with HS number $k$.  The merged tree thus meets the increment merging sufficiency rules in \Cref{prop: merge increment}, and so the composite tree has HS number $k + 1 = \floor{(n+1)/2}$. \medskip
    
    This completes the inductive step.
\end{proof}

\begin{remark}
While Catalan trees have HS number that concentrates sharply at $\frac12 \log_2 N = n/2$, this value comprises a deterministic upper bound for the HS number of butterfly trees. This suggests the average HS number is smaller than for Catalan trees, so butterfly trees have higher parallelization efficiency.
\end{remark}

To model the HS number for butterfly trees, we use \Cref{prop: merge increment}, which controls when an HS increment occurs from $\max(\mathcal T_n^{\B},\widetilde{\mathcal T}_n^{\B})$ under a single merging operator $\oplus/\ominus$. This requires tracking the HS profiles along the top edges of each butterfly tree to determine when increment conditions are satisfied. This yields a recursive form
\begin{equation}\label{eq: rde}
    \hs(\mathcal T_{n+1}^{\B}) = \max(\hs(\mathcal T_n^{\B}), \hs(\widetilde{\mathcal T}_n^{\B})) + I_n,
\end{equation}
where $I_n = I(\mathcal T_n^{\B},\widetilde{\mathcal T}_n^{\B},t)$ is the indicator that the increment conditions from \Cref{prop: merge increment} are satisfied for $\mathcal T_n^{\B}, \widetilde{\mathcal T}_n^{\B}$, which then build $\mathcal T_{n+1}^{\B}$ using $\oplus$ if $t = 0$ or $\ominus$ if $t = 1$. Equivalently, by \Cref{prop: profile update}, $I_n$ indicates whether the merged profiles induce a cascading increment to the root of the composite tree. As butterfly trees are built using only merging operators $\oplus$ and $\ominus$, we can then form a recursive function that returns the full HS profile of a butterfly tree given an input $N-1$ bitstring $\V x$ that encodes the sequence of $\oplus/\ominus$ merging operators, where writing $\V x = (\V y,\V z, t) \in \{0,1\}^{N-1}$ for $\V y,\V z \in \{0,1\}^{N/2-1}$ and bit $t$
\[
\mathcal T^{\B}(\V x) = \begin{cases}
    \mathcal T^{\B}(\V y) \oplus \mathcal T^{\B}(\V z), & t = 0,\\ \vspace{-.5pc} \\ 
    \mathcal T^{\B}(\V y) \ominus \mathcal T^{\B}(\V z), & t = 1.
\end{cases}. 
\]
This is outlined explicitly in \Cref{alg:hs_profile_corrected} in \Cref{sec: algorithms}. 

\begin{remark}[Recursive tree process (RTP)]\label{rmk: RTP}
    An equivalent probabilistic interpretation is obtained by viewing butterfly trees as functionals of an infinite rooted binary tree. Consider an infinite binary tree in which each node is assigned an independent $\mathrm{Bern}(1/2)$ label. For each depth $n$, we restrict to the finite subtree consisting of the first $n$ levels. The $(N-1)$-bit string $\mathbf{x}$ encoding a level-$n$ butterfly tree can then be viewed as the collection of labels along this truncated tree, read in the order induced by the recursive merging construction.
\end{remark}

\begin{remark}
We implement \Cref{alg:hs_profile_corrected} in MATLAB, which computes the HS number in $\mathcal{O}(N)$ time by operating directly on the edge-profile representation, without constructing the underlying tree. This is asymptotically optimal in $N$ for this model.

For comparison, a na\"ive implementation first constructs the butterfly permutation from the input bitstring, then builds the associated binary search tree, and finally computes the HS number via a standard postorder traversal. While both permutation generation and HS evaluation are $\mathcal{O}(N)$, the BST construction incurs an additional cost proportional to the insertion depth. For butterfly trees it is known that
\[
\mathbb{E}[\mathrm{height}(\mathcal{T}_n^{\B})] \ge 2(3/2)^n - 2 \approx \Omega(N^{0.585}),
\]
see~\cite{PZ25}. Thus the na\"ive pipeline has expected runtime at least of order $N \cdot N^{0.585} = N^{1.585}$ when BST construction dominates.

In contrast, \Cref{alg:hs_profile_corrected} avoids explicit tree construction and eliminates this superlinear overhead. In initial experiments, the na\"ive method required approximately 2.7 hours to compute 10 HS values for random butterfly trees with $N = 2^{18}$. Using the profile-based algorithm, we compute 500 HS values for the same size in approximately 6.5 minutes, and extend computations up to $n = 25$ (see \Cref{tab: summary statistics}).
\end{remark}

\begin{remark}\label{rmk: reflection invariance}
Using the input merge bitstring $\mathbf{x}$, invariance of the HS number under reflection implies
\[
\hs(\mathcal T^{\B}(\mathbf{x})) = \hs(\mathcal T^{\B}(\mathbf{1}-\mathbf{x})),
\]
since $\mathcal T^{\B}(\mathbf{x})^R = \mathcal T^{\B}(\mathbf{1}-\mathbf{x})$. Here reflection swaps each $\oplus$ and $\ominus$ merging operator, which corresponds exactly to the transformation $\mathbf{x} \mapsto \mathbf{1}-\mathbf{x}$.
\end{remark}

The appearance of the maximum in \Cref{eq: rde} produces a max-type recursion for the HS number of random butterfly trees. While vast literature exist for max-type recursive distributional equations (see \cite{rde_aldous} for a survey), tools for exact analysis often require stronger additional conditions, such as a smoothing transform if the additive factor $I_n$ were an independent noise term, which is not the case here. The intricate dependence on the local geometry of each input tree so obstructs this direction, which requires the full HS profile of each input tree, and forms an unbounded profile state space that grows with $n$. As such, the first moment no longer has a clear scaling form. We will revisit this direction for general butterfly trees in \Cref{sec: nonsimple}, where we approach the WLLN in \Cref{conj: wlln butterfly}.

In contrast, this maximum is no longer present for simple butterfly trees, built from identical copies at each level, the dynamics simplify substantially. The maximum disappears since it is evaluated on identical inputs, and by \Cref{prop: profile update}, the gluing edge only depends on the top-level profile since $h_c = h_p$. Thus the cascade condition reduces to tracking only the multiplicity indicators along the top edges together with the HS number of the input tree.

This allows a full distributional description, which we now outline.

\subsection{Simple butterfly trees}\label{sec: simple}

Simple butterfly trees are formed using gluing operations on \emph{identical} copies of the previous level tree, producing a fractal structure. In particular, simple butterfly trees fill out a rectangular lattice, where the height is achieved by the node at the bottom corner of this lattice \cite{PZ24} (see \Cref{fig:Simple butterfly trees}). Our previous work established a full distributional description for the height of uniformly random simple butterfly trees:
\begin{equation}\label{eq: heights}
    \texttt{height}(\mathcal T_n^{\B}) \stackrel{d}{=}2^{Y_n} + 2^{n-Y_n} - 2
\end{equation}
for $Y_n \sim \operatorname{Binom}(n,1/2)$. Other statistics follow directly from this lattice structure, including maximal width and width profiles at fixed depth. We now extend these ideas to the HS number.

A key tool is a compressed butterfly permutation representation of the simple butterfly permutation $\pi_n \in \B_{n,s}$. Since
\[
\pi_n = \bigotimes_{j = 1}^n (1\ 2)^{x_j}, \quad x_j \in \{0,1\},
\]
we can encode $\pi_n$ by the bitstring $\mathbf{x} \in \{0,1\}^n$. This representation is sufficient to construct the full permutation while allowing certain statistics to be computed directly from the bitstring itself, without first constructing $\pi_n$. 

\begin{remark}
    The $(N-1)$-bitstring $\mathbf{x}'$ encoding the merging operators for a general butterfly tree can be reduced to an $n$-bitstring $\mathbf{x}$ in the simple butterfly setting. To show this, writing $\mathbf{x}' = (\mathbf{y}', \mathbf{z}', t)$ recursively, we have $\mathbf{y}' = \mathbf{z}'$ at each step, so only the final merging bit needs to be retained. This yields the condensed representation $\mathbf{x}$ with
    \[
    x_j = x'_{2^j - 1}.
    \]
\end{remark}

\begin{figure}[t]
    \centering
    \begin{subfigure}{0.23\textwidth}
    \includegraphics[width=\linewidth]{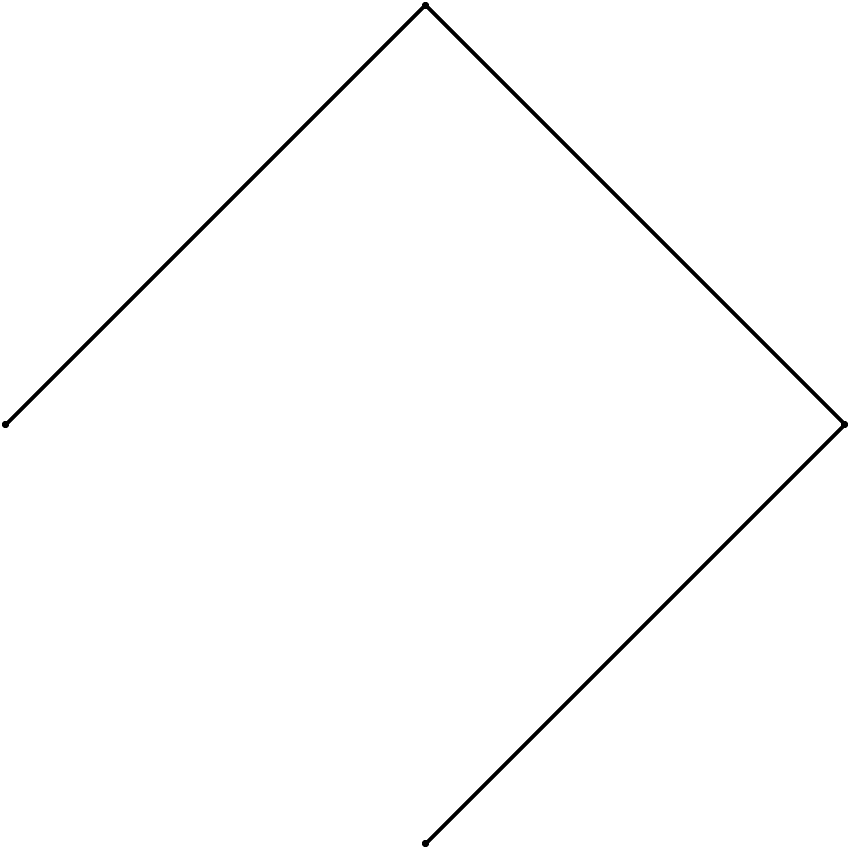}
    \end{subfigure}
    \begin{subfigure}{0.23\textwidth}
    \includegraphics[width=\linewidth]{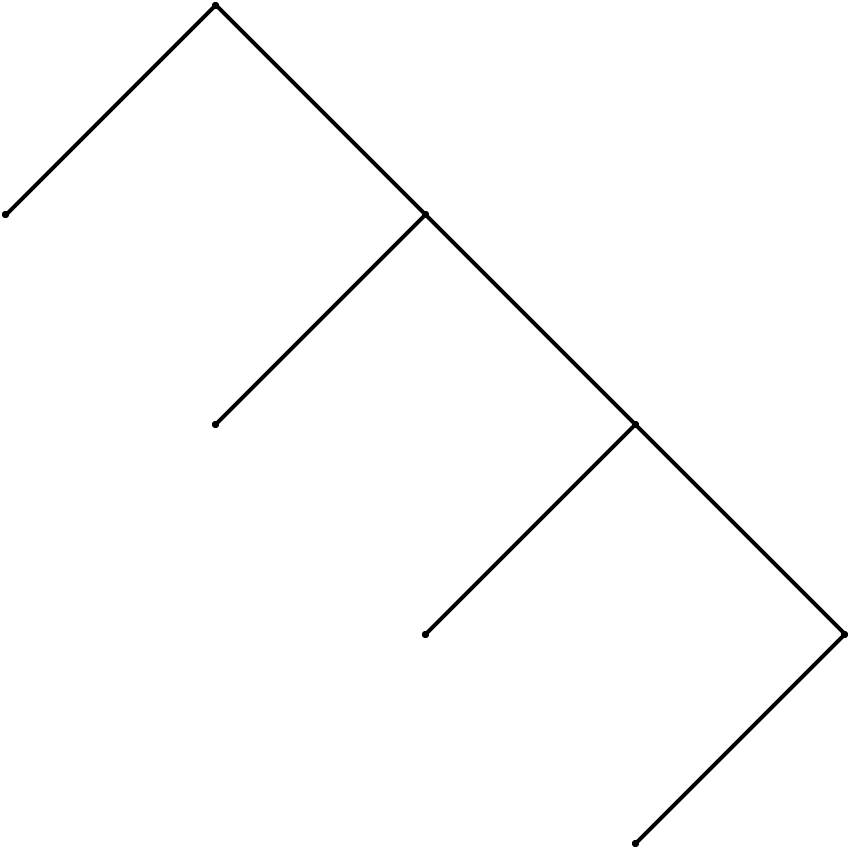}
    \end{subfigure}
    \begin{subfigure}{0.23\textwidth}
    \includegraphics[width=\linewidth]{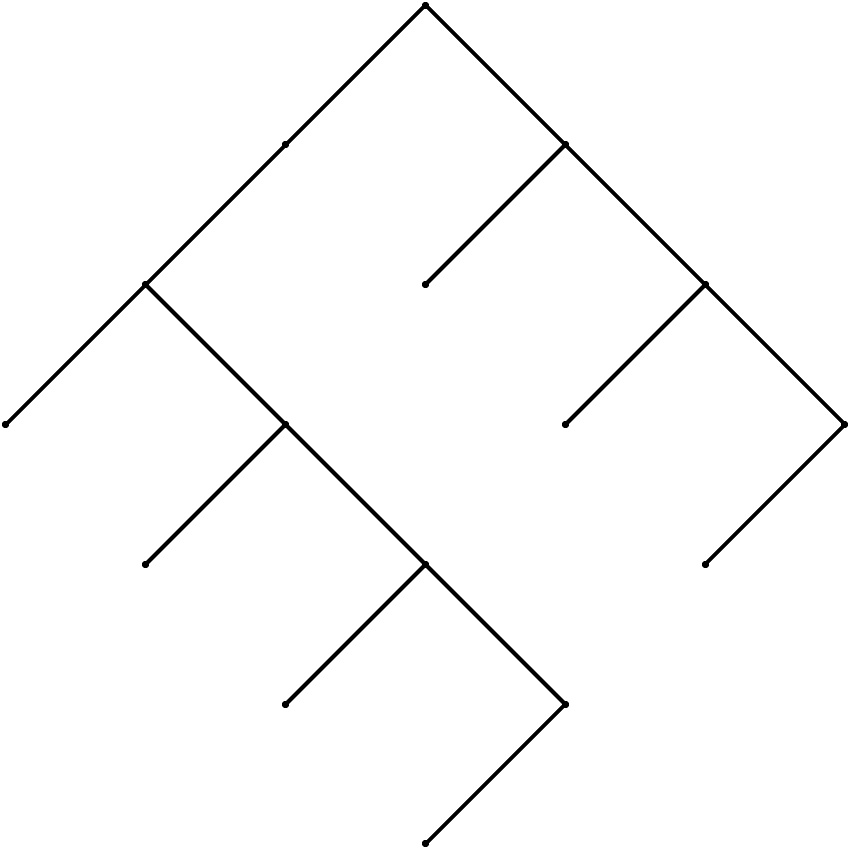}
    \end{subfigure}
    \begin{subfigure}{0.23\textwidth}
    \includegraphics[width=\linewidth]{images/MaxHS_bst_1001100101_6.png}
    \end{subfigure}
    \caption{Intermediate simple butterfly trees $\mathcal T^{\B}(10), \mathcal T^{\B}(100), \mathcal T^{\B}(1001)$ in building $\mathcal T^{\B}(1001100101)$}
    \label{fig:Simple butterfly trees intermediate}
\end{figure}

\begin{remark}[RTP: Simple butterfly trees]
    In the probabilistic interpretation of \Cref{rmk: RTP}, butterfly trees can be viewed as functionals of a rooted infinite binary tree equipped with Bernoulli labels.
    
    In the general butterfly tree setting, each node carries an independent label, so the level-$n$ construction depends on $N-1$ independent Bernoulli variables, corresponding to the full $(N-1)$-bit encoding.
    
    In the simple butterfly tree setting, this structure simplifies to a level-indexed model. Thus only one random bit is introduced per level, and the level-$n$ construction is determined by an $n$-bit string rather than an $(N-1)$-bit string.
    
    Equivalently, the simple butterfly tree is obtained by restricting the labeled infinite binary tree to depth $n$, where randomness enters only through one Bernoulli variable per level, shared across all nodes at that depth.
\end{remark}

The simple butterfly tree heights are a function of $Y_n  = \sum_{j=1}^n x_j$. If $x_j \sim \operatorname{Bern}(p)$ iid, then $Y_n \sim \Bin(n,p)$, as seen in \eqref{eq: heights} for $p = 1/2$. This similarly derives the \textit{lengths of the longest increasing sequence (LIS)} and \textit{longest decreasing sequence (LDS)} \cite{PZ24}, where $\operatorname{LIS}(\pi_n) = 2^{n - Y_n}$ and $\operatorname{LDS}(\pi_n) = 2^{Y_n}$. In the corresponding BST, the LIS and LDS are each realized as the sequence of BST labels starting at the root along the the top-right and top-left edges.

In particular, we note the construction of the simple butterfly trees themselves are direct functions of the compressed butterfly permutation $n$-bitstrings $\V x$, which we can write as $\mathcal T^{\B}(\V x)$, where each bit indicates whether to glue two identical previous levels together by extending the top right (0) or top left (1) edge of the parent copy. Figure \ref{fig:Simple butterfly trees intermediate} shows three intermediate forms of simple butterfly trees toward the right figure in Figure \ref{fig:Simple butterfly trees}, $\mathcal T^{\B}(1001100101)$, which includes other simple butterfly trees $\mathcal T^{\B}(0010011101)$ and  $\mathcal T^{\B}(0000011111)$.

\begin{remark}
    Similarly, general butterfly permutations of length $N = 2^n$ can be encoded using a compressed butterfly permutation now as $(N-1)$-bitstring. This aligns with an encoding from a complementary full binary rooted tree, with components corresponding to the local block rotations at each coordinate within the tree. This idea was implicitly important in the techniques in \cite{Abert_Virag_2005} by Ab\'ert and Vir\'ag, who resolve a long standing open question from Tur\'an on the order of a random element from a $p$-Sylow subgroup of $\operatorname{Sym}(p^n)$ using this correspondence. (Recall the butterfly permutations comprise a 2-Sylow subgroup of separable permutations of $\operatorname{Sym}(2^n)$.) 
    
    Since non-identical butterfly trees can be used to build the next level tree, we note general butterfly trees result in non-rectangular lattices: see Figure \ref{fig: BST HS number keys} for the nonsimple butterfly tree with $2^5=32$ nodes, $\mathcal T^{\B}(1011111001100101010010010110100)$.
\end{remark}

Using the compressed bitstring for simple butterfly trees, we can construct significantly more efficient means to construct the HS number that for general butterfly trees (see \Cref{alg:hs_profile_corrected}). 

\subsubsection{Computing the HS number of simple butterfly trees}\label{sec: computing}

We can rewrite the RDE for butterfly trees now as 
\[
\hs(\mathcal T_{n+1}^{\B}) = \hs(\mathcal T_n^{\B}) + I(\mathcal T_n^{\B},\mathcal T_n^{\B},t),
\]
since $\mathcal T_p = \mathcal T_c$. It thus suffices to model the increment
\[
X_n = I(\mathcal T_{n-1}^{\B},\mathcal T_{n-1}^{\B},t) = \hs(\mathcal T_n^{\B}) - \hs(\mathcal T_{n-1}^{\B}),
\]
since then 
\[
\hs(\mathcal T_n^{\B}) = \sum_{j=1}^n X_j.
\]
We note the increment $X_n = 1$ occurs if and only if a maximal HS number escape path exists along the gluing edge. Again, since $\mathcal T_p = \mathcal T_c$ in the simple butterfly setting (with HS profiles $(h_p,L_p,R_p) = (h_c,L_c,R_c)$), so conditions (i) and (ii) of \Cref{prop: merge increment} always hold, and only the escape condition (iii) remains.

Using \Cref{prop: profile update}, since $h_c = h_p$, we have
\[
k^\star = \min\{k \ge h_c : (a_k^{(G_p)}, b_k^{(G_p)}) \ne (1,1)\} \in \{h_c, h_c+1\},
\]
where $(G,O) = (R,L)$ for $\oplus$ and $(G,O) = (L,R)$ for $\ominus$. Thus an increment occurs if and only if $b_{h_c}^{(G_p)} = 1$. By symmetry of the inputs, this holds precisely when a maximal HS path exists on the opposite top edge, i.e., when the multiplicity indicator satisfies $m^{(O_p)} = 1$. Hence the full profile reduces to tracking multiplicity on the top edges.

In this setting, an escape path exists exactly when:
\begin{itemize}
    \item $\textnormal{\texttt{xor}}(x_n,x_{n-1}) = 1$, so if a multiplicity of nodes is on a top edge then it is on the opposite edge from the current gluing operator, and
    \item no increment occurred at the previous step, so we need $X_{n-1} = 0$, since an increment produces a unique maximal node at the root and resets the multiplicity.
\end{itemize}
As such, we further note $X_j X_{j-1} = 0$ for all $j$.

This yields the following recursive representation.

\begin{proposition}[Recursive representation]\label{prop: s HS number increment}
Let $\mathcal T_n^{\B} = \mathcal T^{\B}(\mathbf{x})$ be a simple butterfly tree with $\mathbf{x} \in \{0,1\}^n$. Then
\[
\hs(\mathcal T_n^{\B}) = \sum_{j=1}^n X_j,
\]
where $X_1 = 0$ and
\[
X_j = \textnormal{\texttt{xor}}(x_j,x_{j-1}) \cdot (1 - X_{j-1}).
\]
\end{proposition}

This recursive formulation admits a direct combinatorial interpretation. Define a new sequence $\mathbf{y} \in \{0,1\}^{n-1}$ by
\[
y_j = \textnormal{\texttt{xor}}(x_{j+1},x_j).
\]
Thus $y_j = 1$ records exactly those steps at which a bit flip occurs, i.e., those indices where an increment is {eligible}. However, not every such index produces an increment. From the recursion,
\[
X_j = y_{j-1}\cdot (1 - X_{j-1}),
\]
so an increment can occur at $j$ only if $y_{j-1} = 1$ and the previous step did not already produce an increment. In particular, consecutive increments are forbidden, i.e., $X_j X_{j-1} = 0$ for all $j$.

Now consider a maximal run of $1$'s in $\mathbf{y}$ of length $k$. This corresponds to $k$ consecutive indices at which increments are eligible. Since no two consecutive indices can both contribute (as $X_j X_{j-1}=0$), the maximum number of increments that can be realized within such a run is obtained by selecting every other index. This yields exactly $\left\lceil k/2 \right\rceil$ increments from a run of length $k$.

Since distinct runs are separated by zeros in $\mathbf{y}$, the constraint resets between runs, and the contributions add independently. Hence $\hs(\mathcal T^{\B}(\mathbf{x}))$ is obtained by summing $\lceil k/2 \rceil$ over all maximal runs of $1$'s in $\mathbf{y}$.

\begin{proposition}[Run-length representation]\label{prop: runs}
If the maximal runs of $1$'s in $\mathbf{y}$ have lengths number $k_1,\dots,k_\ell$, then
\[
\hs(\mathcal T^{\B}(\mathbf{x})) = \sum_{j=1}^\ell \left\lceil \frac{k_j}{2} \right\rceil.
\]
\end{proposition}

\Cref{alg:HS number_run_length} outlines a specific implementation path for the run-length representation. A direct implementation of this method was used to construct all HS numbers of simple butterfly tree samples used throughout this paper. An immediate consequence of these representations is that the HS number can be computed in linear time in $n$.

\begin{remark}\label{rmk: runtime}
Both representations above yield $\mathcal O(n)$ algorithms for computing the HS number for simple butterfly trees $\mathcal T_n^{\B} = \mathcal T^{\B}(\mathbf{x})$ for $\V x \in \{0,1,\}^n$. The recursive form requires a single pass updating $X_j$, while the run-length representation requires forming $\mathbf{y}$ and summing over its maximal runs. This is an exponential improvement over the general butterfly tree construction using \Cref{alg:hs_profile_corrected} with $N = 2^n$ nodes, which requires $\mathcal O(N)$ operations.
\end{remark}

\begin{remark}\label{rmk: invariance}
The HS number is preserved under two natural symmetries of the input string
$\mathbf{x} \in \{0,1\}^n$, illustrated in Figure~\ref{fig:Simple butterfly
trees invariance}.
\begin{itemize}
  \item \textbf{Complement invariance.} $\hs(\mathcal T^{\B}(\mathbf{x})) = \hs(\mathcal
  T^{\B}(\mathbf{1}-\mathbf{x}))$, since complementing $\mathbf{x}$
  interchanges each $\oplus/\ominus$ merge, corresponding to reflecting
  $\mathcal T^{\B}$ across the root that preserves the HS number. This always holds for butterfly trees, as noted in \Cref{rmk: reflection invariance}.
  
  \item \textbf{Reversal invariance.} The map $\textnormal{\texttt{rev}}(\mathbf{x})\colon
  x_j \mapsto x_{n-j+1}$ preserves the multiset of run-lengths number
  $(k_1,\dots,k_\ell)$ from \Cref{prop: runs} since
  $$
  y_j = \textnormal{\texttt{xor}}(x_{j+1},x_j) = \textnormal{\texttt{xor}}(1-x_{j+1},1-x_j)
  $$
  and hence the HS number. Equivalently, reversal
  corresponds to constructing the simple butterfly tree in the opposite order. 
  
  In contrast, the HS for general butterfly trees built using an $(N-1)$-bitstring are not invariant under reversal of the input bitstring, as the input parent and child trees are distinct so a reversal can completely change the geometry of the merged tree. For example, the 8 node butterfly tree $\mathcal T^{\B}(0010000) = \mathcal T^{\B}(001) \oplus \mathcal T^{\B}(000)$ has HS number 1 while $\mathcal T^{\B}(0000100) = \mathcal T^{\B}(000) \oplus \mathcal T^{\B}(010)$ has HS number 0 (see also \Cref{fig: bst_4_nodes}).
\end{itemize}
\end{remark}

\begin{figure}[t]
    \centering
    \begin{subfigure}{0.32\textwidth}
    \includegraphics[width=\linewidth]{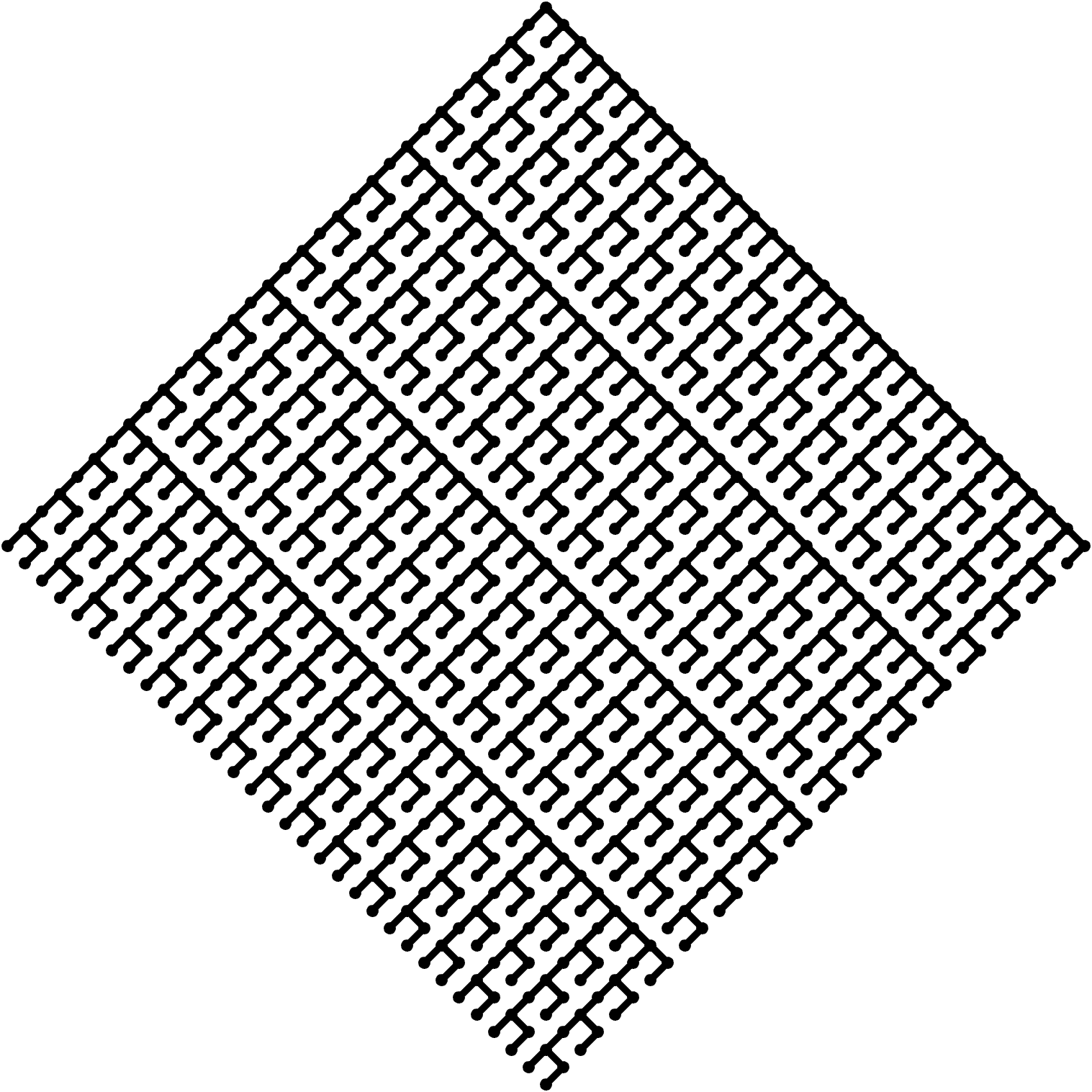}
    \end{subfigure}
    \begin{subfigure}{0.32\textwidth}
    \includegraphics[width=\linewidth]{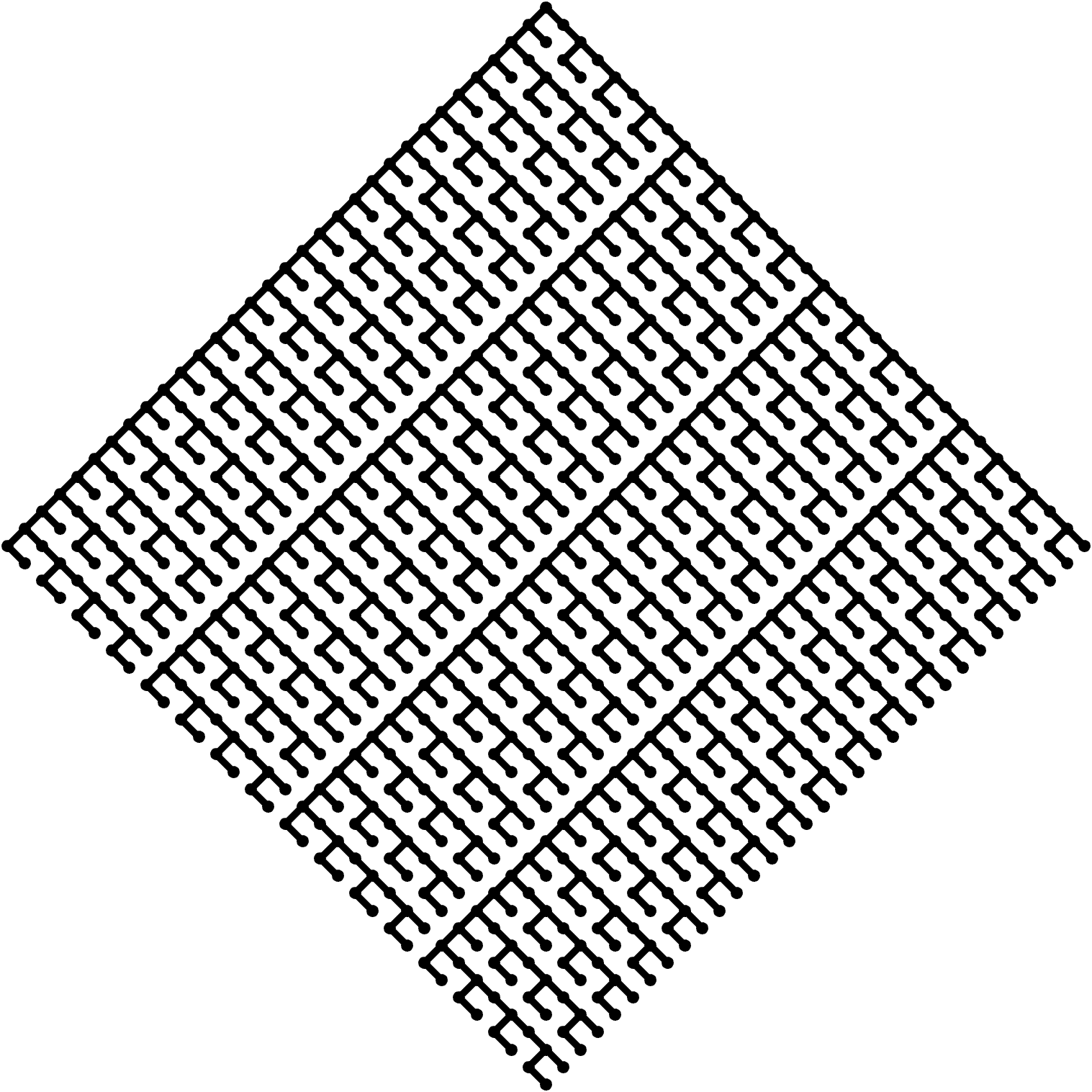}
    \end{subfigure}
    \begin{subfigure}{0.32\textwidth}
    \includegraphics[width=\linewidth]{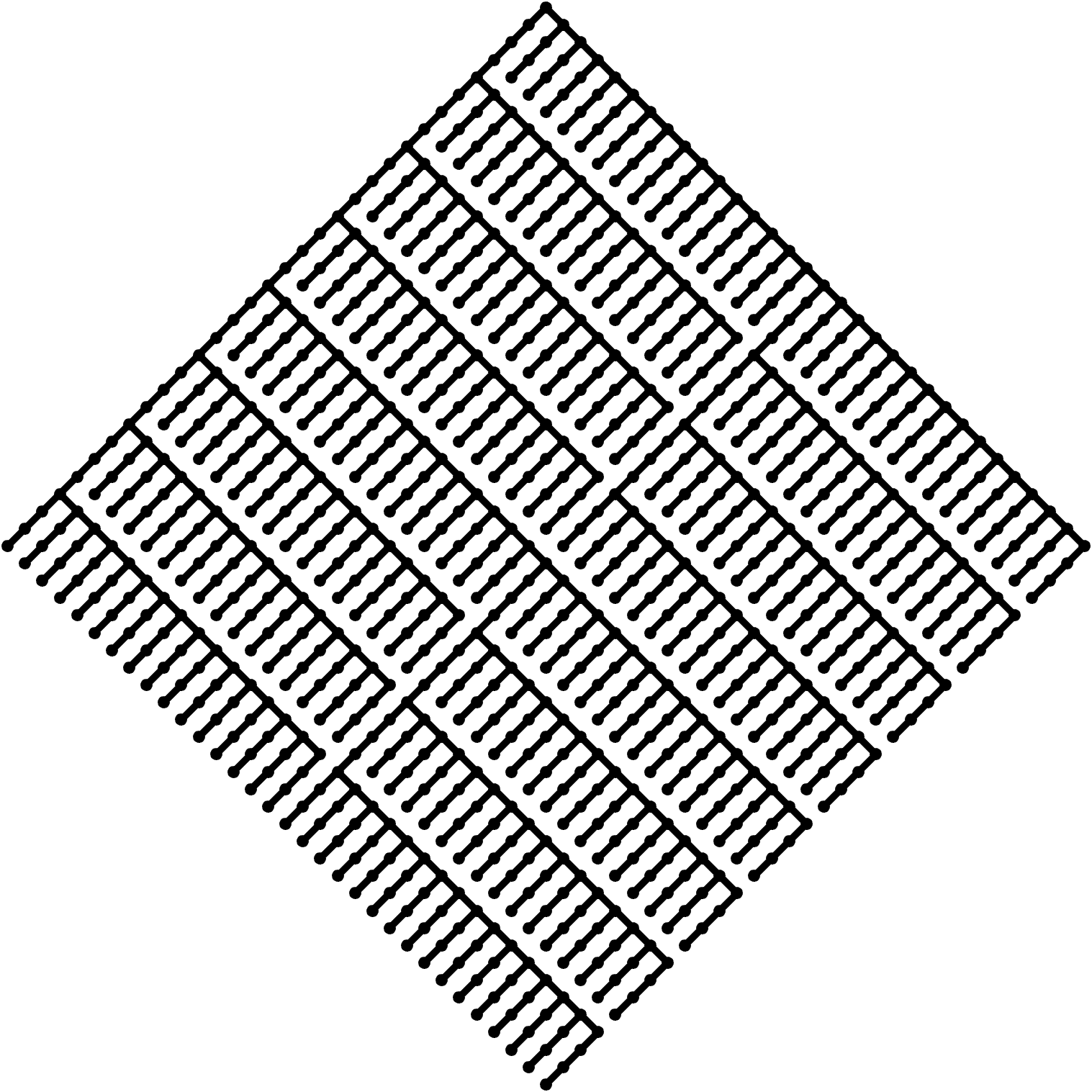}
    \end{subfigure}
    \caption{Simple butterfly tree $\mathcal T^{\B}(\V x)$ with $\V x = 1011000011$ and $\hs(\mathcal T^{\B}(\V x)) = 3$, along with the HS number invariant butterfly trees formed using the inputs $\boldsymbol{1} - \V x=0100111100$ and $\texttt{rev}(\V x)=110000110{1}$ }
    \label{fig:Simple butterfly trees invariance}
\end{figure}

\subsubsection{Random simple butterfly trees}\label{sec: random simple}

We now consider the case of random simple butterfly trees $\mathcal T_n^{\B} = \mathcal T(\V x)$, which are generated using an iid sequence $x_j \sim \operatorname{Bern}(p)$ with $\V x \in \{0,1\}^n$, where $p = 1/2$ corresponds to uniform simple butterfly trees.

Our first goal is to establish simple butterfly trees can be exactly modeled as an additive functional of a Markov process, as seen in Theorem~\ref{thm:HS number_MC_model}. The remaining limit theorems all follow directly from this model.

\begin{proposition}\label{prop: MC}
For $p \in (0,1)$, let $(x_j)_{j \ge 1}$ be iid $\Bern(p)$, and set $X_0 = 0$. Then the process $(M_j)$ with
\[
M_j \;=\; (x_j,\,x_{j-1},\,X_{j-1}) \in \{0,1\}^3
\]
is an irreducible, aperiodic, $8$-state time-homogeneous Markov chain, with stationary distribution \begin{equation}\label{eq: stationary}
\boldsymbol{\pi} =
\frac{1}{1 - pq}
\left[
q^3,\ 
p^2q^2,\ 
p^2q,\ 
pq^3,\ 
pq^2,\ 
p^3q,\ 
p^3,\ 
p^2q^2
\right].
\end{equation}
\end{proposition}

\begin{proof}
By construction and Proposition~\ref{prop: s HS number increment}, the update at step $j$ depends only on the previous state $M_{j-1}$ together with the new bit $x_j$. Writing
\[
M_{j-1} = (a,b,c),
\]
where $a = x_{j-1}$, $b = x_{j-2}$, and $c = X_{j-1}$, the next state is obtained by shifting in the new bit and updating the increment:
\[
M_j = (\xi, a, f(a,b,c)), \qquad \xi = x_j \sim \Bern(p),
\]
where
\[
f(a,b,c) = \textnormal{\texttt{xor}}(a,b)\,(1-c) = X_j.
\]
Thus $M_j$ is determined entirely by $M_{j-1}$ and the new bit $\xi$, so the Markov property holds.

Labeling states $(a,b,c)$ by $k = 4a + 2b + c + 1$, the transition matrix is  
\[
P = 
\begin{bmatrix}
    q & 0 & 0& 0 & p&0&0&0\\
    q & 0 & 0 & 0& p&0&0&0\\
    0&q & 0 & 0&0 & p&0&0\\
    q & 0 & 0 & 0& p&0&0&0\\
    0&0&0&q & 0 & 0 & 0&p\\
    0&0&q & 0 & 0&0 & p&0\\
    0&0&q & 0 & 0&0 & p&0\\
    0&0&q & 0 & 0&0 & p&0
\end{bmatrix}, \quad q = 1-p.
\]

The corresponding transition digraph is  

\begin{center}
\begin{tikzpicture}[ ->, >=Stealth, node distance=2.5cm and 1.8cm,
every node/.style={draw, circle, minimum size=0.7cm, font=\small}, shorten >=2pt ] 
\node (000) at (0,0) {000}; \node (001) [right=of 000] {001}; \node (010) [right=of 001] {010}; \node (101) [right=of 010] {101}; 
\node (011) at (0,-2.5) {011}; \node (100) [right=of 011] {100}; \node (111) [right=of 100] {111}; \node (110) [right=of 111] {110}; 
\path (000) edge[loop left] (); \path (000) edge (100); \path (001) edge (000); \path (001) edge (100); 
\path (010) edge (001); \path[->, bend left=10] (010) edge (101); \path[->, bend left=10] (101) edge (010); \path (101) edge (110); 
\path (110) edge[loop right] (); \path (110) edge (010); \path (111) edge (010); \path (111) edge (110); \path (100) edge (111); 
\path[->, bend left=10] (100) edge (011); \path[->, bend left=10] (011) edge (100); \path (011) edge (000);
\end{tikzpicture} 
\end{center}

The digraph is strongly connected: for example, the cycles  
\[
000 \to 100 \to 011, \qquad 100 \to 111 \to 010 \to 001, \qquad 010 \to 101 \to 110
\]  
are interconnected via states $100$ and $010$, so every state can be reached from any other. Hence, the chain is irreducible. 

Finally, since the chain is irreducible and state $000$ has a 
self-loop with probability $q > 0$ for any $p \in (0,1)$, all 
states are aperiodic. It follows $(M_j)$ is ergodic, so $P$ has a unique stationary 
distribution $\boldsymbol{\pi}$ such that $\boldsymbol{\pi} P = \boldsymbol{\pi}$ and $\sum_i \pi_i = 1$. A direct computation yields $\boldsymbol{\pi}$ is of the form given in \eqref{eq: stationary}.
\end{proof}

We can now fully establish \Cref{thm:HS number_MC_model}.
\begin{proof}[Proof of \Cref{thm:HS number_MC_model}]
    Combine Propositions~\ref{prop: s HS number increment} and \ref{prop: MC}. 
\end{proof}

We will use again the bounded functional $f:\{0,1\}^3 \to \{0,1\}$ defined by
\[
f(a,b,c) = \texttt{xor}(a,b)\,(1-c).
\]
Then, by \Cref{thm:HS number_MC_model},
\[
X_j = f(x_j, x_{j-1}, X_{j-1}) = f(M_j),
\]
and hence
\[
\hs(\mathcal T_n^{\B}) \;=\; \sum_{j=1}^n X_j \;=\; \sum_{j=1}^n f(M_j).
\]
In particular, $\hs(\mathcal T_n^{\B})$ is an additive functional of the Markov chain $(M_j)$.



\begin{figure}[t]
    \centering
    \includegraphics[width=0.8\linewidth]{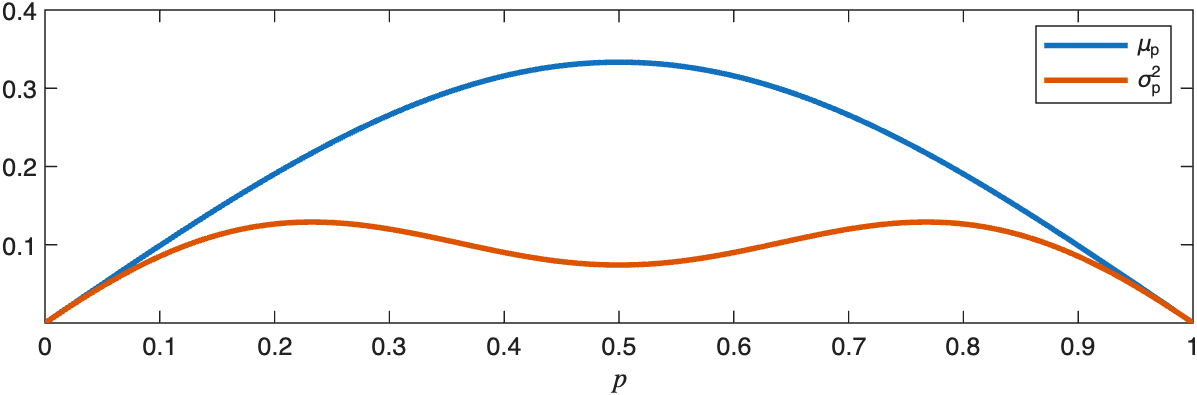}
    \caption{Plot of $\mu_p,\sigma_p^2$}
    \label{fig: HS number mean var}
\end{figure}

\Cref{thm:HS number_MC_model} is the main tool used for the following very strong limit theorems, that differentiates the HS number of random butterfly trees from all previously studied random tree models. For instance, we note that this Markov chain modeling enables a SLLN result for the HS number, in contrast to at best WLLN results for other random tree models.


\begin{proof}[Proof of Theorem~\ref{thm:SLLN} (SLLN)]
Since $f$ is bounded and $(M_j)_{j \ge 1}$ is an irreducible, aperiodic Markov chain on a finite state space with stationary distribution $\boldsymbol{\pi}$, the Ergodic theorem applies:
\[
\frac{1}{n} \hs(\mathcal T_n^{\B})=\frac{1}{n}\sum_{j=1}^n f(M_j) \;\xrightarrow[n\to\infty]{\text{a.s.}}\; \E_{\boldsymbol{\pi}}[f].
\]
A direct calculation yields
\[
\mu_p := \E_{\boldsymbol{\pi}}[f] = \sum_{s\in\{0,1\}^3} f(s) \pi(s) = f(0,1,0)\frac{p^2q}{1 - pq} + f(1,0,0)\frac{pq^2}{1-pq}= \frac{pq}{1 - pq}.
\]
\end{proof}

The WLLN (\Cref{cor:wlln}) is a direct consequence of the SLLN (\Cref{thm:SLLN}).

\begin{remark}
\Cref{fig: HS number mean var} shows a plot of $\mu_p$, which is maximized at $\mu_{1/2} = 1/3$. Recalling that $n = \log_2 N$, the SLLN can be rewritten as
\[
\frac{\hs(\mathcal T_n^{\B})}{\log_{2} N} \;\xrightarrow[n\to\infty]{\text{a.s.}}\; \mu_p \le \mu_{1/2}=\frac13.
\]
This contrasts directly with the WLLN result for Catalan trees, that has the HS number scaled by $\log_2 N$ limit to $1/2$.
\end{remark}

Unlike in the classical random tree models, this Markov modeling enables now additionally second order fluctuations. For additive functionals of finite, irreducible, aperiodic Markov chains, the asymptotic variance can be characterized through the {Poisson equation}
\begin{equation}\label{eq: Poisson eqn}
(\V I - P) g = h,
\end{equation}
where $P$ is the transition matrix and $ h =  f - \mu_p \mathbf{1}$ is the centered functional, using again the bounded functional $f$. (See \cite{meyn2009markov} or \cite{jones2004markov} for further discussion.) Since $h$ is centered under $\boldsymbol{\pi}$, the Poisson equation admits a solution (unique up to additive constants). Although $\V I - P$ is singular (since $P\V 1 = \V 1$), the modified system
\[
(\V I - P + \mathbf{1}\boldsymbol\pi) g = h
\]
is nonsingular and admits a unique solution $g$ orthogonal to $\boldsymbol\pi$. A direct computation obtains explicitly $g = (\mathbf{I} - P + \mathbf{1}\boldsymbol\pi)^{-1} h$ where
\begin{equation}\label{eq: Poisson soln}
g =
\frac{1}{8(1 - pq)} 
\left[
8p - 5,\ 
8p - 5,\ 
3,\ 
8p - 5,\ 
3,\ 
3 - 8p,\ 
3 - 8p,\ 
3 - 8p
\right]^\top.
\end{equation}
Then the variance constant
\[
\sigma_p^2 = \lim_{n\to\infty} \frac1n \Var\!\left[\sum_{j=1}^n X_j\right]
\]
exists and is given by
\[
\sigma_p^2 = \Var_{\boldsymbol\pi}[f] + 2 \sum_{k=1}^\infty \operatorname{Cov}(f(M_1), f(M_k)) 
= \E_{\boldsymbol\pi}\!\left[ (h + 2Pg)\odot h \right],
\]
where $\odot$ denotes the componentwise Hadamard product. In closed form, a direct computation yields
\begin{equation}\label{eq: asym var closed}
\sigma_p^2 = \frac{pq \left(1 - 3pq - 2p^2q^2\right)}{(1 - pq)^3}.
\end{equation}

\begin{remark}
    Both $\mu_p$ and $\sigma_p^2$ are symmetric in $p$ and $q$, and depend only on $r := pq$. Thus $\hs(\mathcal T^{\B}(\V x)) \stackrel{d}{=} \hs(\mathcal T^{\B}(\V y))$ if $x_j \sim \Bern(p)$ and $y_j \sim \Bern(q)$, each iid, which corresponds to reflection of the butterfly tree about its root, that switches each $\oplus/\ominus$ merging rule but preserves the HS number (see \Cref{rmk: invariance}). Moreover, the variance-to-mean ratio
    \[
    g(r) := \frac{\sigma_p^2}{\mu_p} = \frac{1 - 3r - 2r^2}{(1 - r)^2}
    \]
    is strictly decreasing on $r \in [0,\frac14]$, ranging from $g(0)=1$ to $g(1/4)=2/9$, which yields additionally $0 \le \sigma_p^2 \le \mu_p$ (see Figure \ref{fig: HS number mean var}). The variance is maximized at $p^* = \frac{1+\sqrt{13}}{6} \approx 0.7676$ and its symmetric counterpart $1-p^*$, while $\sigma^2_{1/2} = 2/27$.
\end{remark}

With this Markov process representation for the HS number of simple butterfly trees, together with the asymptotic mean and variance computations $\mu_p$ and $\sigma_p^2$, we obtain not only a standard CLT for the fluctuations (\Cref{cor:CLT}), but also a \textit{functional CLT} (Theorem~\ref{thm:FCLT}).

\begin{proof}[Proof of Theorem~\ref{thm:FCLT} (Functional CLT)]
By \Cref{thm:HS number_MC_model}, the quantity $\hs(\mathcal T_n^{\B})$ can be expressed as an additive functional of a finite Markov chain. The result therefore follows directly from the functional CLT for Markov chains \cite[Theorem 17.4]{meyn2009markov}, since the Poisson equation admits a solution and the asymptotic variance $\sigma_p^2$ is finite.
\end{proof}

\begin{figure}[t]
    \centering
    \includegraphics[width=0.8\linewidth]{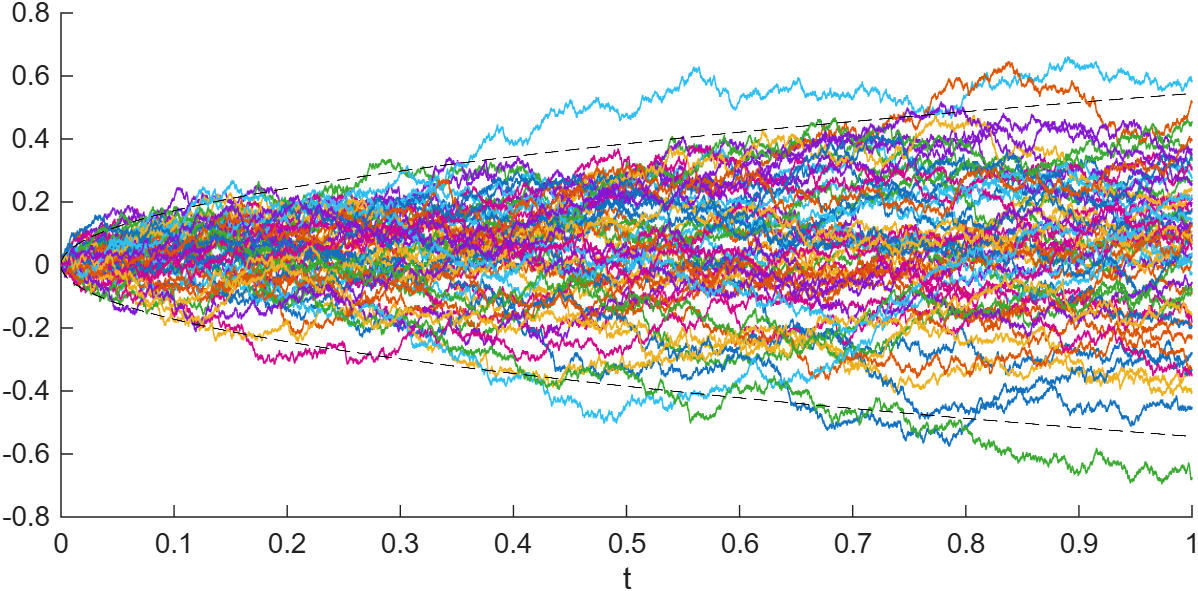}
    \caption{60 sampled paths of $W_n(\frac12,t)$ with $n = 10{,}000$, shown alongside the 95\% confidence envelope bounded by $\pm 2\sigma_{1/2} \sqrt{t} = \pm  \sqrt{\frac{8}{27} t}$ over $t \in [0,1]$.}
    \label{fig:BM limit}
\end{figure}

Figure \ref{fig:BM limit} shows 60 sampled paths of \[
    W_n(p,t)
=
\sqrt n
\left(
\frac1n\sum_{j=1}^{\lfloor nt\rfloor}X_j
-
t\mu_p
\right),
\]
corresponding to $\hs(\mathcal T_n^{\B})$ for iid uniform simple butterfly trees with $n = 10{,}000$, so that $N \approx 1.995 \cdot 10^{3{,}010}$; this again utilized the efficient $\mathcal O(n)$ sampling method for the HS number of simple butterfly trees that would otherwise remain prohibitive for standard HS number algorithms as well as our direct HS number for general butterfly trees algorithm \Cref{alg:hs_profile_corrected} (see also Remark \ref{rmk: runtime}).

\subsubsection{Uniform simple butterfly trees}\label{sec: unif simple}

Uniform simple butterfly trees correspond to the case $p = 1/2$. In this
setting, the structure of the HS number admits a particularly clean description.

Recall from \Cref{prop: runs} that $\hs(\mathcal T^{\B}(\mathbf{x}))$ is
determined by the run-lengths number of the sequence
\[
y_j = \textnormal{\texttt{xor}}(x_{j+1},x_j),
\]
with each maximal run of $1$'s of length $k$ contributing $\lceil k/2 \rceil$.
Since $\mathbf{x}$ is uniform, the sequence $\mathbf{y}$ consists of iid\
$\mathrm{Bern}(1/2)$ variables, and the HS number reduces to counting nonconsecutive
successes in $\mathbf{y}$. Moreover, the Markov chain representation of the
increment process compresses from the 8-state model for general $p$: since
$x_j \sim \mathrm{Bern}(1/2)$ are independent, the variables
$\textnormal{\texttt{xor}}(x_j,x_{j-1})\sim \mathrm{Bern}(1/2)$ are additionally independent of $X_{j-1}$ (these are dependent for $p \ne 1/2$). Hence, the law of $X_j$ only depends on the law of $X_{j-1}$. Hence,
$(X_j)_{j \ge 1}$ itself now comprises a two-state time-homogeneous Markov chain with
transition matrix
\begin{center}
\begin{minipage}{0.4\textwidth}
\centering
$P =
\begin{bmatrix}
1/2 & 1/2 \\
1 & 0
\end{bmatrix}$
\end{minipage}
\hfill
\begin{minipage}{0.5\textwidth}
\centering
\begin{tikzpicture}[->, >=Stealth, node distance=3cm, thick]
\node[circle, draw, minimum size=0.7cm, font=\small] (0) {$0$};
\node[circle, draw, minimum size=0.7cm, font=\small, right of=0] (1) {$1$};
\path[->, bend left=15] (0) edge node[above] {} (1);
\path[->, bend left=15] (1) edge node[below] {} (0);
\path[->] (0) edge[loop left] node[left] {} ();
\end{tikzpicture}
\end{minipage}
\end{center}
and stationary distribution $\boldsymbol{\pi} = [2/3, \ 1/3]$.

This structure yields a closed form for the distribution of $\hs(\mathcal
T_n^{\B})$. For the uniform case, this follows from a direct combinatorial argument.

\begin{proposition}\label{prop: pmf}
Let $\mathcal T_n^{\B}$ be a uniform simple butterfly tree with $N = 2^n$ nodes.
Then
\begin{equation}\label{eq: pmf}
    \P\left(\hs(\mathcal T_n^{\B}) = k\right)
= 2^{k-n}\left[\binom{n-k}{k} + \binom{n-k-1}{k}\right],
\quad k = 0,1,\ldots,\lfloor n/2 \rfloor.
\end{equation}
\end{proposition}

\begin{proof}
By \Cref{prop: runs}, computing $\hs$ from $\mathbf{x}
\in \{0,1\}^n$ reduces to counting nonconsecutive bit flips. Each configuration
with exactly $k$ flips uses $2k$ bits for the flip pairs (each with two
orientations, contributing $2^k$) and distributes the remaining $n-2k$ bits
into gaps. Conditioning on whether the last flip occurs strictly before position
$n$ or at position $n$, a stars-and-bars argument gives $\binom{n-k-1}{k}$ and
$\binom{n-k-1}{k-1}$ placements respectively, with an additional factor of $2$
for trailing bits in the first case. Combining via Pascal's identity yields
\[
2^k \cdot \left[2\binom{n-k-1}{k} + \binom{n-k-1}{k-1}\right]
= 2^k \left[\binom{n-k}{k} + \binom{n-k-1}{k}\right]
\]
configurations, and dividing by $2^n$ completes the proof.
\end{proof}

In contrast, for $p \ne 1/2$ the dependence structure in $\mathbf{y}$ prevents
such a closed form; the distribution can instead be expressed in terms of gap
probabilities between successive flips, which we will not pursue at present.

\begin{remark}
The probability generating function for the HS number of uniform simple butterfly trees, $G_n(x) = \E[x^{\hs(\mathcal T_n^{\B})}]$,
admits a recursive structure relating successive merging rules, which can be written as a function of the generating function for the number of simple butterfly trees with a multiplicity of nodes with maximal HS number, $F_n(x)$, where
\[
G_n(x) = 1 + \frac12(x-1) \sum_{k=0}^{n-1} 2^{-k}F_k(x), \qquad F_{n+1}(x) = 2x F_{n-1}(x) + F_n(x), \quad F_0(x) = 0.
\]
This recursive structure implies
that each successive $G_n(x)$ has interlacing negative real roots (which can be established directly
via tools from \cite{liu2007unified}), as illustrated in
Figure~\ref{fig: pgf zeros}, which plots the zeros of $G_n(x)$ in $[0,20]$
for $n \le 42$.\footnote{Double-precision floating-point errors in MATLAB
produced spurious complex roots for $n \ge 43$.} The interlacing of negative
zeros yields an alternative proof of the CLT (\Cref{cor:CLT}), in the
flavor of Harper's original CLT proof for the Stirling distribution
\cite{Harper1967}.
\end{remark}

\begin{figure}[t]
    \centering
    \includegraphics[width=0.7\linewidth]{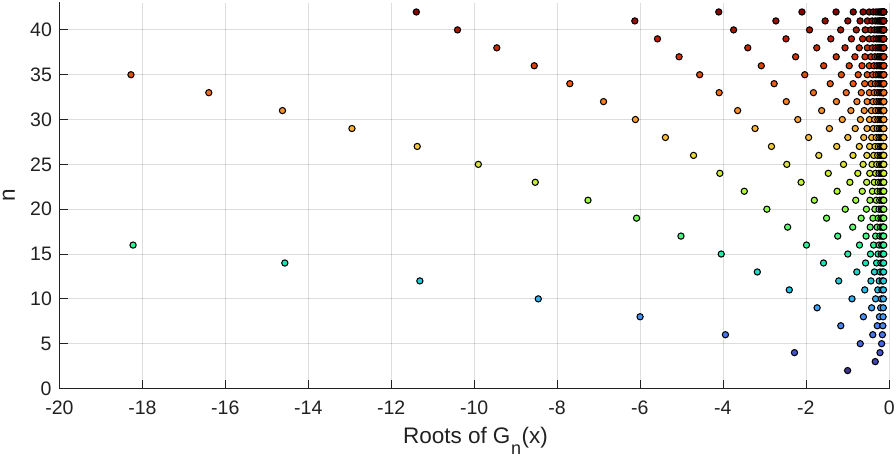}
    \caption{Interlacing roots of $G_n(x)$ for successive levels}
    \label{fig: pgf zeros}
\end{figure}

Either using the explicit probability mass function \eqref{eq: pmf}, the transition matrix, or the probability generating function, one can compute exact
expressions for the mean and variance:
\begin{align*}
\E[\hs(\mathcal T_n^{\B})] &= \frac{n}3 - \frac29 + \frac29
  \left(-\frac12\right)^n = \frac13 \log_2 N +\mathcal O(1), \\
\Var(\hs(\mathcal T_n^{\B})) &= \frac2{27} n + \frac2{81} + \frac4{27} n
  \left(-\frac12\right)^n + \frac{2}{81} \left(-\frac12\right)^n
  - \frac4{81} \left(\frac14\right)^n = \frac2{27} \log_2 N +\mathcal O(1).
\end{align*}
In particular, we note the mean and variance are of the same order (which is also true for $p \ne 1/2$). This further allows
alternative direct proofs of the WLLN via Chebyshev's inequality and CLT via the Quasi-Power Theorem
\cite{Hwang1998convergence}.

\subsection{General Butterfly Trees}\label{sec: nonsimple}

We now return to the general butterfly tree model. We again state the recursive construction yields 
\begin{equation}\label{eq: rde2}
    \hs(\mathcal T_{n+1}^{\B}) 
= \max\!\left(\hs(\mathcal T_n^{\B}), \hs(\widetilde{\mathcal T}_n^{\B})\right) + I_n,
\end{equation}
where $I_n = I(\mathcal T_n^{\B},\widetilde{\mathcal T}_n^{\B},t) \in \{0,1\}$ indicates whether the increment conditions of \Cref{prop: merge increment} are satisfied, with $\oplus$ used if $t=0$ and $\ominus$ if $t=1$. Thus the HS number evolves via a max-plus recursion with a bounded, local correction determined by the merge interface.

The key difference from the simple butterfly setting is that $I_n$ is no longer determined by the root values alone. Instead, it depends on the full edge profile $\texttt{profile}(\mathcal T_n^{\B})$, which records the configuration of maximal HS paths along the gluing edge. Consequently, while the evolution remains Markovian at the level of profiles, the state space grows with $n$ and does not collapse to a fixed finite set. In particular, the presence of the maximum in the RDE prevents a reduction to an additive functional.

In the simple model, identical inputs force the profile to collapse to a single root parameter, yielding a finite-state Markov chain and enabling a complete analysis in \Cref{sec: simple}. In contrast, for general butterfly trees, the parent and child may have substantially different profiles, and the increment depends on their full configuration. As a result, $\hs(\mathcal T_n^{\B})$ cannot be described solely in terms of its current value (if defining an evolving process by tracking the profile of the parent trees at each recursive step), and the finite-state framework no longer applies.


\begin{figure}[t]
\centering

\begin{subfigure}[b]{0.22\textwidth}
\centering
\begin{tikzpicture}[scale=0.8]
\node[circle,draw] (1) at (0,0) {0};
\node[circle,draw] (2) at (1,-1) {0};
\node[circle,draw] (3) at (2,-2) {0};
\node[circle,draw] (4) at (3,-3) {0};
\draw (1)--(2)--(3)--(4);
\end{tikzpicture}
\caption{$\mathcal T^{\B}(000) = \mathcal T(1234)$ \\ 
\centering $\left(0;0, \begin{bmatrix}
    1\\1
\end{bmatrix}; 1,\begin{bmatrix}
    1\\0
\end{bmatrix}\right)$}
\end{subfigure}
\hfill
\begin{subfigure}[b]{0.22\textwidth}
\centering
\begin{tikzpicture}[scale=0.8]
\node[circle,draw] (3) at (0,0) {1};
\node[circle,draw] (1) at (-1,-1) {0};
\node[circle,draw] (2) at (0,-2) {0};
\node[circle,draw] (4) at (1,-1) {0};
\draw (3)--(1);
\draw (1)--(2);
\draw (3)--(4);
\end{tikzpicture}
\caption{$\mathcal T^{\B}(001) = \mathcal T(3412)$ \\ 
\centering $\left(1;0, \begin{bmatrix}
    1&1\\1&0
\end{bmatrix};0,\begin{bmatrix}
    1&1\\0&0
\end{bmatrix}\right)$}
\end{subfigure}
\hfill
\begin{subfigure}[b]{0.22\textwidth}
\centering
\begin{tikzpicture}[scale=0.8]
\node[circle,draw] (1) at (0,0) {0};
\node[circle,draw] (2) at (1,-1) {0};
\node[circle,draw] (4) at (2,-2) {0};
\node[circle,draw] (3) at (1,-3) {0};
\draw (1)--(2)--(4);
\draw (4)--(3);
\end{tikzpicture}
\caption{$\mathcal T^{\B}(010) = \mathcal T(1243)$ \\ 
\centering $\left(0;0, \begin{bmatrix}
    1\\1
\end{bmatrix};1,\begin{bmatrix}
    1\\1
\end{bmatrix}\right)$}
\end{subfigure}
\hfill
\begin{subfigure}[b]{0.22\textwidth}
\centering
\begin{tikzpicture}[scale=0.8]
\node[circle,draw] (3) at (0,0) {1};
\node[circle,draw] (4) at (1,-1) {0};
\node[circle,draw] (2) at (-1,-1) {0};
\node[circle,draw] (1) at (-2,-2) {0};
\draw (3)--(4);
\draw (3)--(2);
\draw (2)--(1);
\end{tikzpicture}
\caption{$\mathcal T^{\B}(011) = \mathcal T(3421)$ \\ 
\centering $\left(1;0, \begin{bmatrix}
    1&1\\0&0
\end{bmatrix};0,\begin{bmatrix}
    1&1\\0&0
\end{bmatrix}\right)$}
\end{subfigure}

\vspace{0.5cm}

\begin{subfigure}[b]{0.22\textwidth}
\centering
\begin{tikzpicture}[scale=0.8]
\node[circle,draw] (2) at (0,0) {1};
\node[circle,draw] (1) at (-1,-1) {0};
\node[circle,draw] (3) at (1,-1) {0};
\node[circle,draw] (4) at (2,-2) {0};
\draw (2)--(1);
\draw (2)--(3);
\draw (3)--(4);
\end{tikzpicture}
\caption{$\mathcal T^{\B}(100) = \mathcal T(2134)$ \\ 
\centering $\left(1;0, \begin{bmatrix}
    1&1\\0&0
\end{bmatrix};0,\begin{bmatrix}
    1&1\\0&0
\end{bmatrix}\right)$}
\end{subfigure}
\hfill
\begin{subfigure}[b]{0.22\textwidth}
\centering
\begin{tikzpicture}[scale=0.8]
\node[circle,draw] (4) at (0,0) {0};
\node[circle,draw] (3) at (-1,-1) {0};
\node[circle,draw] (2) at (-2,-2) {0};
\node[circle,draw] (1) at (-1,-3) {0};
\draw (4)--(3)--(2);
\draw (2)--(1);
\end{tikzpicture}
\caption{$\mathcal T^{\B}(101) = \mathcal T(4312)$ \\ 
\centering $\left(0;1, \begin{bmatrix}
    1\\1
\end{bmatrix};0,\begin{bmatrix}
    1\\1
\end{bmatrix}\right)$}
\end{subfigure}
\hfill
\begin{subfigure}[b]{0.22\textwidth}
\centering
\begin{tikzpicture}[scale=0.8]
\node[circle,draw] (2) at (0,0) {1};
\node[circle,draw] (1) at (-1,-1) {0};
\node[circle,draw] (3) at (0,-2) {0};
\node[circle,draw] (4) at (1,-1) {0};
\draw (4)--(2);
\draw (1)--(2);
\draw (3)--(4);
\end{tikzpicture}
\caption{$\mathcal T^{\B}(110) = \mathcal T(2143)$ \\ 
\centering $\left(1;0, \begin{bmatrix}
    1&1\\0&0
\end{bmatrix};0,\begin{bmatrix}
    1&1\\1&0
\end{bmatrix}\right)$}
\end{subfigure}
\hfill
\begin{subfigure}[b]{0.22\textwidth}
\centering
\begin{tikzpicture}[scale=0.8]
\node[circle,draw] (4) at (0,0) {0};
\node[circle,draw] (3) at (-1,-1) {0};
\node[circle,draw] (2) at (-2,-2) {0};
\node[circle,draw] (1) at (-3,-3) {0};
\draw (4)--(3)--(2)--(1);
\end{tikzpicture}
\caption{$\mathcal T^{\B}(111) = \mathcal T(4321)$ \\ 
\centering $\left(0;1, \begin{bmatrix}
    1\\0
\end{bmatrix};0,\begin{bmatrix}
    1\\1
\end{bmatrix}\right)$}
\end{subfigure}

\caption{Butterfly trees with 4 nodes, with HS number profiles and labels.}
\label{fig: bst_4_nodes}
\end{figure}

Although the state space does not admit a finite reduction, the HS number can still be computed efficiently from the input $(N-1)$-bitstring via the profile recursion (see \Cref{alg:hs_profile_corrected,alg:merge_final}), avoiding explicit construction of the BST. We begin by illustrating the full profile structure in small cases.

\medskip

\begin{example}[$n \le 2$]
For $n \le 1$, $\mathcal T_n^{\B}$ is either a single vertex or a path, so $\hs=0$. The tree $\mathcal T^{\B}(0)$ has edge profiles $L = [1,1]^\top$ and $R = [1,0]^\top$ with $m^{(L)} = 1$, while $\mathcal T^{\B}(1)$ swaps $L$ and $R$ by reflection.

For $n=2$, there are $2$ choices for each parent and child (from $n=1$) and $2$ choices of gluing operator, giving $8$ trees (which matches $|\B_2| = 2^{2^2-1} = 8$). \Cref{fig: bst_4_nodes} displays all trees with their HS profiles.

Profiles for $\mathcal T^{\B}(\mathbf{x})$ and $\mathcal T^{\B}(\mathbf{1}-\mathbf{x})$ are reversed, corresponding to reflection (e.g., $\mathbf{x}=001$ and $110$). Thus there are $7$ distinct profiles, as one pair ($011,101$) is symmetric. Simple butterfly trees correspond to repeating the first two bits (viz., $000,001,110,111$), and admit the compressed $n$-bit representation ($00,01,10,11$).

\end{example}

While \Cref{prop: profile update} provides deterministic rules for updating edge profiles under the merging operation, these rules can in principle be iterated to compute exact HS number distributions at fixed levels $n$. Indeed, one may propagate profiles through the recursion to obtain a complete enumeration of butterfly trees at each level, by only tracking a representative input bitstring per distinct profile along with a multiplicity counter. Still, the number of distinct profiles grows rapidly with $n$, and the associated bookkeeping becomes increasingly complex. Although this structure allows exact computations up to $n=8$ via a reduction to profile equivalence classes (see \Cref{sec: exact}), the state space growth makes further exact analysis impractical.

We therefore turn to the asymptotic regime, where the focus shifts from finite-level enumeration to the large-$n$ behavior of the HS number under random inputs, leading to the limiting distributional results and the WLLN conjecture developed in the next section.

\subsubsection{Random butterfly trees}
\label{sec: rand gen butterfly}

We now turn to uniformly random butterfly trees $\mathcal T^{\B}(\mathbf{x})$, obtained by taking $\mathbf{x} \in \{0,1\}^{N-1}$ with iid $x_j \sim \mathrm{Bern}(1/2)$. This probabilistic viewpoint leads naturally to the recursive distributional formulation in \eqref{eq: rde2}, where the HS number is expressed as a max-type functional of two independent subtrees together with a local increment term.

In this setting, the deterministic recursion in \eqref{eq: rde2} induces a max-type RDE. Writing $(h_p,L_p,R_p)$ and $(h_c,L_c,R_c)$ for the profiles of the parent and child subtrees, and letting $G$ denote the gluing side (i.e., $G = R$ if $t=0$ and $G = L$ if $t=1$), the increment event can be written as
\[
I_n 
= \mathds{1}\!\left( b_k^{(G_p)} = 1 \;\; \text{for all } k = h_c, \ldots, h_p \right).
\]

Under iid inputs and an independent choice of $t$, the indicator $I_n$ is conditionally Bernoulli given the profiles. Its mean is
\[
\beta_n := \mathbb{E}[I_n]
= \mathbb{E}\!\left[\hs(\mathcal T_{n+1}^{\B})\right]
- \mathbb{E}\!\left[\max\!\left(\hs(\mathcal T_n^{\B}),\hs(\widetilde{\mathcal T}_n^{\B})\right)\right].
\]
So $\beta_n$ represents the probability that a full cascade occurs along the gluing edge, i.e., that the increment condition propagates all the way to the root under the merging operation.

We first compute the exact distribution of $\hs(\mathcal T_n^{\B})$ for small $n$.

\subsubsection{Limiting behavior}
\label{sec: limiting dist}

\begin{figure}
    \centering
    \includegraphics[width=0.75\linewidth]{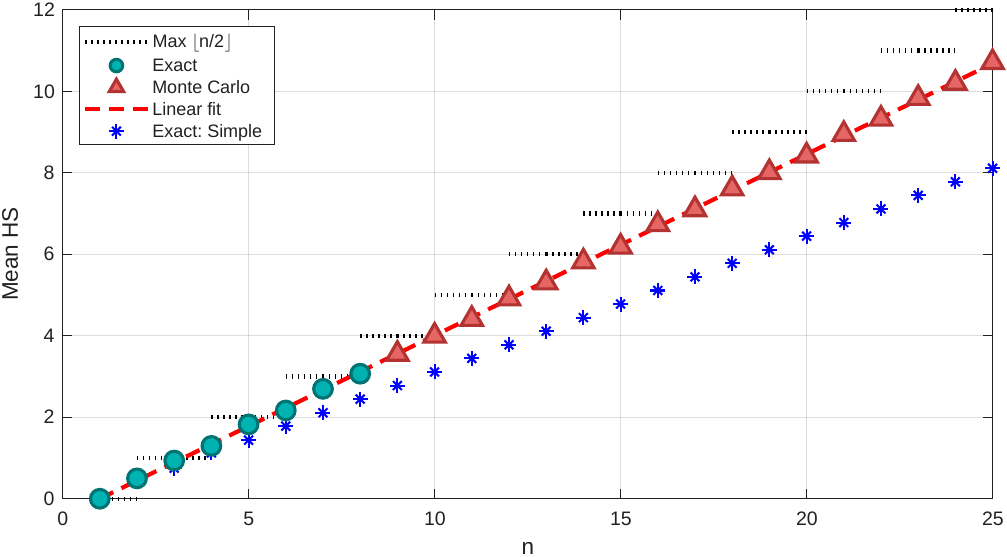}
    \caption{Mean of $\hs(\mathcal T_n^{\B})$ with linear regression fit, compared to the maximal HS number $\floor{n/2}$ for butterfly trees and exact average HS numbers of uniformly random simple butterfly trees.}
    \label{fig:mean_gen}
\end{figure}

As noted previously, the max-type RDE \eqref{eq: rde2} for random butterfly trees does not fall within existing frameworks that permit a complete analysis, in contrast to the simple butterfly setting. Our aim here is therefore more modest: to present supporting evidence for the weak law of large numbers stated in \Cref{conj: wlln butterfly}.

We study the normalized HS number
\[
Y_n := \frac{1}{n}\hs(\mathcal T_n^{\B}).
\]
By \Cref{prop: butterfly support} (the support of butterfly trees), we have $Y_n \in [0,1/2]$, so the sequence $(Y_n)$ is tight and admits subsequential limits $Y_{n_k} \Rightarrow Y$ with $Y \in [0,1/2]$. The conjecture is that this limit is in fact deterministic, yielding
\[
Y_n \xrightarrow[n\to\infty]{\mathbb P} \alpha \in [0,1/2].
\]

To investigate this, we perform Monte Carlo simulations for uniformly random butterfly trees across a range of values of $n$. A summary of the experiments is provided in \Cref{tab: summary statistics} in \Cref{sec: experiment data}, including sample sizes, empirical means, variances (compared with the exact values obtained in the simple butterfly model), and observed ranges. For sufficiently large $n$, the sampled distributions become increasingly concentrated, with only a small number of distinct values (range of $2-3$ value) appearing at all.

These experiments are consistent with convergence to a constant, and suggest the estimate
\[
\alpha \approx 0.4450.
\]
Assuming \Cref{conj: wlln butterfly} with this value of $\alpha$, general uniform butterfly trees would lie between simple butterfly trees, which satisfy a WLLN with limit $1/3$, and Catalan trees, which have limit $1/2$.

\medskip

Focusing on the WLLN for uniformly random butterfly trees, we combine exact computations for $n \le 8$ (see \Cref{sec: exact}) with Monte Carlo experiments up to $n = 25$. The samples are generated using a direct implementation of the HS profile recursion, which evaluates $\hs(\mathcal T^{\B}(\mathbf{x}))$ directly from the $(N-1)$-bit encoding $\mathbf{x}$ without explicitly constructing the tree; see \Cref{alg:hs_profile_corrected,alg:merge_final,alg:build_right_final,alg:build_left_final,alg:utils_final} in \Cref{sec: algorithms}.

Across both regimes, the data are consistent with an approximately linear growth of $\mathbb{E}[\hs(\mathcal T_n^{\B})]$ in $n$. Summary statistics are collected in \Cref{tab: summary statistics} (\Cref{sec: experiment data}), combining exact values for $n \le 8$ with Monte Carlo estimates for larger $n$. A linear regression on the observed means yields
\[
\mathbb{E}[\hs(\mathcal T_n^{\B})]
\approx
-0.4440483947 + 0.4449840084 \cdot n,
\]
with $R^2 = 0.9998609691$; see \Cref{fig:mean_gen}. Equivalently, the normalized quantity $Y_n = \hs(\mathcal T_n^{\B})/n$ appears to stabilize near
\[
\alpha \approx 0.4450.
\]

To support the WLLN convergence in probability, Chebyshev's inequality suggests two sufficient conditions:
\begin{enumerate}[label=(\roman*)]
    \item $\Var(Y_n) = o(1)$,
    \item $\E[Y_n]$ converges.
\end{enumerate}
We focus first on the variance and state a strengthening consistent with the data.

\begin{conjecture}\label{conj: var}
For uniformly random butterfly trees,
\[
\Var(\hs(\mathcal T_n^{\B})) = \Theta(1).
\]
\end{conjecture}

If \Cref{conj: var} holds, then
\[
\Var(Y_n) = \frac{1}{n^2}\Var(\hs(\mathcal T_n^{\B})) = \Theta(n^{-2}) = o(1),
\]
so condition (i) follows. In particular, any subsequential limit of $Y_n$ must be deterministic, as its variance vanishes.

\begin{figure}[t]
    \centering
    \includegraphics[width=0.8\linewidth]{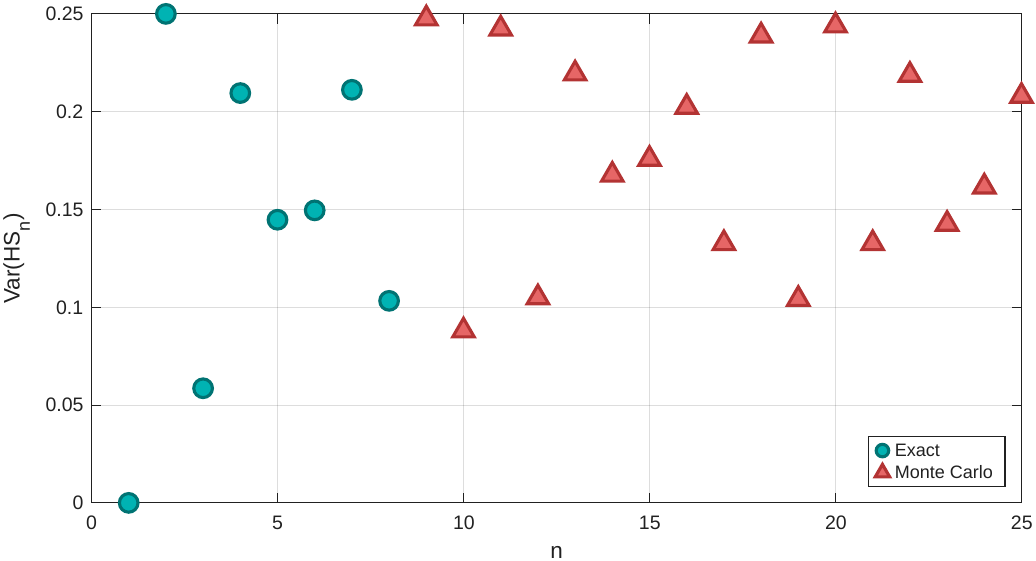}
    \caption{Variance of $\hs(\mathcal T_n^{\B})$. Exact values for $n \le 8$ and Monte Carlo estimates thereafter indicate bounded behavior.}
    \label{fig:variance}
\end{figure}

Exact computations for $n \le 8$ and Monte Carlo simulations up to $n=25$ support \Cref{conj: var}; see \Cref{sec: experiment data}. Across all observed values of $n$, the variance remains bounded and does not exceed approximately $0.25$. Moreover, for each fixed $n$, the sampled HS values concentrate on only two or three distinct values, despite the full support $\{0,1,\ldots,\lfloor n/2\rfloor\}$. This strong concentration is consistent with bounded fluctuations.

The data also suggest a mild oscillatory component. Empirically,
\[
\Var(\hs(\mathcal T_n^{\B})) = D(n) + \mathcal O(1),
\]
where $D$ appears to be a periodic function. As illustrated in \Cref{fig:variance}, the variance exhibits small oscillations (of amplitude about $0.1$ around a mean near $0.15$), with a repeating pattern visible at $n \approx 6,15,24$. This behavior is reminiscent of the oscillatory corrections observed for Catalan trees (cf.\ \Cref{eq:catalan_mean_var}), and suggests that similar barriers to a full fluctuation theory may persist in the general butterfly setting, in contrast to the simple model.

We next consider condition (ii), convergence of $\E[Y_n]$. Let $Y_n'$ be an independent copy of $Y_n$. Using $\max(x,y) = (x+y+|x-y|)/2$ together with the RDE, we obtain
\[
\mathbb{E}[Y_{n+1}]
=
\frac{n}{n+1}\left(
\mathbb{E}[Y_n]
+
\frac{1}{2}\mathbb{E}|Y_n - Y_n'|
\right)
+
\frac{1}{n+1}\beta_n,
\]
where $\beta_n = \mathbb{E}[I_n]$. Rearranging,
\begin{equation}\label{eq: mean recursion}
\Delta_n := \mathbb{E}[Y_{n+1}] - \mathbb{E}[Y_n]
=
\frac{1}{n+1}
\left(
(\beta_n - \mathbb{E}[Y_n])
+
\frac{n}{2}\mathbb{E}|Y_n - Y_n'|
\right).
\end{equation}
Define the dispersion term
\[
c_n := \frac{n}{2}\mathbb{E}[|Y_n - Y_n'|]
= \frac{1}{2}\E\big[|\hs(\mathcal T_n^{\B}) - \hs(\widetilde{\mathcal T}_n^{\B})|\big]
\le \frac{1}{2}\sqrt{2\Var(\hs(\mathcal T_n^{\B}))},
\]
by Cauchy--Schwarz. Under \Cref{conj: var}, we have $c_n = O(1)$, and \eqref{eq: mean recursion} becomes
\begin{equation}\label{eq: EX_n incr}
\Delta_n
=
\frac{1}{n+1}\big(\beta_n + c_n - \mathbb{E}[Y_n]\big).
\end{equation}
Thus the evolution of $\E[Y_n]$ is governed by the cascade increment $\beta_n$ and the dispersion term $c_n$.

The sequence $\E[Y_n]$ is not monotone (e.g.\ $\E[Y_5] > \E[Y_6]$ and $\E[Y_7] > \E[Y_8]$ from the exact computations), so convergence does not follow from bounded monotonicity. A sufficient condition is that $(\Delta_n)$ is summable. We record this as a conjecture.

\begin{conjecture}\label{conj: EX_n cauchy}
There exist constants $C>0$ and $\rho>0$ such that
\[
|\Delta_n| \le \frac{C}{n^{1+\rho}}.
\]
\end{conjecture}

From \eqref{eq: EX_n incr}, such a bound would follow if
\[
\mathbb{E}[Y_n] = \beta_n + c_n + \mathcal{O}(n^{-\rho}),
\]
in which case $\beta_n + c_n \to \alpha$ whenever $\mathbb{E}[Y_n] \to \alpha$. 

Empirically, the sequence $\Delta_n$ exhibits a decaying trend consistent with a power-law envelope in $n$, but the estimates are highly sensitive to fluctuations at moderate values of $n$. (See \Cref{tab: summary statistics}.) In particular, while the magnitude of $\Delta_n$ decreases in a manner broadly compatible with a rate faster than $1/n$, the available data up to $n \le 25$ are insufficient to reliably identify a stable exponent.

These observations provide qualitative support for \Cref{conj: EX_n cauchy}, but do not determine a precise asymptotic decay rate. More extensive simulations would be required to resolve the scaling behavior of $\Delta_n$ with statistical confidence.

\section{Conclusions and future directions}
\label{sec: conclusions}

We characterize the HS number for the butterfly tree family, ranging from the structured simple model to the full recursive construction. For simple butterfly trees, we obtain a complete limit theory, including a functional CLT, by reducing the HS evolution to an additive functional of an explicit 8-state Markov chain. This yields one of the few examples of a nontrivial random tree model in which second-order fluctuations of the HS number are fully Gaussian, in contrast to classical models such as Catalan or Galton--Watson trees where periodic fluctuations  obstruct such limits.

For general butterfly trees, the structure becomes substantially more complex. The HS number cannot be reduced to a finite-state recursion, and we develop instead an $\mathcal{O}(N)$ profile-based algorithm that computes the statistic directly from the $(N-1)$-bit encoding. Our numerical experiments suggest a WLLN of the form $\hs(\mathcal{T}_n^{\B})/n \to \alpha \approx 0.4450$, placing the model between simple butterfly trees (limit $1/3$) and Catalan trees (limit $1/2$).

\medskip

The general butterfly recursion fits naturally into the framework of recursive tree processes (RTPs) and max-type recursive distributional equations (RDEs) introduced by Aldous and Bandyopadhyay~\cite{rde_aldous}. In this interpretation, the HS number is a deterministic functional of iid $\mathrm{Bern}(1/2)$ labels placed on an infinite binary tree, and the level-$n$ statistic depends only on the restriction of this labeled tree to its first $n$ levels.

From this viewpoint, the normalized quantity $\hs(\mathcal{T}_n^{\B})/n$ is a sequence of functionals of the underlying innovation field, and its observed concentration is consistent with an endogenous limit in the sense of RTP theory, meaning that any limiting value is determined by the driving labels alone.

A key structural feature of the model is that the recursion does not close on a fixed low-dimensional state space: the relevant edge profile evolves with $n$ and encodes the increment mechanism at each level. While this can be embedded into a single (countable) state space in principle, the resulting RDE lies outside the standard regimes where existing endogeny criteria apply, which typically rely on contraction or smoothing-transform structure on fixed or effectively finite-dimensional state spaces.

Establishing endogeny in this expanding-state setting, and more generally developing analytic tools for max-type RDEs beyond the contractive and finite-state frameworks, remains an open direction suggested by this work.

\bibliographystyle{alpha} 
\bibliography{references}

\appendix

\section{Algorithms}
\label{sec: algorithms}

This appendix records explicit algorithms for computing the HS number of butterfly trees directly from their bitstring encoding, implementing the profile update rules of \Cref{prop: profile update} without constructing the underlying tree.

Algorithm~\ref{alg:hs_profile_corrected} is the top-level recursion. Given an $(N-1)$-bitstring $\mathbf{x}$, it splits the input into left and right substrings corresponding to the two subtrees, recursively computes their HS profiles, and combines them via the appropriate merging rule, with the final bit determining whether the merge uses the $\oplus$ orientation or swaps the roles of the left and right profiles.

The core computation is Algorithm~\ref{alg:merge_final}, which determines both the updated HS number and the output edge profiles. The threshold $$k^\star  = \min\{k \ge h_c: (a_k,b_k) \ne (1,1)\}$$ identifies the first level at which the parent profile fails to sustain a cascading increment for the gluing edge (top-right if using $\oplus$, top-left if $\ominus$), and governs both whether the HS number increases and how the output profiles are assembled.

Algorithms~\ref{alg:build_right_final} and~\ref{alg:build_left_final} construct the right and left edge profiles of the merged tree, implementing the case analysis of \Cref{prop: profile update} across the regimes $h_c > h_p$, $k^\star = h_p + 1$, and $k^\star = h_p$. Algorithm~\ref{alg:utils_final} defines the profile data structure and its accessors, representing each profile by a multiplicity bit together with binary arrays $(A, B)$, for $A = (a_0,a_1,\ldots,a_{\hs})$ and $B = (b_0,b_1,\ldots,b_{\hs})$, recording presence and escape information at each level.

Together these routines give an $\mathcal{O}(N)$ procedure for evaluating $\hs(\mathcal T^{\B}(\mathbf{x}))$ from the input bitstring. This was implemented in MATLAB and used to verify all exact computations in \Cref{sec: exact} and to perform all Monte Carlo simulations in \Cref{sec: limiting dist}.

For simple butterfly trees the recursion simplifies substantially. Algorithm~\ref{alg:HS number_run_length} gives a non-recursive implementation based on the run-length structure of \Cref{prop: runs}, computing the HS number in a single pass through the input bitstring using only local comparisons.

\begin{algorithm}[t]
\caption{\textsc{HS-Profile}$(\mathbf{x})$}
\label{alg:hs_profile_corrected}
\begin{algorithmic}[1]

\Require Bitstring $\mathbf{x} \in \{0,1\}^{N-1}$ with $N=2^n$
\Ensure $(h_s, L, R)$

\If{$\mathbf{x} = \emptyset$}
    \State \Return $(0, \textsc{makeProfile}(0,1,0), \textsc{makeProfile}(0,1,0))$
\EndIf

\State $m \gets N/2$

\State $\mathbf{y} \gets (x_1,\dots,x_{m-1})$ \Comment{$|\mathbf{y}| = m-1$}
\State $\mathbf{z} \gets (x_m,\dots,x_{N-2})$ \Comment{$|\mathbf{z}| = m-1$}
\State $t \gets x_{N-1}$

\State $(h_p, L_p, R_p) \gets \textsc{HS-Profile}(\mathbf{y})$
\State $(h_c, L_c, R_c) \gets \textsc{HS-Profile}(\mathbf{z})$

\If{$t = 0$}
    \State $(h_s, L, R) \gets \textsc{Merge}_{\oplus}(h_p, L_p, R_p, h_c, L_c, R_c)$
\Else
    \State $(h_s, R, L) \gets \textsc{Merge}_{\oplus}(h_p, R_p, L_p, h_c, R_c, L_c)$
\EndIf

\State \Return $(h_s, L, R)$

\end{algorithmic}
\end{algorithm}

\begin{algorithm}[t]
\caption{\textsc{Merge$_{\oplus}$}}
\label{alg:merge_final}
\begin{algorithmic}[1]

\Require $(h_p, L_p, R_p, h_c, L_c, R_c)$
\Ensure $(h_s, L, R)$

\State $k^\star \gets h_c$

\While{$k^\star \le h_p \;\wedge\; \textsc{getA}(R_p,k^\star)=1 \;\wedge\; \textsc{getB}(R_p,k^\star)=1$}
    \State $k^\star \gets k^\star + 1$
\EndWhile

\If{$h_c > h_p$}
    \State $h_s \gets h_c$
\ElsIf{$k^\star = h_p + 1$}
    \State $h_s \gets h_p + 1$
\Else
    \State $h_s \gets h_p$
\EndIf

\State $R \gets \textsc{BuildRight}(h_p, h_c, R_p, R_c, k^\star, h_s)$
\State $L \gets \textsc{BuildLeft}(h_p, h_c, L_p, L_c, R_p, R, k^\star, h_s)$

\State \Return $(h_s, L, R)$

\end{algorithmic}
\end{algorithm}

\begin{algorithm}[t]
\caption{\textsc{BuildRight}}
\label{alg:build_right_final}
\begin{algorithmic}[1]

\Require $(h_p, h_c, R_p, R_c, k^\star, h_s)$
\Ensure $R$

\State Initialize $\mathbf{A}_R, \mathbf{B}_R \gets \mathbf{0}$

\If{$h_c > h_p$}

    \For{$k = 0$ to $h_c$}
        \State $A_R(k) \gets \textsc{getA}(R_c,k)$
        \State $B_R(k) \gets \textsc{getB}(R_c,k)$
    \EndFor

    \State $m_R \gets 1$

\Else

    \For{$k = 0$ to $h_c$}
        \State $A_R(k) \gets \textsc{getA}(R_c,k)$
        \State $B_R(k) \gets \textsc{getB}(R_c,k)$
    \EndFor

    \For{$k = h_c+1$ to $\min(k^\star, h_s)$}
        \State $A_R(k) \gets 1$
        \State $B_R(k) \gets 0$
    \EndFor

    \For{$k = k^\star+1$ to $\min(h_p, h_s)$}
        \State $A_R(k) \gets \textsc{getA}(R_p,k)$
        \State $B_R(k) \gets \textsc{getB}(R_p,k)$
    \EndFor

    \If{$k^\star = h_p$}
        \State $m_R \gets 1$
    \Else
        \State $m_R \gets R_p.m$
    \EndIf

\EndIf

\State \Return $\textsc{makeProfile}(m_R, \mathbf{A}_R, \mathbf{B}_R)$

\end{algorithmic}
\end{algorithm}

\begin{algorithm}[t]
\caption{\textsc{BuildLeft}}
\label{alg:build_left_final}
\begin{algorithmic}[1]

\Require $(h_p, h_c, L_p, L_c, R_p, R, k^\star, h_s)$
\Ensure $L$

\State Initialize $\mathbf{A}_L, \mathbf{B}_L \gets \mathbf{0}$

\For{$k = 0$ to $\min(h_p-1, h_s)$}
    \State $A_L(k) \gets \textsc{getA}(L_p,k)$
    \State $B_L(k) \gets \textsc{getB}(L_p,k)$
\EndFor

\If{$h_c > h_p$}

    \State $A_L(h_c) \gets 1$, \quad $B_L(h_c) \gets 1$
    \State $m_L \gets 0$

    \If{$L_p.m = 1$}
        \State $A_L(h_p) \gets \textsc{getA}(L_p,h_p)$
        \State $B_L(h_p) \gets \textsc{getB}(L_p,h_p)$
    \EndIf

\ElsIf{$k^\star = h_p + 1$}

    \State $A_L(h_p+1) \gets 1$
    \State $B_L(h_p+1) \gets R.m$
    \State $m_L \gets 0$

    \If{$L_p.m = 1$}
        \State $A_L(h_p) \gets \textsc{getA}(L_p,h_p)$
        \State $B_L(h_p) \gets \textsc{getB}(L_p,h_p)$
    \EndIf

\ElsIf{$k^\star = h_p$}

    \State $A_L(h_p) \gets 1$
    \State $B_L(h_p) \gets 1$
    \State $m_L \gets 0$

\Else

    \State $A_L(h_p) \gets \textsc{getA}(L_p,h_p)$
    \State $B_L(h_p) \gets \textsc{getB}(L_p,h_p)$
    \State $m_L \gets L_p.m$

\EndIf

\State \Return $\textsc{makeProfile}(m_L, \mathbf{A}_L, \mathbf{B}_L)$

\end{algorithmic}
\end{algorithm}

\begin{algorithm}[t]
\caption{\textsc{makeProfile} and Accessors}
\label{alg:utils_final}
\begin{algorithmic}[1]

\Function{makeProfile}{$m,\mathbf{A},\mathbf{B}$}
    \State $P.m \gets m$
    \State $P.A \gets \mathbf{A}$
    \State $P.B \gets \mathbf{B}$
    \State \Return $P$
\EndFunction

\Function{getA}{$P,k$}
    \If{$k+1 \le |P.A|$}
        \State \Return $P.A[k+1]$
        \Comment{$A[k+1] = a_k$}
    \Else
        \State \Return $0$
    \EndIf
\EndFunction

\Function{getB}{$P,k$}
    \If{$k+1 \le |P.B|$}
        \State \Return $P.B[k+1]$
        \Comment{$B[k+1] = b_k$}
    \Else
        \State \Return $0$
    \EndIf
\EndFunction

\end{algorithmic}
\end{algorithm}


\begin{algorithm}[t]
\caption{\textsc{HSRunLength}$(\mathbf{x})$}
\label{alg:HS number_run_length}
\begin{algorithmic}[1]
\Require Bitstring $\mathbf{x} = (x_1,\dots,x_n) \in \{0,1\}^n$
\Ensure $\hs(\mathcal T^{\B}(\mathbf{x}))$

\State $count \gets 0$
\State $run \gets 0$

\For{$j = 1$ to $n-1$}
    \If{$x_{j+1} \neq x_j$}
        \State $run \gets run + 1$
    \ElsIf{$run > 0$}
        \State $count \gets count + \left\lceil \frac{run}{2} \right\rceil$
        \State $run \gets 0$
    \EndIf
\EndFor

\If{$run > 0$}
    \State $count \gets count + \left\lceil \frac{run}{2} \right\rceil$
\EndIf

\State \Return $count$

\end{algorithmic}
\end{algorithm}

\section{Exact computations}
\label{sec: exact}

In this section, we outline exact statistics for uniformly random butterfly trees up to $n = 8$.

    
    
    \paragraph{\fbox{$n\le 1$}} If {$n \le 1$}, then $\mathcal T_0^{\B}$ is a single vertex or a path so $\hs(\mathcal T_n^{\B}) = 0$ deterministically for these random butterfly models. The edge profiles $(L,R)$ for $\mathcal T^{\B}(0) = \mathcal T(12)$ is $L = [1,1]^\top$ and $R = [1,0]^\top$, while $\mathcal T^{\B}(1) = \mathcal T^{\B}(0)^R$ switches the left and right edge profiles. Note then $\beta_0 = 0$.
    
\paragraph{\fbox{$n=2$}} The distribution of $\hs(\mathcal T_n^{\B})$ is determined by the marginal distribution of the full HS number profiles. As is seen in \Cref{fig: bst_4_nodes}, then $\hs(\mathcal T_n^{\B}) \in \{0,1\}$, so $\hs(\mathcal T_2^{\B}) \sim \Bern(\eta_2)$ with 
    $$\eta_2 = \frac{\#\{\mathcal T_n^{\B}: \hs(\mathcal T_n^{\B}) = 1\}}8 = \frac12.$$ 
    In particular, $$\E[\hs(\mathcal T_2^{\B})] = \frac12, \qquad \Var(\hs(\mathcal T_2^{\B})) = \frac14.$$
    Also, now $$\beta_1 = \E[\hs(\mathcal T_2^{\B})] - \E[\max(\hs(\mathcal T_1^{\B}),\hs(\widetilde{\mathcal T}_1^{\B}))] = \frac12.$$
    
\paragraph{\fbox{$n=3$}} There are now $|\B_3| = 2^{2^3-1} = 2^7 = 128$ butterfly trees. These can be realized as there are 8 options for $n=2$ level butterfly trees for each parent and child combination, along with 2 options for each $\oplus$ or $\ominus$ merging. By \Cref{prop: butterfly support}, we have again $|\hs(\mathcal T_n^{\B})| \le \floor{n/2} = 1$, so similarly $\hs(\mathcal T_3^{\B}) \sim \Bern(\eta_3)$. To determine $\eta_3$, we can determine how HS profiles for level 2 trees can be combined to return the HS number of 0. Necessarily both input trees must have HS number 0, and so the merged tree will again have HS number of 0 only if the gluing edge parent profile has $(a_0,b_0) \ne (1,1)$. 

Using \Cref{fig: bst_4_nodes}, if the parent is $\mathcal T^{\B}(000)$, then any HS number 0 child tree option preserves the HS number; this contributes 4 trees. By symmetry we also get 4 more butterfly trees with HS number 0 by using the parent $\mathcal T^{\B}(111) = \mathcal T^{\B}(000)^R$. The other 2 options for parent trees with HS number 0, $\mathcal T^{\B}(010)$ and $\mathcal T^{\B}(101)$, have matching left and right profiles $(a_0,b_0) = (1,1)$, so necessarily joining any subtree will increment the HS number at the root to have HS number 1. So only 8 butterfly trees with 128 nodes have HS number 0, which means 
    $$\eta_3 = \frac{120}{128} = \frac{15}{16}.$$ 
   It follows now 
   $$\E[\hs(\mathcal T_3^{\B})] = \frac{15}{16}, \qquad \Var(\hs(\mathcal T_3^{\B})) = \frac{15}{256},$$
while also noting $\max(\hs(\mathcal T_2^{\B}),\hs(\widetilde{\mathcal T}_2^{\B}))\sim \Bern(\eta)$ for $\eta = 1-1/2^2 = 3/4$, so that
$$
\beta_2 = \frac{15}{16} - \frac34 = \frac{3}{16} = 0.1875
$$

\begin{remark}[Butterfly trees with HS number 0]\label{rmk: HS number=0}
We analyze butterfly trees with $\hs = 0$ through their edge profiles $(L,R)$, which are determined by the \emph{escape profile} $(b_0^{(L)}, b_0^{(R)}) \in \{(1,0),(1,1),(0,1)\}$. 

For $n=1$, only the profiles $(1,0)$ and $(0,1)$ occur. At $n=2$, among the four trees with $\hs=0$, exactly one has profile $(1,0)$, one has $(0,1)$, and the remaining two have $(1,1)$. At $n=3$, the eight trees with $\hs=0$ are obtained by merging level-$2$ trees with $b_0^{(G)}=0$. In this case, a profile $(1,0)$ (resp.\ $(0,1)$) is preserved only when both inputs have that same profile, while all other merges produce $(1,1)$. Thus exactly one tree has profile $(1,0)$, one has $(0,1)$, and the remaining six have $(1,1)$.

This structure persists at all levels. By \Cref{prop: profile update}, the output profile is a deterministic function of the input profiles. For trees with $\hs=0$, only the profiles $(1,0)$, $(0,1)$, and $(1,1)$ can occur, and the only merges that preserve $\hs=0$ are those with $b_0^{(G)}=0$. Among these, the profiles $(1,0)$ and $(0,1)$ arise uniquely from merging identical inputs of the same type (corresponding to the identity and reverse permutations), while all other valid merges produce $(1,1)$. It follows inductively that at level $n$ there is exactly one tree with profile $(1,0)$, one with $(0,1)$, and all remaining trees with $(1,1)$, and hence the total number of trees with $\hs=0$ doubles at each step.

We summarize this as follows:
\begin{proposition}\label{prop: zero}
For butterfly trees $\mathcal T_n^{\B}$ with $N = 2^n$ nodes,
\[
\#\{\mathcal T_n^{\B} : \hs(\mathcal T_n^{\B}) = 0\} = 2^n.
\]
\end{proposition}
In particular, for a uniformly random butterfly tree,
\[
\P\big(\hs(\mathcal T_n^{\B}) = 0\big) = 2^{\,n + 1 - 2^n}.
\]
\end{remark}

Returning to the HS number profiles for $n = 3$, we have 20 distinct HS number edge profiles, as shown in \Cref{tab:HS number_profiles_n3}. This compares to 7 distinct profiles for $n = 2$ and 2 for $n = 1$. We additional note the breakdown for the HS number profiles for HS number 0 matches the profile counts outlined in \Cref{rmk: HS number=0}.

\paragraph{\fbox{$n=4$}} Returning to the exact distribution of uniform butterfly trees $\hs(\mathcal T_n^{\B})$ for {$n = 4$}, we note there are now $|\B_3| = 2^{2^4-1} = 2^{15} = 32{,}768$ such butterfly trees. By \Cref{prop: butterfly support}, then $\hs(\mathcal T_n^{\B}) \in \{0,1,2\}$ since $\floor{n/2} = 2$. By \Cref{prop: zero}, we know $2^4 = 16$ of these have HS number 0, so it suffices to determine how many have HS number 2.

Now to have an HS number of 2, then necessarily the tree was formed from two input $n=3$ level butterfly trees where the parent has HS number 1, where the parent gluing edge profile had $b_1^{(G)} = 1$ and child has HS number 1, or if also $b_0^{(G)} = 1$ then any child with HS number 0 will also lead to a cascading increment in HS number to 2 at the root. In particular, if merging two trees with $\oplus$, then if the parent tree has top-right edge profile $\begin{bmatrix}
    * & 1\\ * & 1
\end{bmatrix}$, then any child tree with HS number 1 will yield an increment at the root of HS number 2. If the parent tree has top-right edge profile $\begin{bmatrix}
    1 & 1\\1& 1
\end{bmatrix}$, then also any child with HS number 0 will also yield a merged tree with HS number 2. Using \Cref{tab:HS number_profiles_n3}, there are 12 trees with right profile $\begin{bmatrix}
    0&1\\0&1
\end{bmatrix}$, 20 trees with right profile $\begin{bmatrix}
    1&1\\0&1
\end{bmatrix}$, and 8 trees with right profile $\begin{bmatrix}
    1&1\\1&1
\end{bmatrix}$. In particular, there are then 40 such trees that can be chosen as the parent tree that can merge using $\oplus$ with any of the 120 trees with HS number 1 to yield a HS number 2 tree, while 8 candidate parent trees can also be combined with any of the 8 HS number 0 trees to yield a HS number 2 merged tree. By symmetry, we have the same number of trees with HS number 2 combined using $\ominus$, that now yields
$$
\#\{\hs = 2\}:=\#\{\mathcal T_4^{\B}: \hs(\mathcal T_4^{\B}) = 2\} = 2 \cdot (40 \cdot 120 + 8 \cdot 8) = 9{,}728.
$$
It follows then $$\#\{\hs = 1\} = 32{,}768 - 9{,}728 - 16 = 23{,}024.$$
Hence, for a uniform butterfly tree $\mathcal T_4^{\B}$, we have 
\[
\E[\hs(\mathcal T_4^{\B})] = \frac{0\cdot 16+1\cdot 23{,}024 + 2\cdot 9{,}728}{32{,}768} = \frac{2{,}655}{2{,}048} \approx 1.2964, \quad \Var(\hs(\mathcal T_4^{\B})) = \frac{878{,}783}{4{,}194{,}304} \approx 0.20952.
\]
Also, now $\max(\hs(\mathcal T_3^{\B}),\hs(\widetilde{\mathcal T}_3^{\B})) \sim \Bern(\eta)$ for $\eta = 1 - 1/16^2 = 1 - 1/256$. Hence,
$$
\beta_3 = \frac{2{,}655}{2{,}048} - 1 + \frac{1}{256} = \frac{615}{2{,}048} \approx 0.30029.
$$

For $n \le 4$, these values are small enough that we can explicitly compute the HS number for every butterfly tree. We initially verified these values computationally by computing the HS numbers for all butterfly trees, which agreed with the exact values derived above. However, for $n=5$, brute-force enumeration becomes infeasible due to the total number of butterfly trees:
\[
|\B_5| = 2^{2^5 - 1} = 2^{31} = 2{,}147{,}483{,}648.
\]

To overcome this, we instead implement a recursive procedure from \Cref{prop: profile update} that computes the HS number directly from the defining bitstring of length $N-1$, where $N = 2^n$ (see \Cref{alg:hs_profile_corrected}). This formulation isolates the effect of the merge operation, reducing the computation to tracking how profiles evolve along the gluing edge and at the root. As a result, we work not with individual trees, but with equivalence classes of trees sharing the same profile.

Let $\mathcal P_n$ denote the set of distinct profiles at level $n$. For each profile, we store a single representative $N-1$ bitstring together with its multiplicity. The recursion then proceeds by combining pairs of profiles, yielding at most $|\mathcal P_n|^2 \cdot 2$ merge operations at each level; we then again preserve a single representative bitstring per unique profile at the next level. This allows us to compute the exact distribution of HS numbers by aggregating counts over profiles rather than enumerating all trees.

The number of distinct profiles grows much more slowly than the total number of trees, and is given in Table~\ref{tab: profile count}.

\begin{table}[h]
    \centering
    \begin{tabular}{c|ccccccccc}
        $n$ & 0 & 1 & 2 & 3 & 4 & 5 & 6 & 7 & 8\\
        \hline
        $|\mathcal P_n|$ & 1 & 2 & 7 & 20 & 54 & 143 & 376 & 986 & 2{,}583
    \end{tabular}
    \caption{Number of distinct profiles $(h,L,R)$ for butterfly trees with $N=2^n$ nodes.}
    \label{tab: profile count}
\end{table}
Using this approach, we computed the exact HS number distributions up to $n=8$. 

\noindent \fbox{$n=5$} We compute the full profile counts, noting $\hs(\mathcal T_5^{\B}) \le \floor{5/2} = 2$:
\begin{align*}
    \#\{\hs = 0\} &= 32\\
    \#\{\hs = 1\} &= 377{,}060{,}320\\
    \#\{\hs = 2\} &= 1{,}770{,}423{,}296
\end{align*}
so that
\[
\E[\hs(\mathcal T_5^{\B})] \approx 1.824418, \qquad \Var(\hs(\mathcal T_5^{\B})) \approx 0.144753, \qquad \beta_4 \approx 0.3188025802.
\]

\noindent \fbox{$n=6,7,8$} For $n \le 5$, standard double-precision arithmetic suffices. However, beginning at $n=6$, the counts exceed the largest exactly representable integer in IEEE double precision ($2^{53} \approx 9.0 \times 10^{15}$), and arithmetic overflow becomes unavoidable. For example, already at $n=6$ we encounter counts on the order of $10^{18}$, such as
\[
\#\{\hs=3\} = 9{,}223{,}372{,}036{,}854{,}775{,}808 \approx 9.22 \cdot 10^{18}.
\]
To obtain exact results at this level, we therefore switched to symbolic integer arithmetic. At $n=7$, the counts grow dramatically still, with
\[
\#\{\hs=3\} =118{,}597{,}725{,}377{,}303{,}585{,}569{,}410{,}371{,}414{,}178{,}922{,}496 \approx 1.19 \cdot 10^{35},
\]
while for $n = 8$, we see \[
\#\{\hs=3\} \approx 5.14 \cdot 10^{46}, \qquad
\#\{\hs=2\} \approx 1.17 \cdot 10^{45}.
\]
From these exact counts, we compute
\[
\E[\hs(\mathcal T_6^{\B})] \approx 2.1662, \quad
\E[\hs(\mathcal T_7^{\B})] \approx 2.6971, \quad
\E[\hs(\mathcal T_8^{\B})] \approx 3.0715,
\]
with corresponding variances remaining bounded:
\[
\Var(\hs(\mathcal T_6^{\B})) \approx 0.1495, \quad
\Var(\hs(\mathcal T_7^{\B})) \approx 0.2112, \quad
\Var(\hs(\mathcal T_8^{\B})) \approx 0.1068.
\]
The corresponding $\beta_n$ values can be computed up to $\beta_7 = \E[\hs(\mathcal T_8^{\B})] - \E[\max(\hs(\mathcal T_7^{\B}),\hs(\widetilde{\mathcal T}_7^{\B}))]$:
\[
\beta_5 \approx 0.1970656394, \qquad \beta_6 \approx 0.3831492326, \qquad \beta_7 \approx 0.1632354406.
\]

Despite the reduction in computation to equivalence class representatives, the computational cost still grows rapidly. The number of merge operations scales as $|\mathcal P_n|^2$, and the runtime increases dramatically:
\[
n=6:\; \text{seconds}, \qquad
n=7:\; \text{minutes}, \qquad
n=8:\; \text{tens of minutes}.
\]
In our implementation, the runtime increased from approximately $14$ seconds at $n=6$, to $215$ seconds at $n=7$, and to over $2{,}200$ seconds (roughly $38$ minutes) at $n=8$. Extrapolating this growth, computing the exact distribution for $n=9$ would require several hours (approximately $4-8$ hours) and substantial memory, making it impractical in our setting. Further extrapolating, $n = 10$ would take approximately $1-4$ days, if at all. Hence, $n=8$ represents a natural stopping point for exact enumeration.

\medskip

For larger values of $n$, we therefore switch to Monte Carlo sampling, using the exact results for $n \le 8$ to validate the implementation and anchor the observed trends. These experiments are described in \Cref{sec: limiting dist} and \Cref{sec: experiment data}. 

\section{Exact enumeration and numerical experiments}
\label{sec: experiment data}

This appendix collects the exact enumerations and Monte Carlo experiments supporting \Cref{sec: limiting dist}.

Table~\ref{tab:HS number_profiles_n3} gives the complete distribution of edge profiles at level $n=3$, classifying each butterfly tree by its HS number and left and right edge profiles in the form $\bigl(m, \begin{bmatrix}A\\B\end{bmatrix}\bigr)$. Even at this small depth the profile distribution is already nontrivial, and the table serves as a concrete reference for the profile recursion developed in \Cref{sec: block}, that continues to grow at each subsequent $n$, with $2{,}583$ distinct profiles at $n = 8$ (see \Cref{tab: profile count}).
\medskip

Table~\ref{tab: summary statistics} reports summary statistics for $\hs(\mathcal T_n^{\B})$ through $n=25$. For $n \le 8$, exact values are used: these are obtained by direct enumeration of profile equivalence classes $\mathcal P_n$ as described in \Cref{sec: exact}. For $n \ge 9$, values are estimated by Monte Carlo sampling using a MATLAB implementation of \Cref{alg:hs_profile_corrected}, drawing bitstrings $\mathbf{x} \in \{0,1\}^{N-1}$ with iid $\Bern(1/2)$ entries and evaluating $\hs(\mathcal T^{\B}(\mathbf{x}))$ via the profile recursion of Algorithms~\ref{alg:hs_profile_corrected}--\ref{alg:utils_final}.

The table highlights several structural features of the model. First, the normalized quantity
\[
\mathbb{E}[Y_n] = \frac1n\mathbb{E}[\hs(\mathcal T_n^{\B})]
\]
remains close to a stable limiting regime, supporting linear growth of the mean HS number. Next, the variance remains uniformly bounded across all observed levels, consistent with strong concentration. This is also supported by the small range of 2-3 distinct HS numbers sampled across each level $n$. Third, the drift term
\[
\Delta_n = \mathbb{E}[Y_{n+1}] - \mathbb{E}[Y_n]
\]
is small and exhibits rapid decay toward zero, consistent with convergence of the normalized mean, although the persistent $O(1)$ fluctuations obviate a clean linear fit.

Simulations are reported through $n=25$, where each sample corresponds to a tree with $2^{25} \approx 3.355 \cdot 10^7$ nodes, which took several hours to gather the 100 data points at this final level. Extending beyond this range would require substantially reduced sample sizes or additional computational resources, and is left for future investigation.
\begin{table}[t]
\centering

\begin{minipage}{0.48\textwidth}
\centering
\begin{tabular}{c|cc|cc|c}
 & \multicolumn{2}{c|}{Left profile} & \multicolumn{2}{c|}{Right profile} &  \\ 
HS number & $m$ & $\begin{bmatrix}A\\B\end{bmatrix}$ 
   & $m$ & $\begin{bmatrix}A\\B\end{bmatrix}$ & \# \\ \hline

0 & 0 & $\begin{bmatrix}1\\1\end{bmatrix}$ & 1 & $\begin{bmatrix}1\\0\end{bmatrix}$ & 1 \\
0 & 1 & $\begin{bmatrix}1\\0\end{bmatrix}$ & 0 & $\begin{bmatrix}1\\1\end{bmatrix}$ & 1 \\
0 & 0 & $\begin{bmatrix}1\\1\end{bmatrix}$ & 1 & $\begin{bmatrix}1\\1\end{bmatrix}$ & 3 \\
0 & 1 & $\begin{bmatrix}1\\1\end{bmatrix}$ & 0 & $\begin{bmatrix}1\\1\end{bmatrix}$ & 3 \\[6pt]

1 & 0 & $\begin{bmatrix}1&1\\1&1\end{bmatrix}$ & 1 & $\begin{bmatrix}1&1\\1&0\end{bmatrix}$ & 2 \\
1 & 1 & $\begin{bmatrix}1&1\\1&0\end{bmatrix}$ & 0 & $\begin{bmatrix}1&1\\1&1\end{bmatrix}$ & 2 \\
1 & 0 & $\begin{bmatrix}0&1\\0&1\end{bmatrix}$ & 1 & $\begin{bmatrix}1&1\\1&0\end{bmatrix}$ & 5 \\
1 & 1 & $\begin{bmatrix}1&1\\1&0\end{bmatrix}$ & 0 & $\begin{bmatrix}0&1\\0&1\end{bmatrix}$ & 5 \\
1 & 0 & $\begin{bmatrix}1&1\\0&0\end{bmatrix}$ & 0 & $\begin{bmatrix}1&1\\0&0\end{bmatrix}$ & 6 \\
1 & 0 & $\begin{bmatrix}1&1\\1&1\end{bmatrix}$ & 1 & $\begin{bmatrix}1&1\\0&0\end{bmatrix}$ & 6 \\

\end{tabular}
\end{minipage}
\hfill
\begin{minipage}{0.48\textwidth}
\centering
\begin{tabular}{c|cc|cc|c}
 & \multicolumn{2}{c|}{Left profile} & \multicolumn{2}{c|}{Right profile} &  \\ 
HS number & $m$ & $\begin{bmatrix}A\\B\end{bmatrix}$ 
   & $m$ & $\begin{bmatrix}A\\B\end{bmatrix}$ & \# \\ \hline

1 & 1 & $\begin{bmatrix}1&1\\0&0\end{bmatrix}$ & 0 & $\begin{bmatrix}1&1\\1&1\end{bmatrix}$ & 6 \\
1 & 0 & $\begin{bmatrix}0&1\\0&1\end{bmatrix}$ & 1 & $\begin{bmatrix}1&1\\0&0\end{bmatrix}$ & 7 \\
1 & 0 & $\begin{bmatrix}1&1\\0&1\end{bmatrix}$ & 1 & $\begin{bmatrix}1&1\\1&0\end{bmatrix}$ & 7 \\
1 & 1 & $\begin{bmatrix}1&1\\0&0\end{bmatrix}$ & 0 & $\begin{bmatrix}0&1\\0&1\end{bmatrix}$ & 7 \\
1 & 1 & $\begin{bmatrix}1&1\\1&0\end{bmatrix}$ & 0 & $\begin{bmatrix}1&1\\0&1\end{bmatrix}$ & 7 \\
1 & 0 & $\begin{bmatrix}1&1\\0&0\end{bmatrix}$ & 0 & $\begin{bmatrix}1&1\\1&0\end{bmatrix}$ & 11 \\
1 & 0 & $\begin{bmatrix}1&1\\1&0\end{bmatrix}$ & 0 & $\begin{bmatrix}1&1\\0&0\end{bmatrix}$ & 11 \\
1 & 0 & $\begin{bmatrix}1&1\\1&0\end{bmatrix}$ & 0 & $\begin{bmatrix}1&1\\1&0\end{bmatrix}$ & 12 \\
1 & 0 & $\begin{bmatrix}1&1\\0&1\end{bmatrix}$ & 1 & $\begin{bmatrix}1&1\\0&0\end{bmatrix}$ & 13 \\
1 & 1 & $\begin{bmatrix}1&1\\0&0\end{bmatrix}$ & 0 & $\begin{bmatrix}1&1\\0&1\end{bmatrix}$ & 13 \\

\end{tabular}
\end{minipage}

\caption{Distribution of HS number profiles for $n=3$. Each profile is recorded as $\left(m, \begin{bmatrix}
    A\\B
\end{bmatrix}\right)$, where $m$ is the multiplicity and $\begin{bmatrix}A\\B\end{bmatrix} = \begin{bmatrix}
    a_0 & a_1 & \cdots & a_{\hs}\\
    b_0 & b_1 & \cdots & b_{\hs}
\end{bmatrix}$ encodes the edge profile ($128$ butterfly trees).}
\label{tab:HS number_profiles_n3}
\end{table}

\begin{table}
\centering
\begin{tabular}{c|c|cc|cc|cc}
n & samples & $\mathbb{E}[\hs_n]$ & Var & min & max & $\mathbb{E}[Y_n]$ & $\Delta_n$ \\
\hline
1 & -- & 0.000000 & 0.000000 & 0 & 0 & 0.000000 & 0.250000 \\
2 & -- & 0.500000 & 0.250000 & 0 & 1 & 0.250000 & 0.062500 \\
3 & -- & 0.937500 & 0.058594 & 0 & 1 & 0.312500 & 0.011597 \\
4 & -- & 1.296387 & 0.209518 & 0 & 2 & 0.324097 & 0.040787 \\
5 & -- & 1.824418 & 0.144753 & 0 & 2 & 0.364884 & -0.003844 \\
6 & -- & 2.166236 & 0.149546 & 0 & 3 & 0.361039 & 0.024254 \\
7 & -- & 2.697052 & 0.211176 & 0 & 3 & 0.385293 & -0.001668 \\
8 & -- & 3.069000 & 0.103291 & 2 & 4 & 0.383625 & 0.010875 \\
9 & 2000 & 3.550500 & 0.247574 & 3 & 4 & 0.394500 & 0.005000 \\
10 & 2000 & 3.995000 & 0.088019 & 3 & 5 & 0.399500 & 0.001091 \\
11 & 2000 & 4.406500 & 0.242379 & 3 & 5 & 0.400591 & 0.008576 \\
12 & 2000 & 4.910000 & 0.104952 & 4 & 6 & 0.409167 & -0.001244 \\
13 & 1000 & 5.303000 & 0.219410 & 4 & 6 & 0.407923 & 0.007577 \\
14 & 1000 & 5.817000 & 0.167679 & 5 & 7 & 0.415500 & -0.003500 \\
15 & 1000 & 6.180000 & 0.175776 & 5 & 7 & 0.412000 & 0.008750 \\
16 & 1000 & 6.732000 & 0.202378 & 6 & 8 & 0.420750 & -0.003221 \\
17 & 500 & 7.098000 & 0.132661 & 6 & 8 & 0.417529 & 0.005137 \\
18 & 500 & 7.608000 & 0.238814 & 7 & 8 & 0.422667 & -0.001193 \\
19 & 500 & 8.008000 & 0.104144 & 7 & 9 & 0.421474 & -0.000724 \\
20 & 200 & 8.415000 & 0.243995 & 8 & 9 & 0.420750 & 0.005202 \\
21 & 200 & 8.945000 & 0.132638 & 8 & 10 & 0.425952 & -0.002316 \\
22 & 200 & 9.320000 & 0.218693 & 9 & 10 & 0.423636 & 0.003755 \\
23 & 100 & 9.830000 & 0.142525 & 9 & 10 & 0.427391 & -0.002391 \\
24 & 100 & 10.200000 & 0.161616 & 10 & 11 & 0.425000 & 0.003400 \\
25 & 100 & 10.710000 & 0.207980 & 10 & 11 & 0.428400 & -- \\
\end{tabular}

\caption{Summary statistics for HS numbers of uniformly random butterfly trees $\mathcal T_n^{\B}$ for $n \le 25$. Exact values are used for $n \le 8$, and Monte Carlo estimates thereafter.}
\label{tab: summary statistics}
\end{table}

\end{document}